\documentclass[11pt]{article}
\usepackage{amssymb,amsmath,epsfig}
\usepackage[colorlinks=false,breaklinks=true,linkcolor=blue]{hyperref}
\usepackage{graphics,psfrag,graphicx,color}
\usepackage{diagbox}
\usepackage{float,hhline}

\numberwithin{equation}{section}
\newcommand{\ds}{\displaystyle}

\newcommand{\bdelta}{{\boldsymbol\delta}}
\newcommand{\bgamma}{{\boldsymbol\gamma}}

\newcommand{\bbeta}{{\boldsymbol\eta}}
\newcommand{\bomega}{{\boldsymbol\omega}}

\newcommand{\bphi}{{\boldsymbol\phi}}

\newcommand{\bpsi}{{\boldsymbol\psi}}

\newcommand{\bv}{{\mathbf{v}}}
\newcommand{\bw}{{\mathbf{w}}}
\newcommand{\f}{\mathbf{f}}

\newcommand{\ba}{\mathbf{a}}

\newcommand{\bc}{\mathbf{c}}
\newcommand{\bi}{\mathbf{i}}

\newcommand{\bu}{\mathbf{u}}

\newcommand{\bz}{{\mathbf{z}}}
\newcommand{\bt}{{\mathbf{t}}}
\newcommand{\bn}{{\mathbf{n}}}
\newcommand{\be}{{\mathbf{e}}}

\newcommand{\0}{{\mathbf{0}}}

\def\bF{\mathbf{F}}

\def\bK{\mathbf{K}}
\def\bI{\mathbf{I}}

\def\bV{\mathbf{V}}
\def\bW{\mathbf{W}}

\newcommand{\bL}{\mathbf{L}}
\newcommand\bH{\mathbf{H}}

\def\ubu{\underline{\bu}}
\def\ubv{\underline{\bv}}
\def\ubw{\underline{\bw}}

\newcommand\bbH{\mathbb{H}}

\newcommand\bbL{\mathbb{L}}

\newcommand{\cA}{\mathcal{A}}
\newcommand{\cB}{\mathcal{B}}

\newcommand{\cE}{\mathcal{E}}

\newcommand{\cJ}{\mathcal{J}}
\newcommand{\cT}{\mathcal{T}}

\newcommand{\cM}{\mathcal{M}}
\newcommand{\cN}{\mathcal{N}}
\newcommand{\cD}{\mathcal{D}}
\newcommand{\cO}{\mathcal{O}}

\def\rp{\mathrm{p}}
\def\rq{\mathrm{q}}
\def\rs{\mathrm{s}}

\def\R{\mathrm{R}}

\def\H{\mathrm{H}}
\def\L{\mathrm{L}}

\def\V{\mathrm{V}}
\def\W{\mathrm{W}}
\def\rc{\mathrm{c}}

\def\rP{\mathrm{P}}
\def\rt{\mathrm{t}}
\def\tD{\mathtt{D}}
\def\DOF{\mathtt{DoF}}
\def\ttd{\mathtt{d}}

\def\tKV{\mathtt{KV}}
\def\tF{\mathtt{F}}

\def\bcurl{\mathbf{curl}}

\def\div{\mathrm{div}}

\def\pil{\left<}
\def\pir{\right>}

\def\iter{\mathtt{iter}}
\def\tol{\textsf{tol}}
\def\DoF{\mathtt{DoF}}
\def\esssup{\mathrm{ess\,sup}}

\def\qin{{\quad\hbox{in}\quad}}
\def\qon{{\quad\hbox{on}\quad}}
\def\qan{{\quad\hbox{and}\quad}}

\def\ov{\overline}

\def\wt{\widetilde}
\def\wh{\widehat}

\newtheorem{thm}{Theorem}[section]
\newtheorem{rem}{Remark}[section]
\newtheorem{lem}[thm]{Lemma}

\newenvironment{proof}{\noindent{\it Proof.}}{\hfill$\square$}

\numberwithin{equation}{section}
\numberwithin{figure}{section}
\numberwithin{table}{section}

\allowdisplaybreaks

\title{Velocity-vorticity-pressure mixed formulation for the Kelvin--Voigt--Brinkman--Forchheimer equations}

\author{{\sc Sergio Caucao}\thanks{Grupo de Investigaci\'on en An\'alisis Num\'erico y C\'alculo Cient\'ifico (GIANuC$^2$) and Departamento de Matem\'atica y F\'isica Aplicadas, 
Universidad Cat\'olica de la Sant\'isima Concepci\'on, Casilla 297, Concepci\'on, Chile, 
email: {\tt scaucao@ucsc.cl}. Supported in part by ANID-Chile through the projects {\sc Centro de Mode\-lamiento Matem\'atico} (FB210005) and Fondecyt 11220393; and by DI-UCSC through the project FGII 04/2023.}
\quad
{\sc Ivan Yotov}\thanks{Department of Mathematics, University of Pittsburgh, Pittsburgh, PA 15260, USA, email: {\tt yotov@math.pitt.edu}. Supported in part by NSF grants DMS 2111129 and DMS 2410686.}}

\date{ }

\begin{document}

\maketitle

\begin{abstract}
\noindent 
In this paper, we propose and analyze a mixed formulation for the Kelvin--Voigt--Brinkman--Forchheimer 
equations for unsteady viscoelastic flows in porous media.
Besides the velocity and pressure, our approach introduces the vorticity as a further unknown.
Consequently, we obtain a three-field mixed variational formulation, where the aforementioned variables are the main unknowns of the system.
We establish the existence and uniqueness of a solution for the weak formulation, and derive
the corresponding stability bounds, employing a fixed-point strategy, along with monotone operators theory and Schauder theorem.
Afterwards, we introduce a semidiscrete continuous-in-time approximation based on stable Stokes elements for the velocity and pressure, and continuous piecewise
polynomial spaces for the vorticity. Additionally, employing backward Euler time discretization, we introduce a fully discrete finite element scheme.
We prove well-posedness, derive stability bounds, and 
establish the corresponding error estimates for both schemes. We provide several numerical
results verifying the theoretical rates of convergence and illustrating the performance and flexibility of the method for a range of domain configurations and model parameters.
\end{abstract}



\maketitle


\section{Introduction}

Fluid flows through porous media at high velocity occur in many industrial applications, such as environmental, chemical, and petroleum engineering. For instance, in groundwater remediation and oil and gas extraction, the flow may be fast near injection or production wells or if the aquifer/reservoir is highly porous. Accurate modeling and simulation of such flows are imperative in these fields to optimize processes, ensure safety, and minimize environmental impact. 
Mathematical models have been developed to address different aspects of these flows. 
The Forchheimer model \cite{forchheimer1901wasserbewegung} addresses nonlinearities inherent in high velocity porous flow regimes.
The Brinkman model \cite{brinkman1949calculation} incorporates both viscous and permeability effects, enabling precise simulations of fluid movement in diverse environments, including highly porous media. On the other hand, many applications of interest involve flows of viscoelastic fluids through porous media, such as polymer injection and foam flooding in enhanced oil and gas recovery, blood perfusion through biological tissues, and industrial filters. 
The Kelvin--Voigt model \cite{KT2009} provides a fundamental framework for describing the viscoelastic behavior of fluids, capturing both viscosity and elasticity. 
The Kelvin--Voigt--Brinkman--Forchheimer (KVBF) model \cite{tt2013}, which generalizes and combines the advantages of the three models, is suitable for fast viscoelastic flows in highly porous media.

Concerning the literature, there are papers devoted to the mathematical analysis of the KVBF equations (see, e.g., \cite{tt2013}, \cite{SQ2018}, \cite{m2020}, and references therein). 
In \cite{tt2013}, the existence of a weak solution to the KVBF problem in velocity-pressure formulation is proved by using the Faedo--Galerkin method. In addition, existence, uniqueness and stability of a stationary solution is studied when the external force is time-independent and small.
Later on, the KVBF model with continuous delay is analyzed in \cite{SQ2018}.
In particular, the authors demonstrate that, following the establishment of pullback-$\mathcal{D}$ absorbing sets for the continuous solution process, the asymptotic compactness obtained through the decomposition method leads to the existence of pullback-$\mathcal{D}$ attractors.
Meanwhile, the existence and uniqueness of a strong solution to the KVBF equations is obtained in \cite{m2020} by exploiting the m-accretive quantization of both the linear and nonlinear operators. Furthermore, the existence of an exponential attractor is established, along with a discussion concerning the inviscid limit of the 3D KVBF equations towards the 3D Navier-Stokes-Voigt system, and subsequently towards the simplified Bardina model.
However, up to the authors' knowledge, there is no literature focused on the numerical analysis of the KVBF problem. 
On the other hand, several papers have been dedicated to the design and analysis of numerical schemes for simulating the Brinkman--Forchheimer equations.
In \cite{lsst2015}, the authors introduce and analyze a perturbed compressible system that serves as an approximation to the Brinkman--Forchheimer equations. They also develop a numerical method for this perturbed system, which relies on a semi-implicit Euler scheme for time discretization and employs the lowest-order Raviart--Thomas elements for spatial discretization. A
pressure stabilization finite element method is developed in \cite{lst2017}. In \cite{KouSunWu},
a time-discrete scheme for a variable porosity Brinkman--Forchheimer model is 
applied for simulating wormhole propagation. In \cite{cy2021}, a mixed formulation based on the pseudostress tensor and the velocity field is presented. By employing classical results on nonlinear monotone operators and a suitable regularization technique in Banach spaces, existence and uniqueness are proved. A fully discrete scheme is developed, which combines a finite element space discretization based on the Raviart--Thomas spaces for the pseudostress tensor and discontinuous piecewise polynomial elements for the velocity with a backward Euler time discretization. Sub-optimal error estimates are derived. These estimates are improved in \cite{covy2022}, where a three-field formulation including the velocity gradient is developed and analyzed. A staggered DG method for a velocity--velocity gradient--pressure formulation of the unsteady Brinkman--Forchheimer problem is developed in \cite{ZhaoLamChung}. Well-posedness and error analysis are presented for the semi-discrete and fully discrete schemes. The method is robust with respect to the Brinkman parameter. 
More recently, a vorticity-based mixed variational formulation is analyzed in \cite{accgry2023}, where the velocity, vorticity, and pressure are the main unknowns of the system. Existence and uniqueness of a weak solution, as well as stability bounds are derived by employing classical results on nonlinear monotone operators. A semidiscrete continuous-in-time mixed finite element approximation and a fully discrete scheme are introduced and optimal rates of convergence are established. 

The purpose of the present work is to develop and analyze a new vorticity-based mixed
formulation of the KVBF problem and to study a suitable 
conforming numerical discretization. To that end, unlike previous KVBF works and motivated by \cite{agmr2015}, \cite{acgmr2021}, and \cite{accgry2023}, we introduce the vorticity as an additional unknown besides the fluid velocity and pressure. 
In addition to the advantage of providing a direct, accurate, and smooth approximation of the vorticity, our approach gives optimal theoretical convergence rates without requiring any small data or quasi-uniformity assumptions on the mesh. Furthermore, unlike \cite{agmr2015}, \cite{acgmr2021}, or \cite{accgry2023}, our method does not require any augmentation process.
It is also important to mention that another novelty and advantage of the present work is that it generalizes the model studied in \cite{accgry2023} by including the nonlinear convective term and an additional time-derivative term, thus considering viscoelastic flows.

We establish the existence of a solution to the continuous weak formulation by employing techniques from \cite{Showalter}, \cite{cgo2021}, and \cite{cgg2023}, combined with a fixed-point argument, the Browder--Minty theorem, and the Schauder theorem. The uniqueness is achieved by contradiction arguments in conjunction with Gr\"onwall's inequality.
Stability for the weak solution is established by means of an energy estimate. We further develop semidiscrete continuous-in-time and fully discrete finite element approximations. 
We emphasize that our formulation relies on the natural $\bH^1$--$\L^2$ spaces for the velocity-pressure pair, facilitating the use of classical stable Stokes elements such as the Taylor--Hood, Crouzeix--Raviart, or MINI elements. Additionally, both continuous and discontinuous piecewise polynomial spaces can be utilized for discretizing the vorticity.
We focus on the continuous polynomial spaces.
We make use of the backward Euler method for the discretization in time. Adapting the tools employed for the analysis of the continuous problem, we prove well-posedness of the discrete schemes and derive the corresponding stability estimates.
We further perform error analysis for the semidiscrete and fully discrete schemes, establishing optimal rates of convergence in space and time.

We have organized the contents of this paper as follows. In Section~\ref{sec:vorticity-formulation} 
we describe the model problem of interest and develop the velocity-vorticity-pressure variational formulation. In Section \ref{sec:well-posedness-model} we show that it is well posed using a fixed-point strategy, along with monotone operators theory and the classical Schauder theorem. 
Next, in Section \ref{sec:semidiscrete-approximation} we present the semidiscrete continuous-in-time approximation, provide particular families of stable finite elements, and obtain error estimates for the proposed methods. Section \ref{sec:fully-discrete-approximation} is devoted to the fully discrete approximation. The performance of the
method is studied in Section \ref{sec:numerical-results} with several numerical 
examples in 2D and 3D, verifying the aforementioned rates of convergence, 
as well as illustrating its flexibility to handle spatially varying parameters in complex geometries. The paper ends with conclusions in Section~\ref{sec:conclusions}.

In the remainder of this section we introduce some standard notation and needed functional spaces.
Let $\Omega\subset \R^d$, $d\in \{2,3\}$, denote a domain with Lipschitz boundary $\Gamma$. 
For $\rs\geq 0$ and $\rp\in[1,+\infty]$, we denote by $\L^\rp(\Omega)$ and $\W^{\rs,\rp}(\Omega)$ the usual Lebesgue and Sobolev spaces endowed with the norms $\|\cdot\|_{\L^\rp(\Omega)}$ and $\|\cdot\|_{\W^{\rs,\rp}(\Omega)}$, respectively.
Note that $\W^{0,\rp}(\Omega)=\L^\rp(\Omega)$. 
If $\rp = 2$, we write $\H^{\rs}(\Omega)$ in place of $\W^{\rs,2}(\Omega)$, and denote the corresponding norm by $\|\cdot\|_{\H^{\rs}(\Omega)}$. 
By $\bH$ and $\bbH$ we will denote the corresponding vectorial and tensorial counterparts of a generic scalar functional space $\H$. 
The $\L^2(\Omega)$ inner product for scalar, vector, or tensor valued functions
is denoted by $(\cdot,\cdot)_{\Omega}$. The $\L^2(\Gamma)$ inner product or duality pairing
is denoted by $\pil\cdot,\cdot\pir_\Gamma$.
Moreover, given a separable Banach space $\V$ endowed with the norm $\|\cdot\|_{\V}$, we let $\L^{\rp}(0,T;\V)$ be the space of classes of functions $f : (0,T)\to \V$ that are Bochner measurable and such that $\|f\|_{\L^{\rp}(0,T;\V)} < \infty$, with
\begin{equation*}
\|f\|^{\rp}_{\L^{\rp}(0,T;\V)} \,:=\, \int^T_0 \|f(t)\|^{\rp}_{\V} \,dt,\quad
\|f\|_{\L^\infty(0,T;\V)} \,:=\, \mathop{\esssup}\limits_{t\in [0,T]} \|f(t)\|_{\V}.
\end{equation*}
In turn, for any vector field $\bv:=(v_i)_{i=1,d}$, we set the gradient and divergence operators, as
\begin{equation*}
\nabla\bv := \left(\frac{\partial\,v_i}{\partial\,x_j}\right)_{i,j=1,d}\qan
\div(\bv) := \sum^{d}_{j=1} \frac{\partial\,v_j}{\partial\,x_j}.
\end{equation*}
In what follows, when no confusion arises, 
$|\cdot|$ denotes the Euclidean norm in $\R^n$ or $\R^{n\times n}$.
In addition, in the sequel we will make use of the well-known H\"older inequality given by
\begin{equation*}
\int_{\Omega} |f\,g| \leq \|f\|_{\L^{\rp}(\Omega)}\,\|g\|_{\L^{\rq}(\Omega)}
\quad \forall\, f\in \L^\rp(\Omega),\,\forall\, g\in \L^\rq(\Omega), 
\quad\mbox{with}\quad \frac{1}{\rp} + \frac{1}{\rq} = 1 \,,
\end{equation*}
and Young's inequality, for $a, b\geq 0$, and $\delta >0$,
\begin{equation}\label{eq:Young-inequality}
a\,b \,\leq\, \frac{\delta^{\rp/2}}{\rp}\,a^\rp + \frac{1}{\rq\,\delta^{\rq/2}}\,b^\rq\,.
\end{equation}
Finally, we recall the continuous injection $i_\rp$ of $\H^1(\Omega)$ into $\L^\rp(\Omega)$ for $\rp\geq 1$ if $d=2$ or $\rp\in [1,6]$ if $d=3$.
More precisely, we have the following inequality
\begin{equation}\label{eq:Sobolev-inequality}
\|w\|_{\L^{\rp}(\Omega)} 
\,\leq\, \|i_\rp\|\,\|w\|_{\H^1(\Omega)}\quad 
\forall\,w \in \H^1(\Omega), 
\end{equation}
with $\|i_\rp\|> 0$ depending only on $|\Omega|$ and $\rp$ (see \cite[Theorem 1.3.4]{Quarteroni-Valli}).

We will denote by $\bi_\rp$ the vectorial version of $i_\rp$.


\section{The model problem and its velocity-vorticity-pressure formulation}\label{sec:vorticity-formulation}

Our model of interest is given by the Kelvin--Voigt--Brinkman--Forchheimer equations (see for instance \cite{tt2013}, \cite{SQ2018}, \cite{m2020}). 
More precisely, given the body force term $\f$ and a suitable initial data
$\bu_0$, the aforementioned system of equations is given by
\begin{equation}\label{eq:KVBF-1}
\begin{array}{c}
\ds \frac{\partial\,\bu}{\partial\,t} - \kappa^2\,\frac{\partial\,\Delta\bu}{\partial\,t} 
- \nu\,\Delta \bu + (\nabla\bu)\bu 
+ \tD\,\bu + \tF\,|\bu|^{\rho-2}\bu + \nabla p =  \f\,,\quad
\div(\bu) = 0 \qin \Omega\times (0,T]\,, \\[2ex]
\bu =\0 \qon \Gamma\times (0,T]\,,\quad
\bu(0) = \bu_0 \qin \Omega\,,\quad
(p,1)_{\Omega} = 0 \qin (0,T]\,,
\end{array}
\end{equation}
where the unknowns are the velocity field $\bu$ and the scalar pressure $p$.
In addition, the constant $\kappa>0$ is a length scale parameter characterizing 
the elasticity of the fluid, 
$\nu>0$ is the Brinkman coefficient (or the effective viscosity), $\tD>0$ is the Darcy coefficient, $\tF>0$ is the Forchheimer coefficient, and $\rho \in [3,4]$ is a given number.

We next introduce a new velocity-vorticity-pressure formulation for \eqref{eq:KVBF-1}.
To that end, we first define the trace operator $\bgamma_{*}$ and vorticity $\bomega$:
\begin{equation*}
\bgamma_{*}(\bv) := \left\{ 
\begin{array}{cl}
\ds \bv\cdot\bt &,\, \mbox{ for } d = 2\,, \\[2ex]
\ds \bv\times\bn &,\, \mbox{ for } d = 3\,,
\end{array}
\right.
\qan
\bomega := \bcurl(\bu) = \left\{ 
\begin{array}{cl}
\ds \frac{\partial\,u_2}{\partial\,x_1} - \frac{\partial\,u_1}{\partial\,x_2} &,\, \mbox{ for } d = 2\,, \\[2ex]
\ds \nabla\times\bu &,\, \mbox{ for } d = 3\,.
\end{array}
\right.
\end{equation*}
Note that the $\bcurl$ of a two-dimensional vector field is a scalar, 
whereas the $\bcurl$ of a three-dimensional vector field is a vector.
In order to avoid a multiplicity of notation, we nevertheless denote it like a vector, 
provided there is no confusion.
In addition, in $2$-D the $\bcurl$ of a scalar field $q$ is a vector given by 
$\bcurl(q)=\left( \frac{\partial\,q}{\partial\,x_2} , - \frac{\partial\,q}{\partial\,x_1} \right)^\rt$.
Then, employing the well-known identity \cite[Section~I.2.3]{Girault-Raviart}
\begin{equation}\label{eq:curl-curl-identity}
\bcurl(\bcurl(\bv)) = -\Delta\bv + \nabla(\div(\bv))
\end{equation}
in combination with the incompressibility condition $\div(\bu) = 0$ in $\Omega\times (0,T]$, we find that \eqref{eq:KVBF-1} can be rewritten, equivalently, as follows: Find $(\bu,\bomega,p)$ in suitable spaces to be indicated below such that
\begin{equation}\label{eq:KVBF-2}
\begin{array}{c}
\ds \frac{\partial\,\bu}{\partial\,t} - \kappa^2\,\frac{\partial\,\Delta\bu}{\partial\,t} + \tD\,\bu + \tF\,|\bu|^{\rho-2}\bu + (\nabla\bu)\bu + \nu\,\bcurl(\bomega) + \nabla p \,=\, \f \qin \Omega\times (0,T] \,, \\ [2ex]
\ds \bomega \,=\, \bcurl(\bu),\quad
\div(\bu) \,=\, 0 \qin \Omega\times (0,T] \,, \\ [2ex]
\ds \bu \,=\, \0 \qon \Gamma\times (0,T],\quad
\bu(0) \,=\, \bu_0 \qin \Omega,\quad
(p,1)_\Omega \,=\, 0 \qin (0,T]\,.
\end{array}
\end{equation}

Next, multiplying the first equation of \eqref{eq:KVBF-2} by a suitable test function $\bv$, we obtain
\begin{align}
(\partial_t\,\bu,\bv)_{\Omega} 
- \kappa^2(\partial_t\,\Delta\bu,\bv)_{\Omega}
+ \tD\,(\bu,\bv)_{\Omega} 
+ \tF\,(|\bu|^{\rho-2}\bu,\bv)_{\Omega} & \nonumber \\[1ex]
+\, ((\nabla\bu)\bu,\bv)_{\Omega}
+ \nu\,(\bcurl(\bomega), \bv)_{\Omega} 
+ (\nabla p,\bv)_{\Omega} & \,=\, (\f,\bv)_{\Omega}\,, \label{eq:KVBF-formulation-0}
\end{align}
where we use the notation $\partial_t := \dfrac{\partial}{\partial\,t}$.
Notice that the fourth and fifth terms in the left-hand side of \eqref{eq:KVBF-formulation-0} 
require $\bu$ to live in a smaller space than $\bL^2(\Omega)$.
In particular, by applying Cauchy--Schwarz and H\"older's inequalities and then the continuous injection 
$\bi_\rho$ (resp. $\bi_4$) of $\bH^1(\Omega)$ into $\bL^\rho(\Omega)$ (resp. $\bL^4(\Omega)$), 
with $\rho\in [3,4]$, we find that
\begin{equation}\label{eq:Forchheimer-bound}
\left| (|\bu|^{\rho-2}\bu,\bv)_{\Omega} \right| 
\,\leq\, \|\bu\|^{\rho-1}_{\bL^{\rho}(\Omega)}\,\|\bv\|_{\bL^{\rho}(\Omega)}
\,\leq\, \|\bi_\rho\|^\rho\,\|\bu\|^{\rho-1}_{\bH^1(\Omega)}\,\|\bv\|_{\bH^1(\Omega)}
\end{equation}
and 
\begin{equation}\label{eq:Convective-bound}
\left| ((\nabla\bu)\bz,\bv)_\Omega \right|
\,\leq\, \|\nabla\bu\|_{\bbL^2(\Omega)}\,\|\bz\|_{\bL^4(\Omega)}\,\|\bv\|_{\bL^4(\Omega)}
\,\leq\, \|\bi_4\|^2\,\|\bu\|_{\bH^1(\Omega)}\,\|\bz\|_{\bH^1(\Omega)}\,\|\bv\|_{\bH^1(\Omega)} \,,
\end{equation}
for all $\bu, \bv, \bz\in \bH^1(\Omega)$,
which, together with the Dirichlet boundary condition $\bu = \0$ on $\Gamma$ (cf. \eqref{eq:KVBF-2}) suggest to look for the unknown $\bu$ in $\bH^1_0(\Omega)$ and to restrict the set of corresponding test functions $\bv$ to the same space.
In addition, employing Green's formula \cite[Theorem~I.2.11]{Girault-Raviart}, the sixth term in 
the left-hand side in \eqref{eq:KVBF-formulation-0}, can be rewritten as
\begin{equation}\label{eq:Green-formula}
(\bcurl(\bomega),\bv)_{\Omega} \,=\, (\bomega,\bcurl(\bv))_{\Omega} 
- \pil \bgamma_{*}(\bv),\bomega\pir_{\Gamma} 
\,=\, (\bomega,\bcurl(\bv))_{\Omega} \quad \forall\,\bv\in \bH^1_0(\Omega)\,.
\end{equation}
Thus, replacing back \eqref{eq:Green-formula} into \eqref{eq:KVBF-formulation-0}, integrating by parts the terms $(\partial_t\,\Delta\bu,\bv)_{\Omega}$ and $(\nabla p,\bv)_{\Omega}$, and incorporating the second and third equations of \eqref{eq:KVBF-2} in a weak sense, we obtain the system
\begin{align}
(\partial_t\,\bu,\bv)_{\Omega} + \kappa^2(\partial_t\,\nabla\bu,\nabla\bv)_{\Omega}
+ \tD\,(\bu,\bv)_{\Omega} 
+ \tF\,(|\bu|^{\rho-2}\bu,\bv)_{\Omega} & \nonumber \\[1ex]
+\, ((\nabla\bu)\bu,\bv)_{\Omega}
+ \nu\,(\bomega,\bcurl(\bv))_{\Omega} 
- (p,\div(\bv))_{\Omega} &\,=\, (\f,\bv)_{\Omega}\,, \label{eq:KVBF-formulation-1} \\[1ex]
\ds \nu\,(\bomega,\bpsi)_{\Omega} - \nu\,(\bpsi,\bcurl(\bu))_{\Omega} &\,=\, 0\,, \label{eq:KVBF-formulation-2} \\[1ex]
\ds (q,\div(\bu))_{\Omega} &\,=\, 0 \,, \label{eq:KVBF-formulation-3}
\end{align}
for all $(\bv, \bpsi, q)\in \bH^1_0(\Omega)\times \bL^2(\Omega)\times \L^2_0(\Omega)$, where $\L^2_0(\Omega) := \Big\{ q\in \L^2(\Omega) :\quad (q,1)_{\Omega} = 0 \Big\}$.

Next, in order to write the above formulation in a more suitable way for the analysis to be developed below, we set
\begin{equation*}
\ubu := (\bu, \bomega)\in \bH^1_0(\Omega)\times \bL^2(\Omega)\,,
\end{equation*}
with corresponding norm given by
\begin{equation*}
\|\ubv\| = \|(\bv,\bpsi)\| := \left( \|\bv\|^2_{\bH^1(\Omega)} + \|\bpsi\|^2_{\bL^2(\Omega)} \right)^{1/2} \quad \forall\,\ubv := (\bv,\bpsi)\in \bH^1_0(\Omega)\times \bL^2(\Omega)\,.
\end{equation*}
Hence, the weak form associated with the Kelvin--Voigt--Brinkman--Forchheimer equations \eqref{eq:KVBF-formulation-1}--\eqref{eq:KVBF-formulation-3} reads:
Given $\f:[0,T]\to \bH^{-1}(\Omega)$ and $\bu_0\in \bH^1_0(\Omega)$, 
find $(\ubu,p):[0,T] \to \big(\bH^1_0(\Omega)\times \bL^2(\Omega)\big)\times \L^2_0(\Omega)$ such that $\bu(0) = \bu_0$ and, for a.e. $t\in (0,T)$,
\begin{equation}\label{eq:KVBF-variational-formulation}
\begin{array}{llll}
\dfrac{\partial}{\partial\,t}\,[\cE(\ubu(t)),\ubv] + [\cA(\bu(t))(\ubu(t)),\ubv] + [\cB'(p(t)),\ubv] & = & [\bF(t),\ubv] & \forall\,\ubv\in \bH^1_0(\Omega)\times \bL^2(\Omega)\,, \\[2ex]
-\,[\cB(\ubu(t)),q] & = & 0 & \forall\,q\in \L^2_0(\Omega)\,,
\end{array}
\end{equation}
where, given $\bz\in \bH^1_0(\Omega)$, the operators $\cE, \cA(\bz) : \big(\bH^1_0(\Omega)\times \bL^2(\Omega)\big)\to \big(\bH^1_0(\Omega)\times \bL^2(\Omega)\big)'$, and $\cB: \big(\bH^1_0(\Omega)\times \bL^2(\Omega)\big)\to \L^2_0(\Omega)'$ are defined, respectively, as
\begin{align}
[\cE(\ubu),\ubv] \,:=\,& (\bu,\bv)_\Omega + \kappa^2\,(\nabla\bu,\nabla\bv)_{\Omega} \,, \label{eq:operator-E} \\[2ex]
[\cA(\bz)(\ubu),\ubv] \,:=\,& [\ba(\ubu),\ubv] + [\bc(\bz)(\ubu),\ubv] \,, \label{eq:operator-Au} \\[2ex]
[\ba(\ubu),\ubv] \,:=\,& \tD\,(\bu,\bv)_{\Omega} 
+ \tF\,(|\bu|^{\rho-2}\bu,\bv)_{\Omega}
+ \nu\,(\bomega,\bpsi)_{\Omega} \nonumber \\[2ex]
& +\, \nu\,(\bomega,\bcurl(\bv))_{\Omega} 
- \nu\,(\bpsi,\bcurl(\bu))_{\Omega} \,, \label{eq:operator-a} \\[2ex]
[\bc(\bz)(\ubu),\ubv] \,:=\,& ((\nabla\bu)\bz,\bv)_{\Omega} \,, \label{eq:operator-cz} \\[2ex]
[\cB(\ubv),q] \,:=\,& - (q,\div(\bv))_{\Omega} \,, \label{eq:operator-B}
\end{align}
and $\bF\in (\bH^1_0(\Omega)\times \bL^2(\Omega))'$ is the bounded linear functional given by
\begin{equation}\label{eq:functional-F}
[\bF,\ubv] \,:=\, (\f,\bv)_{\Omega}\,.
\end{equation}
In all the terms above, $[\cdot,\cdot]$ denotes the duality pairing induced by the corresponding operators.
In addition, we let $\cB':\L^2_0(\Omega)\to \big(\bH^1_0(\Omega)\times \bL^2(\Omega)\big)'$ be the adjoint of $\cB$, which satisfies $[\cB'(q),\ubv] = [\cB(\ubv),q]$ for all $\ubv = (\bv,\bpsi)\in \bH^1_0(\Omega)\times \bL^2(\Omega)$ and $q\in \L^2_0(\Omega)$.

Now we define the kernel space of the operator $\cB$,
\begin{equation*}
\bV :=\, \Big\{ \ubv=(\bv,\bpsi)\in \bH^1_0(\Omega)\times\bL^2(\Omega) :\quad [\cB(\ubv),q] = 0 \quad \forall\,q\in \L^2_0(\Omega) \Big\} \,,
\end{equation*}
which from the definition of the operator $\cB$ (cf. \eqref{eq:operator-B}) can be rewritten as 
\begin{equation}\label{eq:def-K}
\bV = \bK\times \bL^2(\Omega)\,,\quad \mbox{where}\quad \bK \,:=\, \Big\{ \bv\in \bH^1_0(\Omega) :\quad \div(\bv) = 0 \qin \Omega \Big\}.
\end{equation}
This leads us to the reduced problem: Given $\f:[0,T]\to \bH^{-1}(\Omega)$ and $\bu_0\in \bK$, 
find $\ubu:[0,T] \to \bK\times \bL^2(\Omega)$ such that $\bu(0) = \bu_0$ and, for a.e. $t\in (0,T)$,
\begin{equation}\label{eq:KVBF-reduced-VF}
\dfrac{\partial}{\partial\,t}\,[\cE(\ubu(t)),\ubv] + [\cA(\bu(t))(\ubu(t)),\ubv]
\,=\, [\bF(t),\ubv] \quad \forall\,\ubv\in \bK\times \bL^2(\Omega) \,.
\end{equation}
According to the definition of $\bK$ (cf. \eqref{eq:def-K}), owing to the inf-sup condition of $\cB$ (cf. \cite[Corollary B.71]{Ern-Guermond}):
\begin{equation}\label{eq:B-inf-sup-condition}
\sup_{\0\neq \ubv\in \bH^1_0(\Omega)\times \bL^2(\Omega)} \frac{[\cB(\ubv),q]}{\|\ubv\|}
\,\geq\, \sup_{\0\neq \bv\in \bH^1_0(\Omega)} \frac{\ds\int_{\Omega} q\,\div(\bv) }{\|\bv\|_{\bH^1(\Omega)}}
\,\geq\, \beta\,\|q\|_{\L^2(\Omega)} \quad \forall\,q\in \L^2_0(\Omega) \,,
\end{equation}
with $\beta>0$, and using standard arguments, it is not difficult to show that the problem \eqref{eq:KVBF-reduced-VF} 
is equivalent to \eqref{eq:KVBF-variational-formulation}.
This result is stated next and the proof is omitted.
\begin{lem}\label{lem:equivalent-problems}
If $(\ubu,p):[0,T]\to \big( \bH^1_0(\Omega)\times \bL^2(\Omega) \big)\times \L^2_0(\Omega)$ 
is a solution of \eqref{eq:KVBF-variational-formulation}, then $\bu:[0,T]\to \bK$ 
and $\ubu=(\bu,\bomega)$ is a solution of \eqref{eq:KVBF-reduced-VF}.
Conversely, if $\ubu:[0,T]\to \bK\times \bL^2(\Omega)$ is a solution of \eqref{eq:KVBF-reduced-VF}, 
then there exists a unique $p:[0,T]\to \L^2_0(\Omega)$ such that $(\ubu,p)$ 
is a solution of \eqref{eq:KVBF-variational-formulation}.
\end{lem}


\section{Well-posedness of the model}\label{sec:well-posedness-model}

In this section we establish the solvability of \eqref{eq:KVBF-reduced-VF} (equivalently of  \eqref{eq:KVBF-variational-formulation}). 
To that end we first collect some previous results that will be used in the forthcoming analysis.

\subsection{Preliminary results}

We begin by recalling the key result \cite[Theorem~IV.6.1(b)]{Showalter}, which will be used to establish the existence of a solution to \eqref{eq:KVBF-reduced-VF}.
In what follows, an operator $A$ from a real vector space $E$ to its algebraic dual $E'$ is symmetric and monotone if, respectively,
\begin{equation*}
[A(x),y] = [A(y),x] \qan
[A(x) - A(y),x - y] \geq 0 \quad \forall\,x, y\in E\,.
\end{equation*}
In addition,  $Rg(A)$ denotes the range of $A$.
\begin{thm}\label{thm:well-posed-parabolic-problem}
Let the linear, symmetric and monotone operator $\cN$ be given for the real vector space $E$ to its algebraic dual $E'$, and let $E'_b$ be the Hilbert space which is the dual of $E$ with the seminorm
\begin{equation*}
|x|_b = [\cN(x),x]^{1/2} \quad x\in E.
\end{equation*}
Let $\cM\subset E\times E'_b$ be a relation with domain $\cD = \Big\{ x\in E \,:\, \cM(x) \neq \emptyset \Big\}$.
	
Assume $\cM$ is monotone and $Rg(\cN + \cM) = E'_b$.
Then, for each $f\in \W^{1,1}(0,T;E'_b)$ and for each $u_0\in \cD$, there is a solution $u$ of
\begin{equation}\label{eq:parabolic-problem}
\frac{\partial}{\partial\,t}\big(\cN(u(t))\big) + \cM\big(u(t)\big) \ni f(t) \quad\mbox{a.e.}\quad 0 < t < T,
\end{equation}
with
\begin{equation*}
\cN(u)\in \W^{1,\infty}(0,T;E'_b),\quad u(t)\in \cD,\quad \mbox{ for all }\, 0\leq t\leq T,\qan \cN(u(0)) = \cN(u_0).
\end{equation*}
\end{thm}

In addition, in order to provide the range condition in Theorem \ref{thm:well-posed-parabolic-problem} 
we will require the Browder--Minty theorem \cite[Theorem 9.14-1]{Ciarlet}.
\begin{thm}\label{thm:Browder-Minty-theorem}
Let $\V$ be a real separable reflexive Banach space and let $\cA : \V \to \V'$ be a  coercive and hemicontinuous monotone operator. Then $\cA$ is surjective, i.e., given any $f\in\V'$ there exists $u$ such that
\begin{equation*}
u\in\V \qan \cA(u) = f \,.
\end{equation*}
If $\cA$ is strictly monotone, then $\cA$ is also injective.
\end{thm}

Next, we establish the stability properties of the operators involved in \eqref{eq:KVBF-variational-formulation}.
We begin by observing that the operators $\cE, \cB$ and the functional $\bF$ are linear.
In turn, from \eqref{eq:operator-E}, \eqref{eq:operator-B} and \eqref{eq:functional-F}, 
and employing H\"older and Cauchy--Schwarz inequalities, there hold
\begin{align}
\big| [\cB(\ubv),q] \big| &\,\leq \, \sqrt{d}\,\|\ubv\|\,\|q\|_{\L^2(\Omega)} 
\quad \forall\,(\ubv, q)\in \big(\bH^1_0(\Omega)\times \bL^2(\Omega)\big)\times \L^2_0(\Omega) \,, \label{eq:continuity-B} \\[1ex]
\big| [\bF,\ubv] \big| &\,\leq\, \|\f\|_{\bH^{-1}(\Omega)}\,\|\bv\|_{\bH^1(\Omega)} 
\,\leq\, \|\f\|_{\bH^{-1}(\Omega)} \, \|\ubv\| 
\quad \forall\,\ubv\in \bH^1_0(\Omega)\times \bL^2(\Omega) \,, \label{eq:continuity-F} 
\end{align}
and
\begin{equation}\label{eq:continuity-monotonicity-E}
\big| [\cE(\ubu),\ubv] \big| \,\leq \, \max\{1,\kappa^2\}\,\|\ubu\|\,\|\ubv\|,\quad
[\cE(\ubv),\ubv] \,\geq\, \min\{1,\kappa^2\}\,\|\bv\|^2_{\bH^1(\Omega)} 
\quad \forall\,\ubu, \ubv\in \bH^1_0(\Omega)\times \bL^2(\Omega) \,, 
\end{equation}
which implies that $\cB$ and $\bF$ are bounded and continuous, and $\cE$ is 
bounded, continuous, and monotone.

On the other hand, given $\bz\in \bH^1_0(\Omega)$, employing the Cauchy--Schwarz and H\"older inequalities, \eqref{eq:Forchheimer-bound} and \eqref{eq:Convective-bound}, 
it is readily seen that the nonlinear operator 
$\cA(\bz)$ (cf. \eqref{eq:operator-Au}) is bounded, that is
\begin{equation}\label{eq:boundeness-A}
\big| [\cA(\bz)(\ubu),\ubv] \big| \,\leq\, 
C_{\cA}\,\Big\{ \big(1 + \|\bz\|_{\bH^1(\Omega)}\big)\,\|\bu\|_{\bH^1(\Omega)} 
+ \|\bu\|^{\rho - 1}_{\bH^1(\Omega)} + \|\bomega\|_{\bL^2(\Omega)} \Big\}\,\|\ubv\|\,,
\end{equation}
with $C_{\cA}>0$ depending on $\tD, \tF, \nu, \|\bi_4\|$, and $\|\bi_\rho\|$.
In addition, using similar arguments to \eqref{eq:Convective-bound}, it is not difficult to see that the operator $\bc(\bz)$ (cf. \eqref{eq:operator-cz}) satisfies
\begin{equation}\label{eq:continuity-cz-u}
\begin{array}{l}
\ds \big| [\bc(\bz)(\ubu_1 - \ubu_2),\ubv] \big| 
\,\leq\, \|\bz\|_{\bL^4(\Omega)}\,\|\bu_1 - \bu_2\|_{\bH^1(\Omega)}\,\|\bv\|_{\bL^4(\Omega)} \\ [2ex]
\ds\quad \,\leq\, \|\bi_4\|^2\,\|\bz\|_{\bH^1(\Omega)}\,\|\ubu_1 - \ubu_2\|\,\|\ubv\| \quad 
\forall\, \bz \in \bH^1_0(\Omega),\,\forall\, \ubu_1, \ubu_2, \ubv\in \bH^1_0(\Omega)\times \bL^2(\Omega),
\end{array}
\end{equation}
and
\begin{equation}\label{eq:continuity-cz-z}
\begin{array}{l}
\ds \big| [\bc(\bz_1 - \bz_2)(\ubu),\ubv] \big| 
\,\leq\, \|\bz_1 - \bz_2\|_{\bL^4(\Omega)}\,\|\bu\|_{\bH^1(\Omega)}\,\|\bv\|_{\bL^4(\Omega)} \\ [2ex]
\ds\quad \,\leq\, \|\bi_4\|^2\,\|\bz_1 - \bz_2\|_{\bH^1(\Omega)}\,\|\ubu\|\,\|\ubv\| \quad \forall\, \bz_1, \bz_2\in \bH^1_0(\Omega),\,\forall\, \ubu, \ubv\in \bH^1_0(\Omega)\times \bL^2(\Omega) \,.
\end{array}
\end{equation}
In turn, observe that for any $\bz\in \bK$ (cf. \eqref{eq:def-K}), there holds
\begin{equation}\label{eq:null-identity-cz}
[\bc(\bz)(\ubv),\ubv] = 0 \quad \forall\,\ubv\in \bH^1_0(\Omega)\times \bL^2(\Omega).
\end{equation}

Finally, given $\bu\in \bK$ (cf. \eqref{eq:def-K}) and recalling the definition of the operators 
$\cE$ and $\cA(\bu)$ (cf. \eqref{eq:operator-E}, \eqref{eq:operator-Au}), 
we note that problem \eqref{eq:KVBF-reduced-VF} 
can be written in the form of \eqref{eq:parabolic-problem} with
\begin{equation}\label{eq:setting-E-u-N-M}
E \,:=\, \bK\times \bL^2(\Omega)\,,\quad 
u \,:=\, \ubu = (\bu, \bomega) \,,\quad
\cN\,:=\, \cE\,,\quad
\cM\,:=\, \cA(\bu) \,.
\end{equation}
Let $E'_b$ be the Hilbert space that is the dual of $\bK\times \bL^2(\Omega)$ 
with the seminorm induced by the operator $\cE$ (cf. \eqref{eq:operator-E}), which  
thanks to the fact that $\kappa>0$, is given by
\begin{equation*}
\|\ubv\|_{\cE} := \Big\{ (\bv,\bv)_{\Omega} + \kappa^2\,(\nabla\bv,\nabla\bv)_{\Omega} \Big\}^{1/2} \equiv \|\bv\|_{\bH^1(\Omega)} \quad \forall\,\ubv=(\bv,\bpsi)\in \bH^1_0(\Omega)\times \bL^2(\Omega) \,.
\end{equation*}
Then we define the spaces
\begin{equation}\label{eq:setting-Eb-D}
E'_b \,:=\, \bH^{-1}(\Omega)\times \{\0\} \,,\quad
\cD \,:=\, \Big\{ \ubu\in \bK\times \bL^2(\Omega) :\quad \cM(\ubu)\in E'_b \Big\} \,.
\end{equation}
In the next section we prove the hypotheses of Theorem~\ref{thm:well-posed-parabolic-problem} to establish the well-posedness of \eqref{eq:KVBF-reduced-VF}.

\subsection{Range condition}

We begin with the verification of the range condition in Theorem~\ref{thm:well-posed-parabolic-problem}.
Let us consider the resolvent system associated with \eqref{eq:KVBF-reduced-VF}:
Find $\ubu = (\bu,\bomega)\in \bK\times \bL^2(\Omega)$ such that
\begin{equation}\label{eq:reduced-problem}
[(\cE + \cA(\bu))(\ubu),\ubv] = [\wh{\bF},\ubv] \quad \forall\,\ubv\in \bK\times \bL^2(\Omega) \,,
\end{equation}
where $\wh{\bF}\in \bH^{-1}(\Omega)\times \{\0\}$ is 
a functional given by $\wh{\bF}(\ubv) := (\wh{\f},\bv)_\Omega$ for some $\wh{\f}\in \bH^{-1}(\Omega)$.
In the following two sections we prove that \eqref{eq:reduced-problem} 
has a solution by employing a suitable fixed-point approach.

\subsubsection{A fixed-point strategy}

Let us define the operator $\cJ : \bK \to \bK$ by
\begin{equation}\label{eq:operator-J}
\cJ(\bz) := \bu \quad \forall\, \bz\in \bK\,,
\end{equation}
where $\bu$ is the first component of the solution of the partially linearized version of problem \eqref{eq:reduced-problem}:
Find $\ubu=(\bu,\bomega)\in \bK\times \bL^2(\Omega)$ such that
\begin{equation}\label{eq:auxiliary-reduced-problem}
[(\cE + \cA(\bz))(\ubu),\ubv] \,=\, [\wh{\bF},\ubv] \quad \forall\,\ubv\in \bK\times \bL^2(\Omega) \,.
\end{equation}

It is clear that $\ubu=(\bu,\bomega)\in \bK\times \bL^2(\Omega)$ is a solution of problem \eqref{eq:reduced-problem} if and only if $\cJ(\bu) = \bu$.
In this way, to establish 
existence of solution of \eqref{eq:reduced-problem}
it suffices to prove that $\cJ$ has a fixed-point in $\bK$.

Before proceeding with the solvability analysis of \eqref{eq:reduced-problem}, we first establish
the well-definiteness of the fixed-point operator $\cJ$.
To that end, in what follows we prove the hypothesis of the Minty--Browder theorem (cf. Theorem \ref{thm:Browder-Minty-theorem}) applied to the problem \eqref{eq:auxiliary-reduced-problem}.
We begin by observing that, thanks to the reflexivity and separability of $\L^{\rp}(\Omega)$
for $\rp\in (1,+\infty)$, it follows that $\bH^1_0(\Omega), \bL^2(\Omega)$, and $\L^2_0(\Omega)$
are reflexive and separable as well.

We continue by establishing a continuity bound of the nonlinear operator $\cE + \cA(\bz)$.
\begin{lem}\label{lem:continuity-Az}
Let $\bz\in \bK$.
Then, there exists $L_{\tKV} > 0$, depending on $\tD, \tF, \nu, \kappa, \rho, \|\bi_\rho\|, \|\bi_4\|$ and $|\Omega|$, such that
\begin{equation}\label{eq:continuity-of-Az}
\begin{array}{l}
\ds \|(\cE + \cA(\bz))(\ubu) - (\cE + \cA(\bz))(\ubv)\| \\ [2ex]
\ds\quad\,\leq\, L_{\tKV}\,\Big\{ \Big( 1 + \|\bz\|_{\bH^1(\Omega)} + \|\bu\|^{\rho-2}_{\bH^1(\Omega)} + \|\bv\|^{\rho-2}_{\bH^1(\Omega)} \Big)\,\|\bu - \bv\|_{\bH^1(\Omega)} 
+ \|\bomega - \bpsi\|_{\bL^2(\Omega)} \Big\} \,,
\end{array}
\end{equation}
for all $\ubu=(\bu,\bomega), \ubv=(\bv,\bpsi)\in \bK\times \bL^2(\Omega)$.
\end{lem}
\begin{proof}
Let $\bz\in \bK$ and let $\ubu=(\bu,\bomega), \ubv=(\bv,\bpsi), \ubw=(\bw,\bphi)\in \bK\times \bL^2(\Omega)$.
From the definition of the operators $\cE, \cA(\bz)$ (cf. \eqref{eq:operator-E}, \eqref{eq:operator-Au}), 
using the continuity bounds \eqref{eq:continuity-monotonicity-E} and \eqref{eq:continuity-cz-u}, 
and the H\"older and Cauchy--Schwarz inequalities, we deduce that
\begin{equation}\label{eq:continuity-bound-1}
\begin{array}{l}
[(\cE + \cA(\bz))(\ubu) - (\cE + \cA(\bz))(\ubv),\ubw] 
\,\leq\, \tF\,\big\| |\bu|^{\rho-2}\bu - |\bv|^{\rho-2}\bv \big\|_{\bL^\upsilon(\Omega)}\,\|\bw\|_{\bL^\rho(\Omega)} \\[2ex] 
\ds\quad +\, 2\,\max\big\{ 1+\tD, \kappa^2,\nu \big\}\,\|\ubu - \ubv\|\,\|\ubw\| 
+ \|\bi_4\|^2\,\|\bz\|_{\bH^1(\Omega)}\,\|\bu - \bv\|_{\bH^1(\Omega)}\,\|\ubw\| \,,
\end{array}
\end{equation}
with $\upsilon\in [4/3,3/2]$ and $1/\rho + 1/\upsilon=1$.
In turn, using \cite[Lemma 2.1, eq. (2.1a)]{bl1993} and the continuous injection $\bi_\rho$ of $\bH^1(\Omega)$ into $\bL^\rho(\Omega)$ (cf. \eqref{eq:Sobolev-inequality}), we deduce that there exists a constant $c_\rho>0$ depending only on $|\Omega|$ and $\rho$, such that
\begin{equation}\label{eq:continuity-bound-2}
\begin{array}{l}
\ds \big\| |\bu|^{\rho-2}\bu - |\bv|^{\rho-2}\bv \big\|_{\bL^\upsilon(\Omega)} \|\bw\|_{\bL^\rho(\Omega)} 
\,\leq\, c_\rho\,\big( \|\bu\|_{\bL^\rho(\Omega)} + \|\bv\|_{\bL^\rho(\Omega)} \big)^{\rho-2}\,\|\bu - \bv\|_{\bL^\rho(\Omega)}\,\|\bw\|_{\bL^\rho(\Omega)}  \\[2ex]
\ds\quad \,\leq\, 2^{\rho-3}\,c_\rho\,\|\bi_\rho\|^\rho\,\big( \|\bu\|^{\rho-2}_{\bH^1(\Omega)} + \|\bv\|^{\rho-2}_{\bH^1(\Omega)} \big)\,\|\bu - \bv\|_{\bH^1(\Omega)}\,\|\bw\|_{\bH^1(\Omega)} \,.
\end{array}
\end{equation}
Then, replacing back \eqref{eq:continuity-bound-2} into \eqref{eq:continuity-bound-1},
and after simple computations, we obtain \eqref{eq:continuity-of-Az} with 
$$L_{\tKV} = \max\Big\{ 2\,\max\{ 1+\tD, \kappa^2, \nu \}, \|\bi_4\|^2, 2^{\rho-3}\,\tF\,\|\bi_\rho\|^\rho\,c_\rho \Big\}.$$
\end{proof}

We continue our analysis by proving the coercivity and strong monotonicity of the nonlinear operator $\cE + \cA(\bz)$ (cf. \eqref{eq:operator-E}, \eqref{eq:operator-Au}).
\begin{lem}\label{lem:coercivity-monotonicity-Az}
Let $\bz\in \bK$ (cf. \eqref{eq:def-K}).
Then, there exists $\gamma_{\tKV}>0$, depending only on $\tD$ and $\kappa$, such that
\begin{equation}\label{eq:coercivity-of-EAz}
[(\cE +\cA(\bz))(\ubv),\ubv] 
\,\geq\, \gamma_{\tKV}\,\|\bv\|^2_{\bH^1(\Omega)} + \nu\,\|\bpsi\|^2_{\bL^2(\Omega)},
\end{equation}
and
\begin{equation}\label{eq:strong-monotonicity-of-EAz}
[(\cE + \cA(\bz))(\ubu) - (\cE + \cA(\bz))(\ubv), \ubu - \ubv] 
\,\geq\, \gamma_{\tKV}\,\|\bu - \bv\|^2_{\bH^1(\Omega)} + \nu\,\|\bomega - \bpsi\|^2_{\bL^2(\Omega)} \,,
\end{equation}
for all $\ubu=(\bu,\bomega), \ubv=(\bv,\bpsi)\in \bK\times \bL^2(\Omega)$.
\end{lem}
\begin{proof}
Let $\bz\in \bK$ and let $\ubu=(\bu,\bomega), \ubv=(\bv,\bpsi)\in \bK\times \bL^2(\Omega)$.
Then, from the definition of the operators $\cE, \cA(\bz)$ (cf. \eqref{eq:operator-E}, \eqref{eq:operator-Au}) and the identity \eqref{eq:null-identity-cz}, we deduce that
\begin{equation}\label{eq:coercivity-1}
\begin{array}{l}
\ds [(\cE +\cA(\bz))(\ubv),\ubv]
\,=\, [\cE(\ubv),\ubv] + [\ba(\ubv),\ubv] + [\bc(\bz)(\ubv),\ubv] \\[2ex]
\ds\quad =\, (1 + \tD)\,\|\bv\|^2_{\bL^2(\Omega)} 
+ \kappa^2\,\|\nabla\bv\|^2_{\bbL^2(\Omega)} + \tF\,\|\bv\|^\rho_{\bL^\rho(\Omega)} 
+ \nu\,\|\bpsi\|^2_{\bL^2(\Omega)} \,,
\end{array}
\end{equation}	
which, together with the fact that the $\bL^\rho(\Omega)$-term on the right hand-side of \eqref{eq:coercivity-1}
can be neglected, yields \eqref{eq:coercivity-of-EAz} with $\gamma_{\tKV}:=\min \{ 1+\tD, \kappa^2\}$.

On the other hand, proceeding as in \eqref{eq:coercivity-1} and using the fact that $\cE$ and $\bc(\bz)$ are linear, we get
\begin{equation}\label{eq:strong-monotonicity-1}
\begin{array}{l}
\ds [(\cE + \cA(\bz))(\ubu) - (\cE + \cA(\bz))(\ubv), \ubu - \ubv]
\,=\, (1 + \tD)\,\|\bu - \bv\|^2_{\bL^2(\Omega)} 
+ \kappa^2\,\|\nabla(\bu - \bv)\|^2_{\bbL^2(\Omega)} \\[2ex] 
\ds\quad  +\,\, \tF\,(|\bu|^{\rho-2}\bu - |\bv|^{\rho-2}\bv,\bu - \bv)_\Omega 
+ \nu\,\|\bomega - \bpsi\|^2_{\bL^2(\Omega)} \,.
\end{array}
\end{equation}
Thanks to \cite[Lemma 2.1, eq. (2.1b)]{bl1993}, we know that there exists a constant $C_{\rho}>0$ depending only on $|\Omega|$ and $\rho$, such that
\begin{equation}\label{eq:strict-monotonicity-BF-term}
\big( |\bu|^{\rho-2}\bu - |\bv|^{\rho-2}\bv,\bu - \bv \big)_{\Omega} 
\,\geq\, C_\rho\,\|\bu - \bv\|^\rho_{\bL^{\rho}(\Omega)} 
\,>\, 0 \quad \forall\,\bu,\bv\in \bL^\rho(\Omega) \,.
\end{equation}
Thus, \eqref{eq:strong-monotonicity-1} yields \eqref{eq:strong-monotonicity-of-EAz} 
with the same constant $\gamma_{\tKV}$ as in \eqref{eq:coercivity-of-EAz}. 
\end{proof}

\begin{lem}\label{J-well-defined}
The operator $\cJ : \bK \to \bK$ introduced in \eqref{eq:operator-J} 
is well defined. In particular, for each $\bz \in \bK$, there exists a unique solution $\ubu=(\bu,\bomega)\in \bK\times \bL^2(\Omega)$ to \eqref{eq:auxiliary-reduced-problem} and $\cJ(\bz) = \bu$. Moreover,
\begin{equation}\label{eq:u-omega-bound-f}
\|\bu\|_{\bH^1(\Omega)} \,\leq\, \frac{1}{\gamma_{\tKV}}\|\wh{\f}\|_{\bH^{-1}(\Omega)}\,. 
\end{equation}
\end{lem}

\begin{proof}
Let $\bz\in \bK$.
Owing to the continuity, coercivity and strong monotonicity of the operator $\cE+\cA(\bz)$ (cf. Lemmas~\ref{lem:continuity-Az} and \ref{lem:coercivity-monotonicity-Az}), the well-posedness of \eqref{eq:auxiliary-reduced-problem} is a direct consequence of the Minty--Browder theorem (cf. Theorem~\ref{thm:Browder-Minty-theorem}). This is clearly equivalent to the existence of a unique $\bu\in \bK$, such that $\cJ(\bz) = \bu$. Moreover, \eqref{eq:u-omega-bound-f} follows readily by testing \eqref{eq:auxiliary-reduced-problem} with $\ubv = \ubu$ and using the coercivity bound of $\cE + \cA(\bz)$ (cf. \eqref{eq:coercivity-of-EAz}) and the continuity bound of $\wh{\bF}$ (cf. \eqref{eq:continuity-F}).
\end{proof}

We next derive a continuity bound for the operator $\cJ$.

\begin{lem}
For all $\bz_1,\bz_2\in \bK$, there holds
\begin{equation}\label{eq:aux-continuity-T}
\|\cJ(\bz_1) - \cJ(\bz_2)\|_{\bH^1(\Omega)} 
\,\leq\, \frac{\|\bi_4\|}{\gamma^2_{\tKV}}\,\|\wh{\f}\|_{\bH^{-1}(\Omega)}\,\|\bz_1 - \bz_2\|_{\bL^4(\Omega)} \,.
\end{equation}
\end{lem}
\begin{proof}
Given $\bz_1, \bz_2\in \bK$, we let $\bu_1 = \cJ(\bz_1)$ and $\bu_2 = \cJ(\bz_2)$.
According to the definition of $\cJ$ (cf. \eqref{eq:operator-J}--\eqref{eq:auxiliary-reduced-problem}), it follows that
\begin{equation*}
[(\cE + \cA(\bz_1))(\ubu_1) - (\cE + \cA(\bz_2))(\ubu_2),\ubv] \,=\, 0 \quad \forall\,\ubv\in \bK\times \bL^2(\Omega).
\end{equation*}
Taking $\ubv=\ubu_1 - \ubu_2$ in the above system, and recalling the definition of $\cE, \cA(\bz)$ (cf. \eqref{eq:operator-E}, \eqref{eq:operator-Au}), we obtain the identity
\begin{equation}
[(\cE + \cA(\bz_1))(\ubu_1) - (\cE + \cA(\bz_1))(\ubu_2),\ubu_1 - \ubu_2] 
\,=\, -[\bc(\bz_1 - \bz_2)(\ubu_2),\ubu_1 - \ubu_2] \,.
\end{equation}
Hence, using the strong monotonicity of $\cE + \cA(\bz)$ (cf. \eqref{eq:strong-monotonicity-of-EAz}) 
and the continuity bound of $\bc$ (cf. \eqref{eq:continuity-cz-z}), we deduce that
\begin{equation*}
\gamma_{\tKV}\,\|\bu_1 - \bu_2\|^2_{\bH^1(\Omega)}
\,\leq\, \|\bi_4\|\,\|\bu_2\|_{\bH^1(\Omega)}\,\|\bz_1 - \bz_2\|_{\bL^4(\Omega)}\,\|\bu_1 - \bu_2\|_{\bH^1(\Omega)} \,,
\end{equation*}
which, together with \eqref{eq:u-omega-bound-f},
implies \eqref{eq:aux-continuity-T}.
\end{proof}

\subsubsection{Solvability analysis of the fixed-point equation}

Having proved the well-posedness of the problem \eqref{eq:auxiliary-reduced-problem}, which ensures that the operator $\cJ$ is well defined, we now aim to establish existence of a fixed point of the operator $\cJ$.
For this purpose, in what follows we verify the hypothesis of the Schauder fixed-point theorem in a suitable closed set.

Let $\bW$ be the bounded and convex set defined by
\begin{equation}\label{eq:ball-W}
\bW := \Big\{ \bz\in \bK :\quad \|\bz\|_{\bH^1(\Omega)} 
\,\leq\, \frac{1}{\gamma_{\tKV}}\|\wh{\f}\|_{\bH^{-1}(\Omega)} \Big\} \,.
\end{equation}
The following lemma establishes the existence of a fixed point of $\cJ$ 
by means of the Schauder fixed point theorem.

\begin{lem}\label{lem:resolvent-system}
Let $\bW$ be defined as in \eqref{eq:ball-W}. Then the operator $\cJ$ has at least one fixed-point in $\bW$, that is, the resolvent system \eqref{eq:reduced-problem} has a solution $\ubu=(\bu,\bomega)\in \bW\times \bL^2(\Omega)$.
\end{lem}
\begin{proof}
Given $\bz\in \bW$, we first recall from Lemma \ref{J-well-defined} that $\cJ$ is well defined and there exists a unique $\bu\in\bK$ such that $\cJ(\bz)=\bu$, 
which together with \eqref{eq:u-omega-bound-f} implies that $\bu\in \bW$ and proves that $\cJ(\bW) \subseteq \bW$.
Next, we observe from estimate \eqref{eq:aux-continuity-T} that $\cJ$ is continuous.
In addition, using again \eqref{eq:aux-continuity-T}, the compactness of the injection $\bi_4 : \bH^1(\Omega)\to \bL^4(\Omega)$ (see, e.g., \cite[Theorem~1.3.5]{Quarteroni-Valli}), and the well-known fact that every bounded sequence in a Hilbert space has a weakly convergent subsequence,
we deduce that $\ov{\cJ(\bW)}$ is compact.
Then, using the Schauder fixed point theorem written in the form \cite[Theorem 9.12-1(b)]{Ciarlet},
we conclude that the operator $\cJ$ has at least one fixed-point in $\bW$, that is, there exists 
$\ubu=(\bu,\bomega)\in \bW\times \bL^2(\Omega)$ a solution 
to \eqref{eq:reduced-problem}.
\end{proof}


\subsection{Construction of compatible initial data}

Now, we establish a suitable initial condition result, which is necessary to apply Theorem~\ref{thm:well-posed-parabolic-problem} to the context of \eqref{eq:KVBF-reduced-VF}.
\begin{lem}\label{lem:sol0-reduced-problem}
Assume the initial condition $\bu_0\in \bK$ (cf. \eqref{eq:def-K}).
Then, there exists $\bomega_0\in \bL^2(\Omega)$ such that $\ubu_0 = (\bu_0,\bomega_0)$ and
\begin{equation}\label{eq:sol0-H1-L2}
\cA(\bu_0)(\ubu_0)\in \bH^{-1}(\Omega)\times \{\0\} \,.
\end{equation}
\end{lem}
\begin{proof}
We proceed as in \cite[Lemma 3.7]{accgry2023}.
In fact, we define $\bomega_0 := \bcurl(\bu_0)$, with $\bu_0\in \bK$ (cf. \eqref{eq:def-K}). 
It follows that $\bomega_0\in\bL^2(\Omega)$.
In addition, using \eqref{eq:curl-curl-identity}, we get
\begin{equation}\label{eq:omega0-identity}
\nu\,\bcurl(\bomega_0) \,=\, -\nu\,\Delta\bu_0 \qin \Omega \,.
\end{equation}
Next, multiplying the identities \eqref{eq:omega0-identity} and $\nu\,(\bomega_0 - \bcurl(\bu_0)) = \0$ by $\bv\in \bH^1_0(\Omega)$ and $\bpsi\in \bL^2(\Omega)$, respectively, integrating by parts as in \eqref{eq:Green-formula}, and after minor algebraic manipulation, we obtain
\begin{equation}\label{eq:initial-condition-problem}
[\cA(\bu_0)(\ubu_0), \ubv] \,=\, [\bF_0, \ubv] \quad \forall\,\ubv\in \bH^1_0(\Omega)\times \bL^2(\Omega) \,,
\end{equation}
with $\bF_0=(\f_0,\0)$ and
\begin{equation*}
(\f_0,\bv)_{\Omega} \,:=\, \nu\,(\nabla\bu_0,\nabla\bv)_\Omega 
+ \big(\tD\,\bu_0 + \tF\,|\bu_0|^{\rho-2}\bu_0 + (\nabla\bu_0)\bu_0,\bv\big)_\Omega\,,
\end{equation*}
which together with the continuous injection of $\bH^1(\Omega)$ into $\bL^4(\Omega)$ and $\bL^{\rho}(\Omega)$, with $\rho\in [3,4]$, cf. \eqref{eq:Sobolev-inequality}, implies that 
\begin{equation}\label{eq:continuity-bound-F0}
\big| (\f_0,\bv)_{\Omega} \big|
\,\leq\, C_0\,\Big\{ \|\bu_0\|_{\bH^1(\Omega)}
+ \|\bu_0\|^2_{\bH^1(\Omega)} 
+ \|\bu_0\|^{\rho-1}_{\bH^{1}(\Omega)} \Big\}\,\|\bv\|_{\bH^1(\Omega)} \,,
\end{equation}
with $C_0 := \max\big\{ \nu + \tD, \tF\,\|\bi_{\rho}\|^{\rho}, \|\bi_4\|^2 \big\}$.
Thus, $\bF_0\in \bH^{-1}(\Omega)\times \{\0\}$ so then \eqref{eq:sol0-H1-L2} holds, completing the proof.
\end{proof}

\begin{rem}
The assumption on the initial condition $\bu_0\in \bK$ (cf. \eqref{eq:def-K}) 
is less restrictive than the one employed in \cite[Lemma 3.7]{accgry2023}  
(see also \cite[Lemma 3.6]{cy2021}, \cite[Lemma 3.7]{covy2022} and \cite[eq. (2.2)]{dr2014}) 
for the analysis of the unsteady Brinkman--Forchheimer problem since now the datum $\f_0$ is look for in $\bH^{-1}(\Omega)$ instead of $\bL^2(\Omega)$.  
Note also that $\ubu_0$ satisfying \eqref{eq:sol0-H1-L2} is not unique.
In addition, $(\ubu_0,p_0) = ((\bu_0,\bcurl(\bu_0)),0)$ can be chosen as initial condition for \eqref{eq:KVBF-variational-formulation}, that is, $(\ubu_0,p_0)$ satisfy
\begin{equation}\label{eq:initial-condition-full-problem}
\begin{array}{llll}
[\cA(\bu_0)(\ubu_0), \ubv] + [\cB'(p_0),\ubv] & = & [\bF_0, \ubv] & \forall\,\ubv\in \bH^1_0(\Omega)\times \bL^2(\Omega)\,, \\[2ex]
-\,[\cB(\ubu_0),q] & = & 0 & \forall\,q\in \L^2_0(\Omega)\,.
\end{array}
\end{equation}
\end{rem}


\subsection{Main result}

We now establish the well-posedness and stability bounds for the solution of problem \eqref{eq:KVBF-reduced-VF}.
\begin{thm}\label{thm:well-posed-result-reduced}
For each compatible initial data $\ubu_0=(\bu_0,\bomega_0)$ constructed in Lemma \ref{lem:sol0-reduced-problem} and 	
each $\f\in \W^{1,1}(0,T;\bH^{-1}(\Omega))$, 
there exists a unique solution of \eqref{eq:KVBF-reduced-VF}, $\ubu = (\bu,\bomega) :[0,T]\to \bK\times \bL^2(\Omega)$ with $\bu\in \W^{1,\infty}(0,T;\bH^{-1}(\Omega))$ 
and $\bu(0) = \bu_0$. In addition, $\bomega(0) = \bomega_0 = \bcurl(\bu_0)$
and there exists a constant $C_{\tt KVr} > 0$ only depending on 
$\nu, \tD$, and $\kappa$, such that
\begin{equation}\label{eq:stability-bound-u-om-reduced}
\begin{array}{l}
\ds \|\bu\|_{\L^\infty(0,T;\bH^1(\Omega))}
+ \|\bu\|_{\L^{2}(0,T;\bL^2(\Omega))} 
+ \|\bomega\|_{\L^{2}(0,T;\bL^2(\Omega))} \\[2ex]
\ds\qquad \leq\, C_{\tt KVr}\,\sqrt{\exp(T)}\,\Big( 
\|\f\|_{\L^2(0,T;\bH^{-1}(\Omega))} 
+ \|\bu_0\|_{\bH^1(\Omega)}
\Big) \,.
\end{array}
\end{equation}
\end{thm}
\begin{proof}	
We recall that \eqref{eq:KVBF-reduced-VF} fits the problem 
in Theorem \ref{thm:well-posed-parabolic-problem} with 
the definitions \eqref{eq:setting-E-u-N-M} and \eqref{eq:setting-Eb-D}.	
Note that $\cN$ is linear, symmetric and monotone since $\cE$ is 
(cf. \eqref{eq:continuity-monotonicity-E}).
In addition, since $\cA(\bu)$ is strongly monotone for any $\bu\in \bK$, it follows that $\cM$ is monotone.
On the other hand, from Lemma \ref{lem:resolvent-system} we know that given $(\wh{\f}, \0) \in E_b'$ there exists $\ubu\in \bK\times \bL^2(\Omega)$, such that $(\wh{\f}, \0) = (\cN + \cM)(\ubu)$ which implies $Rg(\cN+\cM)=E_b'$.
Finally, considering $\bu_0\in \bK$, from a straightforward application of 
Lemma \ref{lem:sol0-reduced-problem} we are able to find $\bomega_0\in \bL^2(\Omega)$  
such that $\ubu_0 = (\bu_0,\bomega_0)\in \cD$ and $(\f_0,\0)\in E'_b$.
Therefore, applying Theorem~\ref{thm:well-posed-parabolic-problem} to our context, we 
conclude the existence of a solution $\ubu = (\bu,\bomega)$ to \eqref{eq:KVBF-reduced-VF}, 
with $\bu\in \W^{1,\infty}(0,T;\bH^{-1}(\Omega))$ and $\bu(0) = \bu_0$. 

We next show the stability bound \eqref{eq:stability-bound-u-om-reduced}, which will be used to prove that the solution of \eqref{eq:KVBF-reduced-VF} is unique.
Indeed, to derive \eqref{eq:stability-bound-u-om-reduced}, we proceed as in \cite[Theorem~3.3]{cy2021} and choose $\ubv = \ubu$ in \eqref{eq:KVBF-reduced-VF} to get
\begin{equation}\label{eq:stability-0}
\frac{1}{2}\,\partial_t\,\Big( \|\bu\|^2_{\bL^2(\Omega)} + \kappa^2\,\|\nabla\bu\|^2_{\bbL^2(\Omega)} \Big) 
+ [\cA(\bu)(\ubu),\ubu] = (\f,\bu)_{\Omega} \,.
\end{equation}
Next, from the definition of the operator $\cA(\bz)$ (cf. \eqref{eq:operator-Au}), 
employing similar arguments to \eqref{eq:coercivity-of-EAz} and 
using Cauchy--Schwarz and Young's inequalities, we obtain
\begin{equation}\label{eq:stability-1}
\ds \frac{\wh{\gamma}_{\tKV}}{2}\,\partial_t \|\bu\|^2_{\bH^1(\Omega)} 
+ \tD\,\|\bu\|^2_{\bL^2(\Omega)}
+ \tF\,\|\bu\|^\rho_{\bL^{\rho}(\Omega)} 
+ \nu\,\|\bomega\|^2_{\bL^2(\Omega)} 
\,\leq\, \frac{1}{2}\,\Big( \|\f\|^2_{\bH^{-1}(\Omega)} 
+ \|\bu\|^2_{\bH^1(\Omega)} \Big) \,,
\end{equation}
where $\wh{\gamma}_\tKV := \min\big\{ 1, \kappa^2 \big\}$.
Then, integrating \eqref{eq:stability-1} from $0$ to $t\in (0,T]$, we obtain
\begin{equation}\label{eq:stability-3}
\begin{array}{l}
\ds \wh{\gamma}_{\tKV}\,\|\bu(t)\|^2_{\bH^1(\Omega)} 
+ \int^t_0 \Big( 2\,\tD\,\|\bu\|^2_{\bL^2(\Omega)} 
+ 2\,\nu\,\|\bomega\|^2_{\bL^2(\Omega)} \Big) ds \\[2ex]
\ds\quad \,\leq\, \int^t_0 \|\f\|^2_{\bH^{-1}(\Omega)}\,ds + \wh{\gamma}_{\tKV}\,\|\bu(0)\|^2_{\bH^1(\Omega)} + \int^t_0 \|\bu\|^2_{\bH^1(\Omega)}\,ds\,,
\end{array}
\end{equation}
which, together with the Gr\"onwall inequality and the fact that $\bu(0)=\bu_0$, yields \eqref{eq:stability-bound-u-om-reduced}.
Notice that, in order to simplify the stability bound, we have neglected the positive term $\int^t_0 \|\bu\|^\rho_{\bL^{\rho}(\Omega)}\,ds$ in the left hand side of \eqref{eq:stability-3}.	

The aforementioned uniqueness of \eqref{eq:KVBF-reduced-VF} is now provided.	
In fact, let $\ubu_i = (\bu_i,\bomega_i)$, with $i\in \{1,2\}$, 
be two solutions corresponding to the same data. Then, taking \eqref{eq:KVBF-reduced-VF} 
with $\ubv = \ubu_1 - \ubu_2\in \bK\times \bL^2(\Omega)$, 
subtracting the problems, we deduce that
\begin{equation*}
\begin{array}{l}
\ds \frac{1}{2}\,\partial_t\,\Big( \|\bu_1 - \bu_2\|^2_{\bL^2(\Omega)} + \kappa^2\,\|\nabla(\bu_1 - \bu_2)\|^2_{\bbL^2(\Omega)} \Big) \\[2ex]
\ds\quad +\, [\cA(\bu_1)(\ubu_1) - \cA(\bu_1)(\ubu_2), \ubu_1 - \ubu_2] \,=\, - [\bc(\bu_1 - \bu_2)(\ubu_2),\ubu_1 - \ubu_2] \,.
\end{array}
\end{equation*}
Then, using similar arguments to \eqref{eq:strong-monotonicity-of-EAz}, the definition of the operator $\cA(\bz)$ (cf. \eqref{eq:operator-Au}), the identity \eqref{eq:null-identity-cz}, \eqref{eq:strict-monotonicity-BF-term}, and the continuity bound of $\bc(\bz)$ (cf. \eqref{eq:continuity-cz-z}), we get
\begin{equation}\label{eq:uniqueness-0}
\begin{array}{l}
\ds \frac{\wh{\gamma}_\tKV}{2}\,\partial_t \|\bu_1 - \bu_2\|^2_{\bH^1(\Omega)} 
+ \tD\,\|\bu_1 - \bu_2\|^2_{\bL^2(\Omega)} 
+ \nu\,\|\bomega_1 - \bomega_2\|^2_{\bL^2(\Omega)} \\[2ex] 
\ds\qquad \leq\, \|\bi_4\|^2 \|\bu_2\|_{\bH^1(\Omega)} \|\bu_1 - \bu_2\|^2_{\bH^1(\Omega)} \,,
\end{array}
\end{equation}
with $\wh{\gamma}_\tKV$ as in \eqref{eq:stability-1}.
Integrating in time \eqref{eq:uniqueness-0} from $0$ to $t\in (0,T]$, using the fact that $\|\bu\|_{\L^\infty(0,T;\bH^1(\Omega))}$ is bounded by data (cf. \eqref{eq:stability-bound-u-om-reduced}) in conjunction with the Gr\"onwall inequality and algebraic manipulations, we obtain
\begin{equation*}
\|\bu_1(t) - \bu_2(t)\|^2_{\bH^1(\Omega)} 
+ \int^t_0 \Big( \|\bu_1 - \bu_2\|^2_{\bL^2(\Omega)} 
+ \|\bomega_1 - \bomega_2\|^2_{\bL^2(\Omega)} \Big)\,ds 
\,\leq\, C\,\exp(T)\,\|\bu_1(0) - \bu_2(0)\|^2_{\bH^1(\Omega)} \,,
\end{equation*}
with $C>0$ depending on $\nu, \tD, \kappa, \|\bi_4\|$, and data.
Therefore, recalling that $\bu_1(0) = \bu_2(0)$, it follows that $\bu_1(t) = \bu_2(t)$ and $\bomega_1(t) = \bomega_2(t)$ for all $t\in (0,T]$.
	
Finally, since Theorem \ref{thm:well-posed-parabolic-problem} implies that 
$\cM(u)\in \L^\infty(0,T;E'_b)$, we can take $t\to 0$ in all equations 
without time derivatives in \eqref{eq:KVBF-reduced-VF}.
Using that the initial data $\ubu_0=(\bu_0,\bomega_0)$ satisfies 
the same equations at $t=0$ (cf. \eqref{eq:sol0-H1-L2}), and 
that $\bu(0) = \bu_0$, we obtain
\begin{equation}\label{eq:omega-t-0}
\nu\,(\bomega(0) - \bomega_0,\bpsi)_{\Omega} \,=\, 0 \quad \forall\,\bpsi\in \bL^2(\Omega)\,.
\end{equation}
Thus, taking $\bpsi = \bomega(0) - \bomega_0$ in \eqref{eq:omega-t-0} we deduce that $\bomega(0) = \bomega_0 = \bcurl(\bu_0)$, completing the proof.
\end{proof}

\medskip
We conclude this section by establishing the well-posedness and stability bounds for the solution of problem \eqref{eq:KVBF-variational-formulation}.
\begin{thm}\label{thm:stability-bound}
For each $\f\in \W^{1,1}(0,T;\bH^{-1}(\Omega))$ and $\bu_0\in \bK$, 
there exists a unique solution of \eqref{eq:KVBF-variational-formulation}, 
$(\ubu, p) = ((\bu,\bomega), p) :[0,T]\to \big(\bH^1_0(\Omega)\times \bL^2(\Omega)\big)\times \L^2_0(\Omega)$ 
with $\bu\in \W^{1,\infty}(0,T;\bH^{-1}(\Omega))$ 
and $\bu(0) = \bu_0$. In addition, $\bomega(0) = \bomega_0 = \bcurl(\bu_0)$
and there holds the stability bound \eqref{eq:stability-bound-u-om-reduced} with the same constant $C_{\tt KVr}$ only depending on 
$\nu, \tD$, and $\kappa$.
Moreover, there exists a constant $C_{\tt KVp} > 0$ only depending on 
$|\Omega|$, $\|\bi_\rho\|, \|\bi_4\|, \nu, \tD, \tF, \kappa$, and $\beta$, such that
\begin{equation}\label{eq:stability-bound}
\|p\|_{\L^{2}(0,T;\L^2(\Omega))} 
\,\leq\, C_{\tt KVp}\,\Bigg( 
\sum_{j\in \{ 2,3,\rho \}} \left\{ \sqrt{\exp(T)}\,\Big( \|\f\|_{\L^2(0,T;\bH^{-1}(\Omega))} 
+ \|\bu_0\|_{\bH^1(\Omega)} \Big) \right\}^{j-1}
+ \|\bu_0\|^{\rho/2}_{\bL^\rho(\Omega)}
\Bigg) \,.
\end{equation}
\end{thm}
\begin{proof}
We begin by recalling from Lemma \ref{lem:equivalent-problems} that the problems
\eqref{eq:KVBF-variational-formulation} and \eqref{eq:KVBF-reduced-VF} are equivalent.
Thus, the well-posedness of \eqref{eq:KVBF-variational-formulation} follows from Theorem \ref{thm:well-posed-result-reduced}.

On the other hand, to derive \eqref{eq:stability-bound-u-om-reduced} and \eqref{eq:stability-bound}, we first choose $\ubv = \ubu$ and $q=p$ in \eqref{eq:KVBF-variational-formulation} to deduce \eqref{eq:stability-0}--\eqref{eq:stability-3} and consequently \eqref{eq:stability-bound-u-om-reduced} also holds for the problem \eqref{eq:KVBF-variational-formulation}.
In turn, from the inf-sup condition of $\cB$ (cf. \eqref{eq:B-inf-sup-condition}), 
the first equation of \eqref{eq:KVBF-variational-formulation} related to $\bv$, 
the stability bounds of $\bF, \cE$ (cf. \eqref{eq:continuity-F}, \eqref{eq:continuity-monotonicity-E}), 
the definition of $\cA(\bz)$ (cf. \eqref{eq:operator-Au}), and the continuous injections 
of $\bH^1(\Omega)$ into $\bL^4(\Omega)$ and $\bL^{\rho}(\Omega)$, with $\rho\in [3,4]$, we deduce that
\begin{equation}\label{eq:stability-p-inf-sup}
\begin{array}{l}
\ds \beta\,\|p\|_{\L^2(\Omega)}
\,\leq\, 
\sup_{\0\neq\bv\in \bH^1_0(\Omega)} \frac{[\bF,(\bv,\0)] - [\partial_t\,\cE(\ubu),(\bv,\0)] - [\cA(\bu)(\ubu),(\bv,\0)]}{\|\bv\|_{\bH^1(\Omega)}} \\ [3ex]
\ds\quad \,\leq\, \|\f\|_{\bH^{-1}(\Omega)} 
+ \tD\,\|\bu\|_{\bL^2(\Omega)}
+ \nu\,\|\bomega\|_{\bL^2(\Omega)} \\[2ex]
\ds\qquad 
\,+\, \|\bi_4\|^2\,\|\bu\|^2_{\bH^1(\Omega)}  
+ \tF\,\|\bi_\rho\|^\rho\,\|\bu\|^{\rho-1}_{\bH^1(\Omega)} 
+ (1 + \kappa^2)\,\|\partial_t\,\bu\|_{\bH^1(\Omega)} \,.
\end{array}
\end{equation}
Then, taking square in \eqref{eq:stability-p-inf-sup} and integrating from $0$ to $t\in (0,T]$, we deduce that there exits $C_1>0$ depending on $|\Omega|, \|\bi_\rho\|, \|\bi_4\|, \nu, \tF, \tD, \kappa$, and $\beta$, such that
\begin{equation}\label{eq:stability-4}
\begin{array}{l}
\ds \int^t_0 \|p\|^2_{\L^2(\Omega)}\,ds  
\leq C_1\,\Bigg\{ \int^t_0 \Big( \|\f\|^2_{\bH^{-1}(\Omega)}
+ \|\bu\|^2_{\bL^2(\Omega)}
+ \|\bomega\|^2_{\bL^2(\Omega)} \Big)\,ds \\[3ex]
\ds\quad \,+\,  \int^t_0 \Big( \|\bu\|^4_{\bH^1(\Omega)}
+ \|\bu\|^{2\,(\rho-1)}_{\bH^1(\Omega)} 
+ \|\partial_t\,\bu\|^2_{\bH^1(\Omega)} \Big)\,ds \Bigg\} \,.
\end{array}
\end{equation}
Next, in order to bound the last term in \eqref{eq:stability-4}, we differentiate in time 
the equations of \eqref{eq:KVBF-variational-formulation} related to $\bpsi$ and $q$, 
choose $(\ubv,q) = ((\partial_t\,\bu,\bomega),p)$, use \eqref{eq:Convective-bound} in conjunction with  Cauchy--Schwarz and Young's inequalities, to find that
\begin{equation}\label{eq:stability-4-b}
\begin{array}{l}
\ds \frac{1}{2}\,\partial_t\,\Big(\tD\,\|\bu\|^2_{\bL^2(\Omega)} 
+ \frac{2\,\tF}{\rho}\,\|\bu\|^{\rho}_{\bL^{\rho}(\Omega)} 
+ \nu\,\|\bomega\|^2_{\bL^2(\Omega)} \Big)
+ \wh{\gamma}_{\tKV}\,\|\partial_t\,\bu\|^2_{\bH^1(\Omega)} \\[2ex]
\ds\quad \,\leq\, \Big(\|\f\|_{\bH^{-1}(\Omega)} + \|\bi_4\|^2\,\|\bu\|^2_{\bH^1(\Omega)}\Big)\,\|\partial_t\,\bu\|_{\bH^1(\Omega)}
\\[3ex]
\ds\quad \,\leq\, C_2\,\Big( \|\f\|^2_{\bH^{-1}(\Omega)} + \|\bu\|^4_{\bH^1(\Omega)} \Big) 
+ \frac{\wh{\gamma}_{\tKV}}{2}\,\|\partial_t\,\bu\|^2_{\bH^1(\Omega)} \,,
\end{array}
\end{equation}
with $\wh{\gamma}_\tKV$ as in \eqref{eq:stability-1} and $C_2>0$ only depending on $\|\bi_4\|$ and $\kappa$.
Thus, integrating \eqref{eq:stability-4-b} from $0$ to $t\in (0,T]$, we get
\begin{equation}\label{eq:stability-5}
\begin{array}{l}
\ds \tD\,\|\bu(t)\|^2_{\bL^2(\Omega)} 
+ \frac{2\,\tF}{\rho}\,\|\bu(t)\|^\rho_{\bL^{\rho}(\Omega)}
+ \nu\,\|\bomega(t)\|^2_{\bL^2(\Omega)} 
+ \wh{\gamma}_{\tKV}\,\int^t_0 \|\partial_t\,\bu\|^2_{\bH^1(\Omega)}\,ds \\[3ex]
\ds\quad \leq\, 2\,C_2 \int^t_0 \Big( \|\f\|^2_{\bH^{-1}(\Omega)} 
+ \|\bu\|^4_{\bH^1(\Omega)} \Big) ds 
+ \tD\,\|\bu(0)\|^2_{\bL^2(\Omega)} 
+ \frac{2\,\tF}{\rho}\,\|\bu(0)\|^\rho_{\bL^{\rho}(\Omega)}
+ \nu\,\|\bomega(0)\|^2_{\bL^2(\Omega)} \,.
\end{array}
\end{equation}
Combining \eqref{eq:stability-4} with \eqref{eq:stability-3} and \eqref{eq:stability-5}, and using the fact that $(\bu(0),\bomega(0)) = (\bu_0,\bomega_0)$ and $\bomega_0 = \bcurl(\bu_0)$ in $\Omega$ (cf. Lemma \ref{lem:sol0-reduced-problem}), we deduce that
\begin{equation}\label{eq:stability-6}
\begin{array}{l}
\ds \int^t_0 \|p\|^2_{\L^2(\Omega)}\,ds  
\,\leq\, C_3\,\Bigg\{ \int^t_0 \|\f\|^2_{\bH^{-1}(\Omega)}\,ds
+ \|\bu_0\|^2_{\bH^1(\Omega)} 
+ \|\bu_0\|^\rho_{\bL^\rho(\Omega)} \\[3ex]
\ds\quad \,+\, \int^t_0 \Big( \|\bu\|^2_{\bH^1(\Omega)}
+ \|\bu\|^4_{\bH^1(\Omega)}
+ \|\bu\|^{2\,(\rho-1)}_{\bH^1(\Omega)} \Big)\,ds
\Bigg\} \,,
\end{array} 
\end{equation}
with $C_3>0$ only depending on $|\Omega|, \|\bi_\rho\|, \|\bi_4\|, \nu, \tF, \tD, \kappa$, and $\beta$.
Finally, using \eqref{eq:stability-bound-u-om-reduced} to bound $\|\bu\|^2_{\bH^1(\Omega)}$, $\|\bu\|^4_{\bH^1(\Omega)}$, and $\|\bu\|^{2\,(\rho-1)}_{\bH^1(\Omega)}$ in the left-hand side of \eqref{eq:stability-6}, we obtain \eqref{eq:stability-bound} concluding the proof.
\end{proof}

\begin{rem}
Observe that \eqref{eq:stability-bound} can be expanded to include a bound on $\|\partial_t\,\bu\|_{\L^2(0,T;\bH^1(\Omega))}$ and
$\|\bomega\|_{\L^{\infty}(0,T;\bL^2(\Omega))}$, using \eqref{eq:stability-5}. 
We also note that \eqref{eq:stability-bound-u-om-reduced} will be employed in the next section to deal with the nonlinear terms associated to the operator $\cA$ (cf. \eqref{eq:operator-Au}), which is necessary to obtain the corresponding error estimate.
\end{rem}


\section{Semidiscrete continuous-in-time approximation}\label{sec:semidiscrete-approximation}

In this section we introduce and analyze the semidiscrete continuous-in-time approximation of \eqref{eq:KVBF-variational-formulation}. 
We analyze its solvability by employing the strategy developed in Section \ref{sec:well-posedness-model}.
Finally, we derive the error estimates and obtain the corresponding rates of convergence.

\subsection{Existence and uniqueness of a solution}\label{sec:existence-uniqueness-solution}

Let $\cT_h$ be a shape-regular triangulation of $\Omega$ consisting 
of triangles $K$ (when $d=2$) or tetrahedra $K$ (when $d=3$) of diameter $h_K$, 
and define the mesh-size $h:=\max\big\{h_K:\,K\in\cT_h\big\}$. Let $(\bH^{\bu}_h, \H^{p}_h)$ be a pair of stable Stokes elements satisfying the discrete inf-sup condition: there exists a constant $\beta_\ttd>0$, independent of $h$, such that 
\begin{equation}\label{discrete-inf-sup}
\sup_{\0\neq \bv_h\in \bH^{\bu}_h} \frac{\ds\int_{\Omega} q_h\,\div(\bv_h) }{\|\bv_h\|_{\bH^{1}(\Omega)}}
\,\geq\, \beta_\ttd\,\|q_h\|_{\L^2(\Omega)} \quad \forall\,q_h\in \H^p_h\,.
\end{equation}
We refer the reader to \cite{BBF2013} and \cite{Brezzi-Fortin} for examples of stable Stokes elements. To simplify the presentation, we focus on Taylor--Hood \cite{th1973} finite elements for velocity and pressure, and continuous piecewise polynomials spaces for vorticity.
Given an integer $l\geq 0$ and a subset $S$ of $\R^d$, we denote by $\rP_l(S)$
the space of polynomials of total degree at most $l$ defined on $S$. For any $k\geq 1$, we consider:
\begin{align}\label{eq:Taylor-Hood}
\bH^{\bu}_h  &:= \Big\{\bv_h\in [C(\ov{\Omega})]^d :\quad \bv_h|_K \in [\rP_{k+1}(K)]^d 
\quad \forall\,K \in \cT_h \Big\}\cap \bH_0^1(\Omega) \,, \nonumber \\[1ex]
\H^p_h &:= \Big\{q_h\in C(\ov{\Omega}) :\quad q_h|_K \in \rP_k(K) \quad \forall\,K\in \cT_h \Big\}\cap \L_0^2(\Omega) \,, \\[1ex]
\bH^{\bomega}_h &:= \Big\{\bomega_h\in [C(\ov{\Omega})]^{d(d-1)/2} :\quad \bomega_h|_K \in [\rP_k(K)]^{d(d-1)/2} \quad \forall\,K \in \cT_h \Big\} \,. \nonumber
\end{align}
It is well known that the pair $(\bH^{\bu}_h, \H^{p}_h)$ in \eqref{eq:Taylor-Hood} satisfies 
\eqref{discrete-inf-sup} (cf. \cite{Boffi1994}). 
We observe that similarly to \cite{acgmr2021} and \cite{accgry2023}, we can also consider discontinuous piecewise polynomials spaces for the vorticity, that is,
\begin{equation*}
\bH^{\bomega}_h := \Big\{ \bomega_h\in [\L^2(\Omega)]^{d(d-1)/2} :\quad 
\bomega_h|_K \in [\rP_k(K)]^{d(d-1)/2}  \quad \forall\,K \in \cT_h \Big\} \,.
\end{equation*}
In addition to the Taylor--Hood elements for the velocity and pressure, in the numerical experiments in Section \ref{sec:numerical-results} we also consider the classical MINI-element \cite[Sections 8.4.2, 8.6 and 8.7]{BBF2013}
and Crouzeix--Raviart elements with tangential jump penalization (see \cite{crouzeix73}  for the discrete inf-sup condition regarding the lowest-order case and, for instance, \cite{carstensen21} for cubic order).

Now, defining $\ubu_h:=(\bu_h,\bomega_h), \ubv_h:=(\bv_h,\bpsi_h)\in \bH^{\bu}_h\times \bH^{\bomega}_h$, 
the semidiscrete continuous-in-time problem associated with 
\eqref{eq:KVBF-variational-formulation} reads: 
Find $(\ubu_h,p_h) : [0,T]\to \big( \bH^{\bu}_h\times \bH^{\bomega}_h \big)\times \H^p_h$  
such that, for a.e. $t\in (0,T)$,
\begin{equation}\label{eq:discrete-weak-KVBF}
\begin{array}{llll}
\dfrac{\partial}{\partial\,t}\,[\cE(\ubu_h(t)),\ubv_h] +  [\cA_h(\bu_h(t))(\ubu_h(t)),\ubv_h] + [\cB(\ubv_h), p_h(t)]
& = & [\bF(t),\ubv_h] & \forall\, \ubv_h\in \bH^{\bu}_h\times\bH^{\bomega}_h,\\[2ex] 
-[\cB(\ubu_h(t)), q_h] & = & 0 & \forall\, q_h\in \H^{p}_h \,.
\end{array}
\end{equation}
Here, $\cA_h(\bz_h): \big( \bH^{\bu}_h\times \bH^{\bomega}_h \big)\to \big( \bH^{\bu}_h\times \bH^{\bomega}_h \big)'$ is the discrete version of $\cA(\bz)$ (with $\bz\in \bH^{\bu}_h$ in place of $\bz\in \bH^1_0(\Omega)$), which is defined by
\begin{equation}\label{eq:operator-Auh}
[\cA_h(\bz_h)(\ubu_h),\ubv_h] \,:=\, [\ba(\ubu_h),\ubv_h] \,+\, [\bc_h(\bz_h)(\ubu_h),\ubv_h] \,,
\end{equation}
where $\bc_h(\bz_h)$ is the well-known skew-symmetric convection form \cite{Temam1977}:
\begin{equation*}
[\bc_h(\bz_h)(\ubu_h),\ubv_h] := ((\nabla\bu_h)\bz_h,\bv_h)_{\Omega} + \frac{1}{2}\,(\div(\bz_h)\bu_h,\bv_h)_{\Omega} \,,
\end{equation*}
for all $\bu_h, \bv_h, \bz_h\in \bH^{\bu}_h$.
Observe that integrating by parts, similarly to \eqref{eq:null-identity-cz}, there holds
\begin{equation}\label{eq:null-identity-czh}
[\bc_h(\bz_h)(\ubv_h),\ubv_h] = 0 \quad \forall\,\bz_h\in \bH^\bu_h \qan \forall\,\ubv_h\in \bH^{\bu}_h\times \bH^{\bomega}_h \,.
\end{equation}
As initial condition we take $(\ubu_{h,0},p_{h,0})=((\bu_{h,0},\bomega_{h,0}),p_{h,0})$ 
to be a suitable approximations of $(\ubu_0,p_0)=((\bu_0,\bomega_0),0)$, the solution of a slight modification of \eqref{eq:initial-condition-full-problem}, that is, we chose $(\ubu_{h,0}, p_{h,0})$ solving 
\begin{equation}\label{eq:discrete-initial-condition-full-problem}
\begin{array}{lll}
(\nabla\bu_{h,0},\nabla\bv_{h})_\Omega + [\cA_h(\bu_{h,0})(\ubu_{h,0}), \ubv_h] + [\cB'(p_{h,0}),\ubv_h] & = & [\bF_0, \ubv_h] + (\nabla\bu_0,\nabla\bv_{h})_\Omega \,, \\[2ex]
-\,[\cB(\ubu_{h,0}),q_h] & = & 0 \,,
\end{array}
\end{equation}
for all $\ubv_h\in \bH^{\bu}_h\times \bH^{\bomega}_h$ and $q_h\in \H^p_h$.
The well-posedness of \eqref{eq:discrete-initial-condition-full-problem} follows from the discrete inf-sup condition \eqref{eq:discrete-inf-sup-condition} and similar arguments to the proof of Lemma \ref{lem:resolvent-system}.
Alternatively, we can proceed as in \cite[eq. (4.4)]{accgry2023} and apply a fixed-point strategy in conjunction with \cite[Theorem 3.1]{cgo2021} to ensure existence and uniqueness of \eqref{eq:discrete-initial-condition-full-problem}.

Next, we introduce the discrete kernel of $\cB$, that is, 
\begin{equation}\label{eq:discrete-kernel-B}
\bV_h := \bK_h\times \bH^{\bomega}_h\,,\quad \mbox{where}\quad
\bK_h = \Big\{ \bv_h\in \bH^{\bu}_h :\quad (q_h,\div(\bv_h))_{\Omega} \,=\, 0 \quad 
\forall\, q_h\in \H^p_h \Big\} \,.
\end{equation}
Then, we can introduce the reduced problem: Given $\f:[0,T]\to \bH^{-1}(\Omega)$,
find $\ubu_h : [0,T]\to \bK_h\times \bH^{\bomega}_h$  
such that, for a.e. $t\in (0,T)$,
\begin{equation}\label{eq:discrete-weak-KVBF-reduced}
\dfrac{\partial}{\partial\,t}\,[\cE(\ubu_h(t)),\ubv_h] +  [\cA_h(\bu_h(t))(\ubu_h(t)),\ubv_h] 
= [\bF(t),\ubv_h] \quad \forall\, \ubv_h\in \bK_h\times\bH^{\bomega}_h \,,
\end{equation}
which, using \eqref{discrete-inf-sup} and similarly to Lemma \ref{lem:equivalent-problems}, is equivalent to \eqref{eq:discrete-weak-KVBF}.
As a preliminary initial condition for \eqref{eq:discrete-weak-KVBF-reduced} we take $\ubu_{h,0}:=(\bu_{h,0},\bomega_{h,0})$ 
to be solution of the reduced problem associated to \eqref{eq:discrete-initial-condition-full-problem}, that is, we chose $\ubu_{h,0}$ solving 
\begin{equation}\label{eq:pre-discrete-initial-condition}
(\nabla\bu_{h,0},\nabla\bv_{h})_\Omega + [\cA_h(\bu_{h,0})(\ubu_{h,0}),\ubv_h] = [\bF_0,\ubv_h] + (\nabla\bu_0,\nabla\bv_{h})_\Omega \quad \forall\, \ubv_h\in \bK_h\times\bH^{\bomega}_h \,,
\end{equation}
with $\bF_0\in \bH^{-1}(\Omega)\times\{\0\}$ being the right-hand side of \eqref{eq:initial-condition-problem}. 
Notice that the well-posedness of problem \eqref{eq:pre-discrete-initial-condition} 
follows from similar arguments to the proof of Lemma \ref{lem:resolvent-system}. 
In addition, taking $\ubv_h=\ubu_{h,0}$ in \eqref{eq:pre-discrete-initial-condition},
we deduce from the definition of the operator $\cA_h(\bu_{h,0})$ (cf. \eqref{eq:operator-Auh}),
the identity \eqref{eq:null-identity-czh},
and the continuity bound of $\bF_0$ (cf. \eqref{eq:continuity-bound-F0}), that
there exists a constant $\wh{C}_{0}>0$ depending only on $|\Omega|$, $\|\bi_\rho\|$, $\|\bi_4\|$, $\nu$, $\tD$, and $\tF$, and hence independent of $h$, such that 
\begin{equation}\label{eq:uh0-wh0-bound-u0-w0}
\|\bu_{h,0}\|^2_{\bH^1(\Omega)} 
+ \|\bu_{h,0}\|^\rho_{\bL^\rho(\Omega)}
+ \|\bomega_{h,0}\|^2_{\bL^2(\Omega)} 
\,\leq\, \wh{C}_{0}\,\Big\{ 
\|\bu_0\|^2_{\bH^1(\Omega)} 
+ \|\bu_0\|^4_{\bH^1(\Omega)}  
+ \|\bu_0\|^{2(\rho-1)}_{\bH^1(\Omega)}
\Big\}\,.
\end{equation}
Thus, from \eqref{eq:pre-discrete-initial-condition} we deduce a initial condition for \eqref{eq:discrete-weak-KVBF-reduced}, that is, $\ubu_{h,0}:=(\bu_{h,0},\bomega_{h,0})$, the solution of
\begin{equation}\label{eq:discrete-initial-condition}
[\cA_h(\bu_{h,0})(\ubu_{h,0}),\ubv_h] = [\bF_{h,0},\ubv_h] \quad \forall\, \ubv_h\in \bK_h\times\bH^{\bomega}_h\,,
\end{equation}
with $\bF_{h,0}=(\f_{h,0},\0)$ and
$(\f_{h,0},\bv_h)_\Omega := (\f_0,\ubv_h)_\Omega + (\nabla\bu_0,\nabla\bv_{h})_\Omega - (\nabla\bu_{h,0},\nabla\bv_{h})_\Omega$,
which, thanks to \eqref{eq:continuity-bound-F0} and \eqref{eq:uh0-wh0-bound-u0-w0}, yields
\begin{equation}\label{eq:assumption-fh0}
\big| (\f_{h,0},\bv_h)_\Omega \big| \,\leq\, 
C_{\ttd,0}\,\Big\{ \|\bu_0\|_{\bH^1(\Omega)}
+ \|\bu_0\|^2_{\bH^1(\Omega)} 
+ \|\bu_0\|^{\rho-1}_{\bH^{1}(\Omega)} \Big\}\,\|\bv_h\|_{\bH^1(\Omega)} \,,
\end{equation} 
with $C_{\ttd,0}>0$, depending only on $|\Omega|$, $\|\bi_\rho\|$, $\|\bi_4\|$, $\nu$, $\tD$, and $\tF$.
Thus, $\bF_{h,0}\in \bH^{-1}(\Omega)\times \{ \0 \}$.
We observe that this choice is necessary to guarantee that the discrete initial data 
is compatible in the sense of Lemma \ref{lem:sol0-reduced-problem}, which is needed for 
the application of Theorem \ref{thm:well-posed-parabolic-problem}. 

In this way, the well-posedness of \eqref{eq:discrete-weak-KVBF-reduced} (equivalently of \eqref{eq:discrete-weak-KVBF}), 
follows analogously to its continuous counterpart
provided in Theorem~\ref{thm:well-posed-result-reduced}.
More precisely, we first address the discrete counterparts of Lemmas \ref{lem:continuity-Az} and \ref{lem:coercivity-monotonicity-Az}, whose proofs, being almost verbatim of the continuous ones, are omitted.
\begin{lem}\label{lem:discrete-properties-A}
Let $\bz_h\in \bK_h$ (cf. \eqref{eq:discrete-kernel-B}).
Then, with the same constant $\gamma_{\tKV}$ defined in \eqref{eq:coercivity-of-EAz}, there holds
\begin{equation}\label{eq:coercivity-of-EAzh}
[(\cE +\cA_h(\bz_h))(\ubv_h),\ubv_h] 
\,\geq\, \gamma_{\tKV}\,\|\bv_h\|^2_{\bH^1(\Omega)} + \nu\,\|\bpsi_h\|^2_{\bL^2(\Omega)},
\end{equation}
and
\begin{equation}\label{eq:strong-monotonicity-of-EAzh}
[(\cE + \cA_h(\bz_h))(\ubu_h) - (\cE + \cA_h(\bz_h))(\ubv_h), \ubu_h - \ubv_h] 
\,\geq\, \gamma_{\tKV}\,\|\bu_h - \bv_h\|^2_{\bH^1(\Omega)} + \nu\,\|\bomega_h - \bpsi_h\|^2_{\bL^2(\Omega)} \,,
\end{equation}
for all $\ubu_h = (\bu_h,\bomega_h)$, $\ubv_h = (\bv_h,\bpsi_h)\in \bH^{\bu}_h\times \bH^{\bomega}_h$.
In addition, the operator $\cE+\cA_h:(\bH^{\bu}_h\times \bH^{\bomega}_h)\to (\bH^{\bu}_h\times \bH^{\bomega}_h)'$ is continuous in the sense of \eqref{eq:continuity-of-Az}, but with the constant $$L_{\tKV,\ttd} = \max\left\{ 2\,\max\{ 1+\tD, \kappa^2, \nu \}, \|\bi_4\|^2\left( 1 + \frac{\sqrt{d}}{2} \right), 2^{\rho-3}\,\tF\,\|\bi_\rho\|^\rho\,c_\rho \right\}.$$
\end{lem}

We continue with the discrete inf-sup condition of $\cB$.
\begin{lem}\label{lem:discrete-inf-sup}
It holds that
\begin{equation}\label{eq:discrete-inf-sup-condition}
\sup_{\0\ne \ubv_h\in\bH^{\bu}_h\times\bH^{\bomega}_h}\dfrac{[\cB(\ubv_h), q_h]}{\|\ubv_h\|}
\,\geq\, \beta_\ttd\,\|q_h\|_{\L^2(\Omega)} \quad \forall\, q_h\in \H^{p}_h \,,
\end{equation}
where $\beta_\ttd$ is the inf-sup constant from \eqref{discrete-inf-sup}.
\end{lem}
\begin{proof}
The statement follows directly from \eqref{discrete-inf-sup}.
\end{proof}

We are now in a position to establish the semi-discrete continuous in time analogue of Theorems \ref{thm:well-posed-result-reduced} and \ref{thm:stability-bound}.
To that end, we first introduce the ball of $\bK_h$ given by
\begin{equation}\label{eq:discrete-ball-W}
\bW_{\ttd} := \Big\{ \bz_h\in \bK_h :\quad \|\bz_h\|_{\bH^1(\Omega)} 
\,\leq\, \frac{1}{\gamma_{\tKV}}\|\wh{\f}\|_{\bH^{-1}(\Omega)} \, \Big\} \,.
\end{equation}
The aforementioned result is stated now.
\begin{thm}\label{thm:well-posed-discrete-result}		
For each compatible initial data $(\ubu_{h,0},p_{h,0}):= ((\bu_{h,0},\bomega_{h,0}),p_{h,0})$ satisfying \eqref{eq:discrete-initial-condition} and $\f\in \W^{1,1}(0,T;\bH^{-1}(\Omega))$, 
there exists a unique solution to \eqref{eq:discrete-weak-KVBF}, $(\ubu_h,p_h) = ((\bu_h,\bomega_h),p_h):[0,T]\to (\bH^{\bu}_h\times \bH^{\bomega}_h)\times \H^p_h$, with $\bu_h\in \W^{1,\infty}(0,T;\bH^{\bu}_h)$
and $(\bu_h(0),\bomega_h(0)) = (\bu_{h,0},\bomega_{h,0})$. 
Moreover, there exists a constant $\wh{C}_{\tt KVr}>0$ depending only on $|\Omega|, \|\bi_\rho\|, \|\bi_4\|, \nu, \tD, \tF$, and $\kappa$, such that
\begin{equation}\label{eq:discrete-stability-bound}
\begin{array}{l}
\|\bu_h\|_{\L^\infty(0,T;\bH^1(\Omega))}
+ \|\bu_h\|_{\L^{2}(0,T;\bL^2(\Omega))} 
+ \|\bomega_h\|_{\L^{2}(0,T;\bL^2(\Omega))} \\[2ex]
\ds\qquad \,\leq\, \wh{C}_{\tt KVr}\,\sqrt{\exp(T)}\,\Big( \|\f\|_{\L^2(0,T;\bH^{-1}(\Omega))}  
+ \|\bu_0\|_{\bH^1(\Omega)} 
+ \|\bu_0\|^2_{\bH^1(\Omega)} 
+ \|\bu_0\|^{\rho-1}_{\bH^1(\Omega)} 
\Big) \,,
\end{array}
\end{equation}	
and a constant $\wh{C}_{\tt KVp}>0$ depending only on $|\Omega|, \|\bi_\rho\|, \|\bi_4\|, \nu, \tD, \tF, \kappa$, and $\beta_\ttd$, such that
\begin{equation}\label{eq:discrete-stability-bound-p}
\begin{array}{l}
\|p_h\|_{\L^{2}(0,T;\L^2(\Omega))} \\[2ex]
\ds\qquad
\leq\, \wh{C}_{\tt KVp}\,
\sum_{j\in \{ 2,3,\rho \}} \left\{ \sqrt{\exp(T)}\,\Big( \|\f\|_{\L^2(0,T;\bH^{-1}(\Omega))}  
+ \|\bu_0\|_{\bH^1(\Omega)} 
+ \|\bu_0\|^2_{\bH^1(\Omega)} 
+ \|\bu_0\|^{\rho-1}_{\bH^1(\Omega)} \Big) \right\}^{j-1} \,.
\end{array}
\end{equation}
\end{thm}
\begin{proof} 
According to Lemma \ref{lem:discrete-properties-A}, the discrete inf-sup condition for $\cB$
provided by \eqref{eq:discrete-inf-sup-condition} in Lemma \ref{lem:discrete-inf-sup}, a fixed-point approach as the one used in Lemma \ref{lem:resolvent-system}, but now with $\bW_{\ttd}$ (cf. \eqref{eq:discrete-ball-W}), 
and considering that $(\ubu_{h,0},p_{h,0})$ satisfies \eqref{eq:discrete-initial-condition},  
the proof of existence and uniqueness
of solution of \eqref{eq:discrete-weak-KVBF-reduced} (equivalently of \eqref{eq:discrete-weak-KVBF}) with $\bu_h\in \W^{1,\infty}(0,T;\bH^{\bu}_h)$ and $\bu_h(0)=\bu_{h,0}$, follows similarly to the proof of Theorem \ref{thm:well-posed-result-reduced} by applying Theorem \ref{thm:well-posed-parabolic-problem}.
Moreover, from the discrete version of \eqref{eq:omega-t-0}, we deduce that $\bomega_h(0) = \bomega_{h,0}$.
	
On the other hand, mimicking the steps followed in the proof of Theorems \ref{thm:well-posed-result-reduced} and \ref{thm:stability-bound}, 
we obtain, respectively, the discrete versions of \eqref{eq:stability-0}--\eqref{eq:stability-3} and \eqref{eq:stability-p-inf-sup}--\eqref{eq:stability-5}.
Then, using the fact that $(\bu_h(0),\bomega_h(0)) = (\bu_{h,0},\bomega_{h,0})$ and \eqref{eq:uh0-wh0-bound-u0-w0}, we derive \eqref{eq:discrete-stability-bound} and \eqref{eq:discrete-stability-bound-p}, thus completing the proof.
\end{proof}

\subsection{Error analysis}

Now we derive suitable error estimates for the semidiscrete scheme \eqref{eq:discrete-weak-KVBF}.
To that end, we first recall that the discrete inf-sup condition of $\cB$ (cf. \eqref{eq:discrete-inf-sup-condition}), 
and a classical result on mixed methods (see, for instance \cite[eq. (2.89) in Theorem~2.6]{Gatica}) 
ensure the existence of a constant $C>0$, independent of $h$, such that:
\begin{equation}\label{eq:infimo-Vh-MhQh}
\inf_{\ubv_h\in \bV_h} \|\ubu - \ubv_h\| 
\,\leq\, C\,\inf_{\ubv_h\in \bH^{\bu}_h\times \bH^{\bomega}_h} \|\ubu - \ubv_h\|\,.
\end{equation}

Next, in order to obtain the theoretical rates of convergence for the discrete scheme 
\eqref{eq:discrete-weak-KVBF}, we recall the approximation properties of 
the finite element subspaces $\bH^{\bu}_h, \bH^{\bomega}_h$, and $\H^{p}_h$ 
(cf. \eqref{eq:Taylor-Hood}) that
can be found in \cite{BBF2013}, \cite{Brezzi-Fortin}, and \cite{Ern-Guermond}.
Assume that $\bu\in \bH^{1+s}(\Omega), \bomega\in [\H^{s}(\Omega)]^{d(d-1)/2}$, and $p\in \H^s(\Omega)$, for some $s\in (1/2,k+1]$.
Then there exists $C>0$, independent of $h$, such that
\begin{align}
\inf_{\bv_h\in \bH^{\bu}_h} \|\bu - \bv_h\|_{\bH^1(\Omega)} 
& \,\leq\, C\,h^{s}\,\|\bu\|_{\bH^{1+s}(\Omega)} \,, \label{eq:approximation-u} \\
\inf_{\bpsi_h\in \bH^{\bomega}_h} \|\bomega - \bpsi_h\|_{\bL^2(\Omega)} 
& \,\leq\, C\,h^{s}\,\|\bomega\|_{\bH^{s}(\Omega)} \,, \label{eq:approximation-omega} \\
\inf_{q_h\in \H^p_h} \|p - q_h\|_{\L^2(\Omega)} 
&\,\leq\, C\,h^{s}\,\|p\|_{\H^{s}(\Omega)} \,. \label{eq:approximation-p}
\end{align}
Owing to \eqref{eq:infimo-Vh-MhQh} and \eqref{eq:approximation-u}--\eqref{eq:approximation-p},
it follows that, under an extra regularity assumption on the exact solution, there exist positive constants
$C(\ubu)$, $C(\partial_t\,\ubu)$,  $C(p)$, and $C(\partial_t\,p)$, depending on
$\bu, \bomega$ and $p$, respectively,
such that
\begin{equation}\label{eq:inf-approximation-properties}
\begin{array}{c}
\ds \inf_{\ubv_h\in \bV_h} \|\ubu - \ubv_h\| \,\leq\, C(\ubu)\,h^{s}\,,\quad
\inf_{\ubv_h\in \bV_h} \|\partial_t\,\ubu - \ubv_h\| \,\leq\, C(\partial_t\,\ubu)\,h^{s} \,, \\[3ex]
\ds \inf_{q_{h}\in \H^p_h} \|p - q_h\|_{\L^2(\Omega)} \,\leq\, C(p)\,h^{s}\,, \qan
\inf_{q_{h}\in \H^p_h} \|\partial_t\,p - q_h\|_{\L^2(\Omega)} \,\leq\, C(\partial_t\,p)\,h^{s} \,.
\end{array}
\end{equation}

In turn, in order to simplify the subsequent analysis, we write
$\be_{\ubu} = (\be_{\bu}, \be_{\bomega}) = (\bu - \bu_h, \bomega - \bomega_h)$ 
and $\be_{p} = p - p_h$.
Next, given arbitrary $\wh{\ubv}_h := (\wh{\bv}_h, \wh{\bpsi}_h):[0,T]\to \bV_h$ 
(cf. \eqref{eq:discrete-kernel-B}) and $\wh{q}_h:[0,T]\to \H^p_h$, 
as usual, we shall decompose the errors into
\begin{equation}\label{eq:error-decomposition}
\be_{\ubu} = \bdelta_{\ubu} + \bbeta_{\ubu}
= (\bdelta_{\bu}, \bdelta_{\bomega}) + (\bbeta_{\bu}, \bbeta_{\bomega})\,,\quad
\be_{p} = \bdelta_{p} + \bbeta_{p}\,,
\end{equation}
with
\begin{equation}\label{eq:delta-eta-definition}
\begin{array}{c}
\ds \bdelta_{\bu} = \bu - \wh{\bv}_h\,,\quad 
\bdelta_{\bomega} = \bomega - \wh{\bpsi}_h\,,\quad 
\bdelta_{p} = p - \wh{q}_h\,, \\ [2ex]
\ds \bbeta_{\bu} = \wh{\bv}_h - \bu_h\,,\quad
\bbeta_{\bomega} = \wh{\bpsi}_h - \bomega_h\,,\quad
\bbeta_{p} = \wh{q}_h - p_h\,.
\end{array}
\end{equation}
In addition, we stress for later use that
for each $\ubv_h :[0,T]\to \bV_h$ (cf. \eqref{eq:discrete-kernel-B}) it holds that
$\partial_t\,\ubv_h(t) \in \bV_h$. In fact, given $(\ubv_h, q_h):[0,T]\to \bV_h\times \H^p_{h}$, after simple 
algebraic computations, we obtain
\begin{equation}\label{eq:dt-vh-discrete-kernel}
[\cB(\partial_t\,\ubv_h),q_h] 
\,=\, \partial_t\big( [\cB(\ubv_h),q_h] \big) - [\cB(\ubv_h),\partial_t\,q_h]
\,=\, 0 \,,
\end{equation}
where, the latter is obtained by observing that $\partial_t\,q_h(t)\in \H^p_{h}$.

Finally, since the exact solution $\bu\in \bH^1_0(\Omega)$ satisfies $\div(\bu)=0$ in $\Omega$, we have
\begin{equation*}
[\cA_h(\bu)(\ubu),\ubv_h] \,=\, [\cA(\bu)(\ubu),\ubv_h] \quad \forall\,\ubv_h\in \bH^{\bu}_h\times \bH^{\bomega}_h \,.
\end{equation*}
In this way, by subtracting the discrete and continuous problems \eqref{eq:KVBF-variational-formulation}
and \eqref{eq:discrete-weak-KVBF}, respectively, we obtain the following error system:
\begin{equation}\label{eq:error-system}
\begin{array}{llll}
\ds \frac{\partial}{\partial\,t}\,[\cE(\be_{\ubu}),\ubv_h] + [\cA_h(\bu)(\ubu) - \cA_h(\bu_h)(\ubu_h),\ubv_h] + [\cB(\ubv_h),\be_{p}] & = & 0 & \forall\,\ubv_h\in \bH^{\bu}_h\times \bH^{\bomega}_h\,, \\ [2ex]
\ds [\cB(\be_{\ubu}),q_h] & = & 0 & \forall\,q_h\in \H^p_h\,.
\end{array}
\end{equation}
We now establish the main result of this section, namely, the theoretical rate of convergence of the semidiscrete scheme \eqref{eq:discrete-weak-KVBF}.
Note that optimal rates of convergences are obtained for all the unknowns.
\begin{thm}\label{thm:rate-of-convergence}
Let $((\bu,\bomega),p):[0,T]\to \big( \bH^1_0(\Omega)\times \bL^2(\Omega) \big)\times \L^2_0(\Omega)$ 
with $\bu\in \W^{1,\infty}(0,T;\bH^{-1}(\Omega))$ and 
$((\bu_h,\bomega_h),p_h):[0,T]\to \big( \bH^{\bu}_h\times \bH^{\bomega}_h \big)\times \H^p_h$ 
with $\bu_h\in \W^{1,\infty}(0,T;\bH^{\bu}_h)$ 
be the unique solutions of the continuous and semidiscrete problems 
\eqref{eq:KVBF-variational-formulation} and \eqref{eq:discrete-weak-KVBF}, respectively.
Assume further that there exists $s\in (1/2,k+1]$, such that 
$\bu\in \bH^{1+s}(\Omega),\,\bomega\in [\H^{s}(\Omega)]^{d(d-1)/2}$, and $p\in \H^{s}(\Omega)$.
Then, there exists $C(\ubu,p) > 0$ depending only on
$C(\ubu), C(\partial_t\,\ubu), C(p)$, $C(\partial_t\,p), \|\bi_\rho\|,  \|\bi_4\|, |\Omega|, \nu, \tD, \tF, \kappa, \beta_\ttd, T, \|\f\|_{\L^2(0,T;\bH^{-1}(\Omega))} $, and $\|\bu_0\|_{\bH^1(\Omega)}$, such that
\begin{align}\label{eq:rate-estimate-u-omega-p}
\|\be_{\bu}\|_{\L^\infty(0,T;\bH^1(\Omega))} 
& + \|\be_{\bu}\|_{\L^{2}(0,T;\bL^2(\Omega))}
+ \|\be_{\bomega}\|_{\L^2(0,T;\bL^2(\Omega))} \nonumber \\[2ex]
& + \|\be_{p}\|_{\L^2(0,T;\L^2(\Omega))}
\,\leq\, C(\ubu,p)\,\left( h^s + h^{s\,(\rho-1)} \right) \,.
\end{align}
\end{thm}
\begin{proof}
First, adding and subtracting suitable terms in the first equation of \eqref{eq:error-system}, 
with $\ubv_h = \bbeta_{\ubu} =(\bbeta_{\bu},\bbeta_{\bomega}):[0,T]\to \bV_h$ (cf. \eqref{eq:discrete-kernel-B}), 
proceeding as in \eqref{eq:strong-monotonicity-of-EAz}, and using the fact that 
$\bbeta_{\ubu}(t)\in \bV_h$, thus $[\cB(\bbeta_{\ubu}),\bbeta_{p}] = 0$, we deduce that
\begin{equation}\label{eq:first-error-estimate}
\begin{array}{l}
\ds \frac{1}{2}\,\partial_t\,\left( \|\bbeta_{\bu}\|^2_{\bL^2(\Omega)} + \kappa^2\|\nabla\bbeta_{\bu}\|^2_{\bbL^2(\Omega)} \right)
+ \tD\,\|\bbeta_{\bu}\|^2_{\bL^2(\Omega)} 
+ \tF\,C_{\rho}\,\|\bbeta_{\bu}\|^{\rho}_{\bL^\rho(\Omega)}  
+ \nu\,\|\bbeta_\bomega\|^2_{\bL^2(\Omega)} \\[2ex]
\ds\quad \leq\, - (\partial_t \bdelta_{\bu},\bbeta_{\bu})_{\Omega} - \kappa^2\,(\partial_t \nabla \bdelta_{\bu},\nabla\bbeta_{\bu})_{\Omega}
- [\cA_h(\bu)(\ubu) - \cA_h(\bu)(\wh{\ubv}_h),\bbeta_{\ubu}]  \\[2ex]
\ds\quad  -\, [\bc(\bu - \bu_h)(\ubu_h),\ubv_h] + (\bdelta_p,\div(\bbeta_\bu))_\Omega \,.
\end{array}
\end{equation}
The terms on the right hand side can be bounded using the Cauchy--Schwarz, H\"older and Young's inequalities (cf. \eqref{eq:Young-inequality}), \eqref{eq:continuity-of-Az}, as follows:
\begin{align}
&  - (\partial_t \bdelta_{\bu},\bbeta_{\bu})_{\Omega} - \kappa^2\,(\partial_t \nabla \bdelta_{\bu},\nabla\bbeta_{\bu})_{\Omega} \le \frac{\max\{1,\kappa^2\}}{2}\left( \|\partial_t \bdelta_{\bu}\|^2_{\bH^1(\Omega)} + \|\bbeta_{\bu}\|^2_{\bH^1(\Omega)} \right), \label{eta-u} \\
& - [\cA_h(\bu)(\ubu) - \cA_h(\bu)(\wh{\ubv}_h),\bbeta_{\ubu}] \nonumber \\
& \quad \le \wt{C}_1\,\left\{ \big( 1 + \|\bu\|_{\bH^1(\Omega)} + \|\bu\|^{\rho-2}_{\bH^1(\Omega)} - \|\wh{\bv}_h\|^{\rho-2}_{\bH^1(\Omega)} \big) \|\bdelta_\bu\|_{\bH^1(\Omega)} + \|\bdelta_\bomega\|_{\L^2(\Omega)} \right\}\big( \|\bbeta_\bu\|_{\bH^1(\Omega)} + \|\bbeta_\bomega\|_{\bL^2(\Omega)} \big) \nonumber \\
& \quad \leq \wh{C}_1\,\left\{ \big( 1 + \|\bu\|_{\bH^1(\Omega)} + \|\bu\|^{\rho-2}_{\bH^1(\Omega)} + \|\bdelta_\bu\|^{\rho-2}_{\bH^1(\Omega)} \big) \|\bdelta_\bu\|_{\bH^1(\Omega)} + \|\bdelta_\bomega\|_{\L^2(\Omega)} \right\}\big( \|\bbeta_\bu\|_{\bH^1(\Omega)} + \|\bbeta_\bomega\|_{\bL^2(\Omega)} \big) \nonumber \\
& \quad \leq C_1\,\left\{ \big( 1 + \|\bu\|^2_{\bH^1(\Omega)} + \|\bu\|^{2\,(\rho-2)}_{\bH^1(\Omega)} \big)\|\bdelta_\bu\|^2_{\bH^1(\Omega)} + \|\bdelta_\bu\|^{2(\rho-1)}_{\bH^1(\Omega)} + \|\bdelta_\bomega\|^2_{\L^2(\Omega)} \right\} \nonumber \\
& \qquad + \big( \|\bbeta_\bu\|^2_{\bH^1(\Omega)} + \frac{\nu}{2}\|\bbeta_\bomega\|^2_{\bL^2(\Omega)} \big), 
\label{eta-u-Lp}\\
& -[\bc(\bu - \bu_h)(\ubu_h),\bbeta_\bu] \le
\|\bi_4\|^2\,\|\bu_h\|_{\bH^1(\Omega)}\big( \|\bdelta_\bu\|_{\bH^1(\Omega)} + \|\bbeta_\bu\|_{\bH^1(\Omega)} \big)\|\bbeta_\bu\|_{\bH^1(\Omega)}, \nonumber  \\
& \quad \leq C_2\,\|\bu_h\|_{\bH^1(\Omega)}\big( \|\bdelta_\bu\|^2_{\bH^1(\Omega)} + \|\bbeta_\bu\|^2_{\bH^1(\Omega)} \big),
\label{eta-w} \\
& (\bdelta_p,\div(\bbeta_\bu))_\Omega \le \frac{\sqrt{d}}{2}\big( \|\bdelta_p\|^2_{\L^2(\Omega)} + \|\bbeta_\bu\|^2_{\bH^1(\Omega)} \big), \label{div-eta-u}
\end{align}
where $C_1, C_2 > 0$ depend on $\|\bi_4\|, \|\bi_\rho\|, \kappa, \tD, \tF$, and $\nu$. We note that in \eqref{eta-u-Lp} we used the continuous injection of $\bH^1(\Omega)$ into $\bL^{\rho}(\Omega)$, with $\rho\in [3,4]$, cf. \eqref{eq:Sobolev-inequality}. Combining \eqref{eq:first-error-estimate}--\eqref{div-eta-u}, and neglecting the term $\|\bbeta_{\bu}\|^{\rho}_{\bL^\rho(\Omega)}$ in \eqref{eq:first-error-estimate} to simplify the error estimate, we obtain
\begin{equation}\label{eq:error-estimate-2}
\begin{array}{l}
\ds \partial_t\,\left( \|\bbeta_{\bu}\|^2_{\bL^2(\Omega)} + \kappa^2\|\nabla\bbeta_{\bu}\|^2_{\bbL^2(\Omega)} \right)
+ \tD\,\|\bbeta_{\bu}\|^2_{\bL^2(\Omega)} 
+ \nu\,\|\bbeta_\bomega\|^2_{\bL^2(\Omega)} \\[2ex]
\ds\quad\,\leq\, C_3\,\Big( \|\partial_t\,\bdelta_{\bu}\|^2_{\bH^1(\Omega)} 
+ \big( 1 + \|\bu_h\|_{\bH^1(\Omega)} + \|\bu\|^2_{\bH^1(\Omega)} + \|\bu\|^{2\,(\rho-2)}_{\bH^1(\Omega)} \big)\|\bdelta_{\bu}\|^{2}_{\bH^1(\Omega)} \\[2ex] 
\ds\quad \,+\, \|\bdelta_{\bu}\|^{2(\rho-1)}_{\bH^1(\Omega)}
+ \|\bdelta_{\bomega}\|^2_{\bL^2(\Omega)} 
+ \|\bdelta_{p}\|^2_{\L^2(\Omega)} 
+ \big( 1 + \|\bu_h\|_{\bH^1(\Omega)}\big)\|\bbeta_{\bu}\|^2_{\bH^1(\Omega)} \Big),
\end{array}
\end{equation}
with $C_3$ a positive constant depending on $|\Omega|, \|\bi_{4}\|, \|\bi_{\rho}\|, \nu, \tD, \tF$, and $\kappa$.
Integrating \eqref{eq:error-estimate-2} from $0$ to $t\in (0,T]$, recalling that $\|\bu\|_{\L^\infty(0,T;\bH^1(\Omega))}$ and $\|\bu_h\|_{\L^\infty(0,T;\bH^1(\Omega))}$ are bounded by data (cf. \eqref{eq:stability-bound-u-om-reduced}, \eqref{eq:discrete-stability-bound}), 
we find that
\begin{equation}\label{eq:error-estimate-3}
\begin{array}{l}
\ds \|\bbeta_{\bu}(t)\|^2_{\bH^1(\Omega)} 
+ \int^t_0 \Big( \|\bbeta_{\bu}\|^2_{\bL^2(\Omega)} 
+ \|\bbeta_\bomega\|^2_{\bL^2(\Omega)} \Big)\,ds 
\leq\, C_4\,\left\{ \int^t_0 \Big(  \|\partial_t\,\bdelta_{\bu}\|^2_{\bH^1(\Omega)} 
+ \|\bdelta_{\ubu}\|^2 \Big)\, ds \right. \\[3ex]
\ds\quad +\, \left. \int^t_0 \Big( \|\bdelta_{\bu}\|^{2(\rho-1)}_{\bH^1(\Omega)}  
+ \|\bdelta_{p}\|^2_{\L^2(\Omega)} \Big)\, ds \right\}
\,+\, \wh{C}_4\left\{ \int^t_0 \|\bbeta_{\bu}\|^2_{\bH^1(\Omega)}\, ds 
+ \|\bbeta_{\bu}(0)\|^2_{\bH^1(\Omega)} \right\},
\end{array}
\end{equation}
with $C_5, \wh{C}_4>0$ depending on $|\Omega|, \|\bi_{4}\|, \|\bi_{\rho}\|, \nu, \tD, \tF, \kappa$, and data.

On the other hand, to estimate $\|\be_{p}\|_{\L^2(0,T;\L^2(\Omega))}$, we observe that from 
the discrete inf-sup condition of $\cB$ (cf. \eqref{eq:discrete-inf-sup-condition}), 
the first equation of \eqref{eq:error-system}, and the continuity bounds of $\cB, \cE, \cA_h$ 
(cf. \eqref{eq:continuity-B}, \eqref{eq:continuity-monotonicity-E}, \eqref{eq:continuity-of-Az}), there holds
\begin{equation*}
\begin{array}{l}
\ds \beta_\ttd\,\|\bbeta_{p}\|_{\L^2(\Omega)} 
\,\leq\,
\sup_{\0\neq\ubv_h\in \bH^{\bu}_h\times \bH^{\bomega}_h} 
\frac{-\big( [\partial_t\,\cE(\be_{\ubu}),\ubv_h] + [\cA_h(\bu)(\ubu) - \cA_h(\bu_h)(\ubu_h),\ubv_h] + [\cB(\ubv_h),\bdelta_{p}] \big)}{\|\ubv_h\|} \\[4ex]
\ds\,\, \leq C\,\Big( \|\partial_t \be_{\bu}\|_{\bH^1(\Omega)} 
+ \big( 1 + \|\bu\|_{\bH^1(\Omega)} + \|\bu\|^{\rho-2}_{\bH^1(\Omega)} + \|\bu_h\|^{\rho-2}_{\bH^1(\Omega)} \big) \|\be_{\bu}\|_{\bH^1(\Omega)} 
+ \|\be_{\bomega}\|_{\bL^2(\Omega)} 
+ \|\bdelta_{p}\|_{\L^2(\Omega)} \Big)\,,
\end{array}
\end{equation*}
with $C>0$ depending on $|\Omega|, \|\bi_4\|, \|\bi_\rho\|, \nu, \tD, \tF$, and $\kappa$.
Then, taking square in the above inequality, integrating from $0$ to $t\in (0,T]$, 
using again the fact that $\|\bu\|_{\L^\infty(0,T;\bH^1(\Omega))}$ and $\|\bu_h\|_{\L^\infty(0,T;\bH^1(\Omega))}$ are bounded by data (cf. \eqref{eq:stability-bound-u-om-reduced}, \eqref{eq:discrete-stability-bound})
and employing \eqref{eq:error-estimate-3}, we deduce that
\begin{equation}\label{eq:error-estimate-p-1}
\begin{array}{l}
\ds \int^t_0 \|\bbeta_{p}\|^2_{\L^2(\Omega)} \,ds 
\,\leq\, 
C_5\,\int^t_0 \Big( \|\partial_t\,\bdelta_{\bu}\|^2_{\bH^1(\Omega)} 
+ \|\bdelta_{\ubu}\|^2 + \|\bdelta_{\bu}\|^{2(\rho-1)}_{\bH^1(\Omega)}  
+ \|\bdelta_{p}\|^2_{\L^2(\Omega)} \Big)\, ds \\[3ex]
\ds\quad +\, \wh{C}_5\left\{ \int^t_0 \Big( \|\bbeta_{\bu}\|^2_{\bH^1(\Omega)} + \|\partial_t\,\bbeta_{\bu}\|^2_{\bH^1(\Omega)} \Big)\,ds 
+ \|\bbeta_{\bu}(0)\|^2_{\bH^1(\Omega)} \right\} \,,
\end{array}
\end{equation}
with $C_5, \wh{C}_5>0$ depending on $|\Omega|, \|\bi_{4}\|, \|\bi_\rho\|, \nu, \tD, \tF, \beta_\ttd, \kappa$, and data.
Next, in order to bound the term $\|\partial_t\,\bbeta_{\bu}\|_{\bH^1(\Omega)}$ in \eqref{eq:error-estimate-p-1}, we differentiate in time
the equation of \eqref{eq:error-system} related to $\bpsi_h$ and choose 
$\ubv_h = (\partial_t\,\bbeta_{\bu},\bbeta_{\bomega})$ 
to find that
\begin{align}\label{eq:bound-dt-etau-1}
& \ds \min\{1,\kappa^2\}\|\partial_t\,\bbeta_{\bu}\|^2_{\bH^1(\Omega)} + \frac{1}{2}\,\partial_t\Big( \tD\,\|\bbeta_{\bu}\|^2_{\bL^2(\Omega)} 
+ \nu\,\|\bbeta_{\bomega}\|^2_{\bL^2(\Omega)} \Big) \nonumber  \\[2ex]
&\ds\quad =\, -(\partial_t\,\bdelta_{\bu},\partial_t\,\bbeta_{\bu})_{\Omega}
- \kappa^2\,(\partial_t\,\nabla\bdelta_{\bu},\partial_t\,\nabla\bbeta_{\bu})_{\Omega}
- \tD\,(\bdelta_{\bu},\partial_t\,\bbeta_{\bu})_{\Omega}
- \nu (\partial_t\,\bdelta_\bomega,\bbeta_\bomega)_{\Omega} \nonumber \\[2ex]
&\ds\quad \,\,-\, 
\nu (\bdelta_\bomega,\bcurl(\partial_t\,\bbeta_\bu))_{\Omega}
+ \nu (\bbeta_\bomega,\bcurl(\partial_t \bdelta_\bu))_{\Omega}
+ (\bdelta_p,\div(\partial_t\bbeta_\bu))_{\Omega} \\[2ex]
&\ds\quad\,\,-\, 
\tF\,(|\bu|^{\rho-2}\bu - |\bu_h|^{\rho-2}\bu_h,\partial_t\,\bbeta_{\bu})_{\Omega}
- ((\nabla\bu)\bu - (\nabla\bu_h)\bu_h, \partial_t \bbeta_\bu)_\Omega
\,. \nonumber
\end{align}
Notice that $(\bbeta_p,\div(\partial_t\bbeta_\bu))_{\Omega} = 0$ since $(\bbeta_{\bu}(t),\0)\in \bV_h$ 
(cf. \eqref{eq:discrete-kernel-B} and \eqref{eq:dt-vh-discrete-kernel}).
In turn, using the H\"older inequality, the estimate \eqref{eq:continuity-bound-2} and the continuous injection of $\bH^1(\Omega)$ into $\bL^{\rho}(\Omega)$ we deduce that there exists a constant $c_\rho>0$ depending on $|\Omega|$ and $\rho$ such that
\begin{align}\label{eq:Forchheimer-error-bound}
(|\bu|^{\rho-2}\bu - |\bu_h|^{\rho-2}\bu_h,\partial_t\,\bbeta_{\bu})_{\Omega}
&\,\leq\, c_\rho\,\big( \|\bu\|_{\bL^{\rho}(\Omega)} + \|\bu_h\|_{\bL^{\rho}(\Omega)} \big)^{\rho-2} \|\be_\bu\|_{\bL^{\rho-1}(\Omega)} \|\partial_t\,\bbeta_{\bu}\|_{\bL^\rho(\Omega)} \nonumber \\[2ex]
&\,\leq\, c_\rho\,\|\bi_\rho\|^\rho\,\big( \|\bu\|_{\bH^1(\Omega)} + \|\bu_h\|_{\bH^1(\Omega)} \big)^{\rho-2} \|\be_\bu\|_{\bH^1(\Omega)} \|\partial_t\,\bbeta_{\bu}\|_{\bH^1(\Omega)} \,.
\end{align}
Similarly, but now adding and subtracting the term $(\nabla\bu)\bu_h$ (it also works with $(\nabla\bu_h)\bu$), using the continuous injection of $\bH^1(\Omega)$ into $\bL^4(\Omega)$, we obtain
\begin{align}\label{eq:Convective-error-bound}
((\nabla\bu)\bu - (\nabla\bu_h)\bu_h, \partial_t \bbeta_\bu)_\Omega
&\,\leq\, \Big(\|\nabla\bu\|_{\bL^2(\Omega)}\|\be_\bu\|_{\bL^4(\Omega)} + \|\bu_h\|_{\bL^4(\Omega)}\|\nabla\be_\bu\|_{\bbL^2(\Omega)} \Big) \|\partial_t \bbeta_\bu\|_{\bL^4(\Omega)}  \nonumber \\[2ex]
&\,\leq\, \|\bi_4\|^2\,\big( \|\bu\|_{\bH^1(\Omega)} + \|\bu_h\|_{\bH^1(\Omega)} \big) \|\be_\bu\|_{\bH^1(\Omega)} \|\partial_t\,\bbeta_{\bu}\|_{\bH^1(\Omega)} \,.
\end{align}
Thus, integrating \eqref{eq:bound-dt-etau-1} from $0$ to $t\in (0,T]$, using 
the estimates \eqref{eq:Forchheimer-error-bound} and \eqref{eq:Convective-error-bound}, 
the Cauchy--Schwarz and Young's inequalities, and the fact that $\|\bu\|_{\L^\infty(0,T;\bH^1(\Omega))}$ and $\|\bu_h\|_{\L^\infty(0,T;\bH^1(\Omega))}$ are bounded by data (cf. \eqref{eq:stability-bound-u-om-reduced}, \eqref{eq:discrete-stability-bound}), in a way similar to \eqref{eta-u}--\eqref{div-eta-u}, we find that
\begin{equation}\label{eq:error-estimate-dt-etau}
\begin{array}{l}
\ds \|\bbeta_{\bu}(t)\|^2_{\bL^2(\Omega)} 
+ \|\bbeta_{\bomega}(t)\|^2_{\bL^2(\Omega)} 
+ \int^t_0 \|\partial_t\,\bbeta_{\bu}\|^2_{\bH^1(\Omega)}\,ds \\[3ex]
\ds \leq\, C_6\,\bigg( \int^t_0
\Big( \|\partial_t\,\bdelta_{\ubu}\|^2
+ \|\bdelta_\bu\|^2_{\bH^1(\Omega)}
+ \|\bdelta_\bomega\|^2_{\L^2(\Omega)}
+ \|\bdelta_p\|^2_{\L^2(\Omega)}
\Big) ds \\[3ex]
\ds\quad +\, \int^t_0 \Big( \|\bbeta_\bu\|^2_{\bH^1(\Omega)} + \|\bbeta_{\bomega}\|^2_{\bL^2(\Omega)} \Big)\,ds 
+ \|\bbeta_{\bu}(0)\|^2_{\bL^2(\Omega)} + \|\bbeta_{\bomega}(0)\|^2_{\bL^2(\Omega)} \bigg) \,,
\end{array}
\end{equation}
where $C_6 > 0$ depends on $|\Omega|, \|\bi_{4}\|, \|\bi_{\rho}\|, \nu, \tD, \tF, \kappa$, and data. 
Then, combining estimates \eqref{eq:error-estimate-3}, \eqref{eq:error-estimate-p-1} and \eqref{eq:error-estimate-dt-etau}, using the Gr\"onwall inequality, and some algebraic manipulations, we deduce that
\begin{equation}\label{eq:full-error-estimate-preliminar}
\begin{array}{l}
\ds \|\bbeta_{\bu}(t)\|^2_{\bH^1(\Omega)} 
+ \|\bbeta_{\bomega}(t)\|^2_{\bL^2(\Omega)} 
+ \int^t_0 \Big( \|\bbeta_{\bu}\|^2_{\bL^2(\Omega)} 
+ \|\bbeta_\bomega\|^2_{\bL^2(\Omega)}
+ \|\bbeta_{p}\|^2_{\L^2(\Omega)}
+ \|\partial_t\,\bbeta_{\bu}\|^2_{\bH^1(\Omega)} \Big)\,ds \\[3ex]
\ds\quad \leq\, C_7\,\exp(T)\bigg( \int^t_0
\Big( \|\partial_t\,\bdelta_{\ubu}\|^2
+ \|\bdelta_{\ubu}\|^2
+ \|\bdelta_{\bu}\|^{2(\rho-1)}_{\bH^1(\Omega)} 
+ \|\bdelta_{p}\|^2_{\L^2(\Omega)} \Big)\, ds 
\,+\, \|\bbeta_{\ubu}(0)\|^2 \bigg) \,,
\end{array}
\end{equation}
with $C_7>0$ depending on $|\Omega|, \|\bi_4\|, \|\bi_{\rho}\|, \nu, \tD, \tF, \beta_\ttd, \kappa$, and data.

Finally, in order to bound the last term in \eqref{eq:full-error-estimate-preliminar}, we subtract 
the continuous and discrete initial condition problems \eqref{eq:initial-condition-full-problem} 
and \eqref{eq:discrete-initial-condition-full-problem} to obtain the error system:
\begin{equation*}
\begin{array}{lll}
\ds (\nabla\bu_0 - \nabla\bu_{h,0},\nabla\bv_{h})_\Omega + [\cA_h(\bu_0)(\ubu_0) - \cA_h(\bu_{h,0})(\ubu_{h,0}),\ubv_h] + [\cB(\ubv_h),p_0 - p_{h,0}] & = & 0 \,, \\ [2ex]
\ds -\,[\cB(\ubu_0 - \ubu_{h,0}),q_h] & = & 0 \,,
\end{array}
\end{equation*} 
for all $\ubv_h\in \bH^{\bu}_h\times \bH^{\bomega}_h$ and $q_h\in \H^p_h$.
Then, proceeding as in \eqref{eq:error-estimate-2}, recalling from 
Theorems \ref{thm:well-posed-result-reduced} and \ref{thm:well-posed-discrete-result} 
that $(\bu(0),\bomega(0)) = (\bu_0,\bomega_0)$ and $(\bu_h(0),\bomega_h(0)) = (\bu_{h,0},\bomega_{h,0})$, 
respectively, we get
\begin{equation}\label{eq:error-estimate-sol0}
\|\bbeta_{\bu}(0)\|^2_{\bH^1(\Omega)} 
+ \|\bbeta_\bomega(0)\|^2_{\bL^2(\Omega)}
\,\leq\, \wt{C}_0\,\Big(  \|\bdelta_{\ubu_0}\|^2
+ \|\bdelta_{\bu_0}\|^{2\,(\rho-1)}_{\bH^1(\Omega)} 
+ \|\bdelta_{p_0}\|^2_{\L^2(\Omega)} \Big)\,,
\end{equation}
where, similarly to \eqref{eq:delta-eta-definition}, we denote 
$\bdelta_{\ubu_0}=(\bdelta_{\bu_0},\bdelta_{\bomega_0}) = (\bu_0 - \wh{\bv}_{h}(0), \bomega_0 - \wh{\bpsi}_{h}(0))$ 
and $\bdelta_{p_0} = p_0 - \wh{q}_{h}(0)$, with arbitrary 
$(\wh{\bv}_{h}(0),\wh{\bpsi}_{h}(0))\in \bV_h$ and $\wh{q}_{h}(0)\in \H^p_{h}$, 
and $\wt{C}_0$ is a positive constant depending on $|\Omega|, \|\bi_4\|, \|\bi_{\rho}\|, \nu, \tD, \tF$, and $\kappa$.

Thus, combining \eqref{eq:full-error-estimate-preliminar} with \eqref{eq:error-estimate-sol0}, 
and using the error decomposition \eqref{eq:error-decomposition}, there holds
\begin{equation}\label{eq:error-estimate-4}
\begin{array}{l}
\ds \|\be_{\bu}(t)\|^2_{\bH^1(\Omega)}
+ \int^t_0 \Big( \|\be_{\bu}\|^2_{\bL^2(\Omega)} 
+ \|\be_\bomega\|^2_{\bL^2(\Omega)} 
+ \|\be_{p}\|^2_{\L^2(\Omega)} \Big)\,ds
\,\leq\, C_8\,\exp(T)\,\Psi(\ubu,p)\,,
\end{array}
\end{equation}
where
\begin{equation*}
\begin{array}{l}
\ds \Psi(\ubu,p) \,:=\, \|\bdelta_{\ubu}(t)\|^2 
+ \int^t_0 \Big( \|\partial_t\,\bdelta_{\ubu}\|^2
+ \|\bdelta_{\ubu}\|^2  
+ \|\bdelta_{\ubu}\|^{2\,(\rho-1)}
+ \|\bdelta_{p}\|^2_{\L^2(\Omega)} \Big)\,ds \\ [3ex]
\ds\quad\,+\, \|\bdelta_{\ubu_0}\|^2 
+ \|\bdelta_{\ubu_0}\|^{2\,(\rho-1)} 
+ \|\bdelta_{p_0}\|^2_{\L^2(\Omega)} \,,
\end{array}
\end{equation*}
with $C_8>0$ depending on $|\Omega|, \|\bi_4\|, \|\bi_{\rho}\|, \nu, \tD, \tF, \beta_\ttd$, $\kappa$, and data.
Finally, using the fact that $\wh{\ubv}_h:[0,T]\to \bV_h$ and $\wh{q}_h:[0,T]\to \H^p_h$ 
are arbitrary, taking infimum in \eqref{eq:error-estimate-4} 
over the corresponding discrete subspaces $\bV_h$ and $\H^p_h$, 
and applying the approximation properties \eqref{eq:inf-approximation-properties}, 
we derive \eqref{eq:rate-estimate-u-omega-p} and conclude the proof.
\end{proof}

\begin{rem}
Observe that \eqref{eq:rate-estimate-u-omega-p} can be expanded to include a bound on $\|\partial_t\,\be_\bu\|_{\L^2(0,T;\bH^1(\Omega))}$ and
$\|\be_\bomega\|_{\L^{\infty}(0,T;\bL^2(\Omega))}$, using \eqref{eq:full-error-estimate-preliminar}. 
\end{rem}


\section{Fully discrete approximation}\label{sec:fully-discrete-approximation}

In this section we introduce and analyze a fully discrete approximation of \eqref{eq:KVBF-variational-formulation} (cf. \eqref{eq:discrete-weak-KVBF}). 
To that end, for the time discretization we employ the backward Euler method. 
Let $\Delta t$ be the time step, $T = N\Delta t$, and let $t_n = n\Delta t$, $n = 0, ..., N$. 
Let $d_t u^n = (\Delta t)^{-1}(u^n - u^{n-1})$ be the
first order (backward) discrete time derivative, where $u^n := u(t_n)$. Then the fully discrete method
reads: given $\f^n\in \bH^{-1}(\Omega)$ and $(\ubu_h^0,p_h^0) = ((\bu_{h,0},\bomega_{h,0}),p_{h,0})$ satisfying \eqref{eq:discrete-initial-condition}, find $(\ubu_h^n,p_h^n):=((\bu_h^n,\bomega_h^n),p_h^n)\in (\bH^{\bu}_h\times\bH^{\bomega}_h)\times \H^p_h$, $n = 1, ..., N$, such that
\begin{equation}\label{eq:fully-discrete-weak-Brinkman-Forchheimer}
\begin{array}{llll}
d_t[\cE(\ubu_h^n),\ubv_h] +  [\cA_h(\bu_h^n)(\ubu_h^n),\ubv_h] + [\cB(\ubv_h), p_h^n] & = & [\bF^n,\ubv_h] 
& \forall\, \ubv_h\in \bH^{\bu}_h\times\bH^{\bomega}_h \,, \\[2ex] 
-[\cB(\ubu_h^n), q_h] & = & 0 & \forall\, q_h\in \H^{p}_h \,,
\end{array}
\end{equation}
where $[\bF^n,\ubv_h] := (\f^n,\bv_h)_\Omega$.

In what follows, given a separable Banach space $\V$ endowed with the norm $\| \cdot \|_{\V}$, 
we make use of the following discrete in time norms
\begin{equation}\label{eq:fully-discrete-norm}
\ds \|u\|^2_{\ell^2(0,T;\V)}:=\Delta t\,\sum_{n=1}^N \|u^n\|_\V^2 \,\qan\, \|u\|_{\ell^\infty(0,T;\V)}:= \max_{0\leq n\leq N} \|u^n\|_\V \,.
\end{equation}
We also recall the well-known identity:
\begin{equation}\label{eq:identity-uhn}
(d_t\,u_h^n, u_h^n)_{\Omega}
= \dfrac{1}{2}\,d_t\,\|u_h^n\|^2_{\L^2(\Omega)}
+ \dfrac{1}{2}\,\Delta t\,\|d_t\,u_h^n\|^2_{\L^2(\Omega)} \,,
\end{equation}
and the following discrete Gr\"onwall inequality \cite[Lemma 1.4.2]{Quarteroni-Valli}.
\begin{lem}\label{lem:discrete-Gronwall}
Let $\Delta\,t>0, B\geq 0$, and let $a_n, b_n, c_n, d_n, n\geq 0$, be non-negative sequence such that $a_0\leq B$ and
\begin{equation*}
a_n + \Delta\,t \sum^n_{l=1} b_l \,\leq\,
\Delta\,t \sum^{n-1}_{l=1} d_l\,a_l + \Delta\,t \sum^n_{l=1} c_l + B \,,\quad n\geq 1\,.
\end{equation*}
Then,
\begin{equation*}
a_n + \Delta\,t \sum^n_{l=1} b_l \,\leq\,
\exp\left(\Delta\,t \sum^{n-1}_{l=1} d_l\right)\left(\Delta\,t \sum^n_{l=1} c_l + B \right) \,,\quad n\geq 1\,.
\end{equation*}
\end{lem}
Next, we state the main results for method \eqref{eq:fully-discrete-weak-Brinkman-Forchheimer}.

\begin{thm}\label{thm:well-posed-result-fully-discrete-problem}
For each $(\ubu_h^0,p_h^0) := ((\bu_{h,0},\bomega_{h,0}),p_{h,0})$ satisfying \eqref{eq:discrete-initial-condition-full-problem} and  
$\f^n\in \bH^{-1}(\Omega)$, $n=1,...,N$,  
there exist a unique solution $(\ubu_h^n,p_h^n) := ((\bu_h^n,\bomega_h^n),p_h^n)\in (\bH^{\bu}_h\times\bH^{\bomega}_h)\times \H^p_h$ to \eqref{eq:fully-discrete-weak-Brinkman-Forchheimer}, with $n=1,...,N$. 
Moreover, there exists a constant $\wt{C}_{\tt KVr}>0$ depending only on $|\Omega|, \|\bi_\rho\|, \|\bi_4\|, \nu, \tD, \tF$, and $\kappa$, such that
\begin{equation}\label{eq:fully-discrete-stability-bound}
\begin{array}{l}
\|\bu_h\|_{\ell^\infty(0,T;\bH^1(\Omega))}
+ \Delta t \|d_t \bu_h\|_{\ell^{2}(0,T;\bH^1(\Omega))}
+ \| \bu_h\|_{\ell^2(0,T;\bL^2(\Omega))} 
+ \|\bomega_h\|_{\ell^{2}(0,T;\bL^2(\Omega))} \\[2ex]
\ds\qquad \,\leq\, \wt{C}_{\tt KVr}\,\sqrt{\exp(T)}\,\Big( \|\f\|_{\ell^2(0,T;\bH^{-1}(\Omega))} 
+ \|\bu_0\|_{\bH^1(\Omega)} 
+ \|\bu_0\|^2_{\bH^1(\Omega)} 
+ \|\bu_0\|^{\rho-1}_{\bH^1(\Omega)} 
\Big) \,,
\end{array}
\end{equation}	
and a constant $\wt{C}_{\tt KVp}>0$ depending only on $|\Omega|, \|\bi_\rho\|, \|\bi_4\|, \nu, \tD, \tF, \kappa$, and $\beta_\ttd$, such that
\begin{equation}\label{eq:fully-discrete-stability-bound-p}
\begin{array}{l}
\|p_h\|_{\ell^{2}(0,T;\L^2(\Omega))} \\[2ex]
\ds\qquad
\leq\, \wt{C}_{\tt KVp}\,
\sum_{j\in \{ 2,3,\rho \}} \left\{ \sqrt{\exp(T)}\,\Big( \|\f\|_{\ell^2(0,T;\bH^{-1}(\Omega))} 
+ \|\bu_0\|_{\bH^1(\Omega)} 
+ \|\bu_0\|^2_{\bH^1(\Omega)} 
+ \|\bu_0\|^{\rho-1}_{\bH^1(\Omega)} \Big) \right\}^{j-1} \,.
\end{array}
\end{equation}
\end{thm}
\begin{proof}
Existence of a solution of the fully discrete problem \eqref{eq:fully-discrete-weak-Brinkman-Forchheimer} at each time step $t_n$, $n=1,...,N$, can be established by induction. In particular assuming that a solution exists at $t_{n-1}$, existence of a solution at $t_n$ follows from similar arguments to the proof of Lemma~\ref{lem:resolvent-system}, using the discrete inf-sup condition \eqref{eq:discrete-inf-sup-condition}. We postpone the proof of uniqueness until after the stability bound.

The derivation of \eqref{eq:fully-discrete-stability-bound} and \eqref{eq:fully-discrete-stability-bound-p} can be obtained similarly as 
in the proof of Theorems \ref{thm:well-posed-result-reduced} and \ref{thm:stability-bound}, respectively.
In fact, we choose 
$(\ubv_h,q_h) = (\ubu_h^n,p_h^n)$ in \eqref{eq:fully-discrete-weak-Brinkman-Forchheimer}, 
use the identity \eqref{eq:identity-uhn}, 
the definition of the operator $\cA_h$ (cf. \eqref{eq:operator-Auh}), 
and the Cauchy--Schwarz and Young's inequalities (cf. \eqref{eq:Young-inequality}), 
to obtain
\begin{equation}\label{eq:stability-1-h}
\begin{array}{l}
\ds\dfrac{1}{2}\,d_t\Big(\|\bu_h^n\|^2_{\bL^2(\Omega)} + \kappa^2\,\|\nabla\bu_h^n\|^2_{\bL^2(\Omega)} \Big)
+ \frac{1}{2}\,\Delta t\,\Big( \|d_t\bu_h^n\|^2_{\bL^2(\Omega)} + \kappa^2\,\|d_t\nabla\bu_h^n\|^2_{\bL^2(\Omega)} \Big) \\[2ex]
\ds\quad +\, \tD\,\|\bu^n_h\|^2_{\bL^2(\Omega)}
+ \tF\,\|\bu^n_h\|^\rho_{\bL^{\rho}(\Omega)}
+ \nu\,\|\bomega^n_h\|^2_{\bL^2(\Omega)} 
\ds \,\leq\, \frac{1}{2}\,\Big(\|\f^n\|^2_{\bH^{-1}(\Omega)} 
+ \|\bu_h^n\|^2_{\bH^1(\Omega)} \Big)\,.
\end{array}
\end{equation}
Then, summing up over the time index $n=1,...,m$, with $m=1,\dots,N$, in \eqref{eq:stability-1-h} and multiplying by $\Delta t$, we get
\begin{equation}\label{eq:stability-3-h}
\begin{array}{l}
\ds \wh{\gamma}_{\tKV}\,\|\bu_h^m\|^2_{\bH^1(\Omega)}
+ \wh{\gamma}_{\tKV}\,(\Delta t)^2 \sum_{n=1}^m\|d_t\bu_h^n\|^2_{\bH^1(\Omega)}
+ 2\,\Delta t \sum_{n=1}^m\, \Big(
\tD\,\|\bu_h^n\|^2_{\bL^2(\Omega)}
+ \nu\,\|\bomega_h^n\|^2_{\bL^2(\Omega)} \Big) \\[2ex]
\ds\quad \leq\, \Delta t\,\sum_{n=1}^m \|\f^n\|^2_{\bH^{-1}(\Omega)} 
+ \wh{\gamma}_{\tKV}\,\|\bu_h^0\|^2_{\bH^1(\Omega)} + \Delta t\,\sum_{n=1}^m \|\bu^n_h\|^2_{\bH^1(\Omega)}\,,
\end{array}
\end{equation}	
with $\wh{\gamma}_\tKV$ as in \eqref{eq:stability-1}.
Notice that, in order to simplify the stability bound, we have neglected the term $\|\bu_h^n\|^\rho_{\bL^\rho(\Omega)}$ in the left-hand side of \eqref{eq:stability-1-h}.
Thus, analogously to \eqref{eq:stability-3}, applying the discrete Gr\"onwall inequality (cf. Lemma \ref{lem:discrete-Gronwall}) in \eqref{eq:stability-3-h} and recalling that $N\,\Delta t = T$, and using the estimate \eqref{eq:uh0-wh0-bound-u0-w0}, we deduce the stability bound \eqref{eq:fully-discrete-stability-bound}.

On the other hand, from the discrete inf-sup condition of $\cB$ 
(cf. \eqref{discrete-inf-sup}) and the first equation 
of \eqref{eq:fully-discrete-weak-Brinkman-Forchheimer} related to $\bv_h$, we deduce the discrete version of \eqref{eq:stability-p-inf-sup}, that is, 
\begin{equation}\label{eq:stability-sigmah-n-C2}
\begin{array}{l}
\ds \beta_\ttd\,\|p_h^n\|_{\L^2(\Omega)}
\,\leq\, \|\f^n\|_{\bH^{-1}(\Omega)} 
+ \tD\,\|\bu_h^n\|_{\bL^2(\Omega)} 
+ \nu\,\|\bomega_h^n\|_{\bL^2(\Omega)}  \\[2ex]
\ds\quad \,+\, \|\bi_4\|^2\,\|\bu_h^n\|^2_{\bH^1(\Omega)} 
+ \tF\,\|\bi_\rho\|^\rho\,\|\bu_h^n\|^{\rho-1}_{\bH^1(\Omega)} 
+ (1 + \kappa^2)\|d_t\bu_h^n\|_{\bH^1(\Omega)} \,.
\end{array}
\end{equation}
Then, squaring \eqref{eq:stability-sigmah-n-C2}, summing up over the time index $n=1,...,m$, with $m=1,\dots,N$, and multiplying by $\Delta t$, 
we deduce analogously to \eqref{eq:stability-4}, that there exists $C_1>0$ depending on $|\Omega|, \|\bi_4\|, \|\bi_\rho\|, \nu, \tD, \tF, \kappa$, and $\beta_\ttd$, such that
\begin{align}\label{eq:stability-4-h}
\begin{array}{l}
\ds \Delta t\,\sum_{n=1}^m\,\|p_h^n\|_{\L^2(\Omega)}^2
\,\leq\, C_1\,\bigg\{ \Delta t\,\sum_{n=1}^m 
\Big( \|\f^n\|_{\bH^{-1}(\Omega)}^2
+ \|\bu^n_h\|_{\bL^2(\Omega)}^2 + \|\bomega_h^n\|_{\bL^2(\Omega)}^2 \Big) \\[3ex]
\ds\quad +\, \Delta t\,\sum_{n=1}^m\,\Big( 
\|\bu^n_h\|^{4}_{\bH^1(\Omega)} 
+ \|\bu^n_h\|^{2\,(\rho-1)}_{\bH^1(\Omega)}
+ \|d_t\bu_h^n\|_{\bH^1(\Omega)}^2 \Big) \bigg\}\,.
\end{array}
\end{align}
Next, in order to bound the last term in \eqref{eq:stability-4-h}, we  
choose $(\ubv_h,q_h) = ((d_t\,\bu_h^n,\bomega_h^n),p_h^n)$ in \eqref{eq:fully-discrete-weak-Brinkman-Forchheimer},
apply some algebraic manipulation, use the identity \eqref{eq:identity-uhn}
and the Cauchy--Schwarz and Young's inequalities, to obtain
the discrete version of \eqref{eq:stability-4-b}:
\begin{align}\label{eq:stability-5-h}
&\ds \wh{\gamma}_\tKV\,\|d_t\bu_h^n\|^2_{\bH^1(\Omega)} 
+ \frac{1}{2}\,d_t\,\Big( \tD\,\|\bu_h^n\|^2_{\bL^2(\Omega)} 
+ \nu\,\|\bomega_h^n\|^2_{\bL^2(\Omega)} \Big) 
+ \dfrac{1}{2}\,\Delta t\,\Big( \tD\,\|d_t\bu_h^n\|^2_{\bL^2(\Omega)}
+ \nu\,\|d_t\bomega_h^n\|^2_{\bL^2(\Omega)} \Big) \nonumber \\
& \ds\quad +\, \tF\,(|\bu_h^n|^{\rho-2}\bu_h^n,d_t\bu_h^n)_\Omega 
\,\leq\, \left(\|\f^n\|_{\bH^{-1}(\Omega)} + \|\bi_4\|^2\left( 1 + \frac{\sqrt{d}}{2} \right) \|\bu^n_h\|^2_{\bH^1(\Omega)}\right) \|d_t\bu^n_h\|_{\bH^1(\Omega)} \nonumber \\
& \ds\quad \,\leq\, C_2\,\Big( \|\f^n\|^2_{\bH^{-1}(\Omega)} + \|\bu^n_h\|^4_{\bH^1(\Omega)} \Big) + \dfrac{\wh{\gamma}_{\tKV}}{2}\,\|d_t\bu^n_h\|^2_{\bH^1(\Omega)} \,,
\end{align}
with $\wh{\gamma}_\tKV$ as in \eqref{eq:stability-1} and $C_2>0$ depending on $\|\bi_4\|, d$, and $\kappa$.
In turn, employing H\"older and Young's inequalities, we are able to deduce (cf. \cite[eq. (5.13)]{covy2022}):
\begin{equation}\label{eq:stability-6-h}
(|\bu_h^n|^{\rho-2}\bu_h^n,d_t\bu_h^n)_\Omega 
\,\geq\, \dfrac{(\Delta t)^{-1}}{\rho} \Big(\|\bu_h^n\|_{\bL^\rho(\Omega)}^\rho 
- \|\bu_h^{n-1}\|_{\bL^\rho(\Omega)}^\rho\Big)
= \dfrac{1}{\rho}\,d_t\,\|\bu_h^n\|_{\bL^\rho(\Omega)}^\rho\,.
\end{equation}
Thus, combining \eqref{eq:stability-5-h} with \eqref{eq:stability-6-h}, using Young's inequality, summing up over the time index $n=1,...,m$, with $m=1,\dots,N$ and multiplying by $\Delta t$, we get the discrete version of \eqref{eq:stability-5}:
\begin{equation}\label{eq:stability-7-h}
\begin{array}{l}
\ds \tD\,\|\bu^m_h\|^2_{\bL^2(\Omega)}
+ \dfrac{2\,\tF}{\rho}\|\bu_h^m\|_{\bL^\rho(\Omega)}^\rho 
+ \nu\,\|\bomega^m_h\|^2_{\bL^2(\Omega)} 
\ds + \wh{\gamma}_\tKV\,\Delta t\,\sum_{n=1}^m\,\|d_t\bu_h^n\|_{\bH^1(\Omega)}^2 \\[3ex]
\ds\quad \leq\, 2\,C_2\,\Delta t\,\sum_{n=1}^m \Big( \|\f^n\|_{\bH^{-1}(\Omega)}^2 + \|\bu^n_h\|^4_{\bH^1(\Omega)} \Big) 
+ \tD\,\|\bu_h^0\|^2_{\bL^2(\Omega)} 
+ \frac{2\,\tF}{\rho}\,\|\bu_h^0\|_{\bL^\rho(\Omega)}^\rho 
+ \nu\,\|\bomega_h^0\|^2_{\bL^2(\Omega)} \,.
\end{array}
\end{equation}
Combining \eqref{eq:stability-4-h} with \eqref{eq:stability-3-h} and \eqref{eq:stability-7-h}, using the fact that $(\bu^0_h,\bomega^0_h) = (\bu_{h,0},\bomega_{h,0})$
and \eqref{eq:uh0-wh0-bound-u0-w0}, we deduce that
\begin{equation}\label{eq:stability-8-h}
\begin{array}{l}
\ds \Delta t\,\sum_{n=1}^m\,\|p_h^n\|_{\L^2(\Omega)}^2
\,\leq\, C_3\,\bigg\{ \Delta t\,\sum_{n=1}^m 
\|\f^n\|_{\bH^{-1}(\Omega)}^2
+ \|\bu_0\|^2_{\bH^1(\Omega)} 
+ \|\bu_0\|^4_{\bH^1(\Omega)}  
+ \|\bu_0\|^{2(\rho-1)}_{\bH^1(\Omega)}  \\[3ex]
\ds\quad +\, \Delta t\,\sum_{n=1}^m\,\Big( \|\bu^n_h\|^{2}_{\bH^1(\Omega)} + \|\bu^n_h\|^{4}_{\bH^1(\Omega)} + \|\bu^n_h\|^{2\,(\rho-1)}_{\bH^1(\Omega)} \Big) \bigg\}\,,
\end{array}
\end{equation}
with $m=1,\dots,N$ and $C_3>0$ depending on $|\Omega|, \|\bi_4\|, \|\bi_\rho\|, \nu, \tD, \tF, \kappa, d$, and $\beta_\ttd$.
Then, using \eqref{eq:fully-discrete-stability-bound} to bound $\|\bu^n_h\|^{2}_{\bH^1(\Omega)}$, $\|\bu^n_h\|^{4}_{\bH^1(\Omega)}$, and $\|\bu^n_h\|^{2\,(\rho-1)}_{\bH^1(\Omega)}$ in the left-hand side of \eqref{eq:stability-8-h}, we obtain \eqref{eq:fully-discrete-stability-bound-p}.

Finally, uniqueness of the solution at each time step can be established
using that $\|\bu_h\|_{\ell^\infty(0,T;\bH^1(\Omega))}$ is bounded by data (cf. \eqref{eq:fully-discrete-stability-bound}), following the argument showing uniqueness of the weak solution in Theorem~\ref{thm:well-posed-result-reduced}. In particular, starting from the time-discrete version of \eqref{eq:uniqueness-0}, the uniqueness follows from summing in time and using the discrete Gr\"onwall inequality (cf. Lemma \ref{lem:discrete-Gronwall}).
\end{proof}

Now, we proceed with establishing rates of convergence for 
the fully discrete scheme \eqref{eq:fully-discrete-weak-Brinkman-Forchheimer}.
To that end, we subtract the fully discrete problem \eqref{eq:fully-discrete-weak-Brinkman-Forchheimer} 
from the continuous counterparts \eqref{eq:KVBF-variational-formulation} 
at each time step $n = 1,\dots, N$, to obtain the following error system:
\begin{equation}\label{eq:error-system-2}
\begin{array}{lll}
\ds d_t\,[\cE(\be^n_{\ubu}),\ubv_h] + [\cA_h(\bu^n)(\ubu^n) - \cA_h(\bu^n_h)(\ubu^n_h),\ubv_h] + [\cB(\ubv_h),\be^n_{p}] & = & [r_n(\bu),\bv_h] \,, \\ [2ex]
\ds [\cB(\be^n_{\ubu}),q_h] & = & 0 \,,
\end{array}
\end{equation}
for all $\ubv_h\in \bH^{\bu}_{h}\times \bH^{\bomega}_h$ and $q_h\in \H^p_h$, where 
\begin{equation*}
[r_n(\bu),\bv_h] := (r_n(\bu), \bv_h)_\Omega + \kappa^2\,(r_n(\nabla\bu), \nabla\bv_h)_\Omega \,,
\end{equation*}
and $r_n$ denotes the difference between the time derivative and its discrete analog, that is
\begin{equation*}
r_n(\bu) \,=\, d_{t}\,\bu^n - \partial_t\,\bu(t_n).
\end{equation*}
In addition, we recall from \cite[Lemma 4]{byz2015} that for sufficiently smooth $\bu$, there holds
\begin{equation}\label{eq:bound-second-time-derivative}
\Delta\,t\,\sum^N_{n=1} \|r_n(\bu)\|^2_{\bH^1(\Omega)} 
\,\leq\, C(\partial_{tt}\,\bu)\,(\Delta\,t)^2,\quad\mbox{with}\quad C(\partial_{tt}\,\bu) \,:=\, C\,\|\partial_{tt}\,\bu\|^2_{\L^2(0,T;\bH^1(\Omega))}.
\end{equation}
Then, the proof of the theoretical rate of convergence of the fully discrete
scheme \eqref{eq:fully-discrete-weak-Brinkman-Forchheimer} follows the structure of the proof 
of Theorem \ref{thm:rate-of-convergence}, using discrete-in-time arguments as in the proof 
of Theorem~\ref{thm:well-posed-result-fully-discrete-problem}, the discrete Gr\"onwall inequality (cf. Lemma \ref{lem:discrete-Gronwall}) and the estimate \eqref{eq:bound-second-time-derivative} (see \cite[Theorem 5.4]{covy2022} for a similar approach).
\begin{thm}\label{thm:rate-of-convergence-fully-discrete}
Let the assumptions of Theorem~\ref{thm:rate-of-convergence} hold.
Then, for the solution of the fully discrete problem 
\eqref{eq:fully-discrete-weak-Brinkman-Forchheimer} there exists 
$\wh{C}(\ubu,p) > 0$ depending only on
$C(\ubu), C(\partial_t\,\ubu)$, $C(\partial_{tt}\,\ubu), C(p)$, $C(\partial_t\,p), |\Omega|$, $\|\bi_\rho\|$, $\|\bi_4\|, \nu, \tD, \tF, \kappa, \beta_\ttd, T, \|\f\|_{\ell^2(0,T;\bH^{-1}(\Omega))}$, and $\|\bu_0\|_{\bH^1(\Omega)}$, such that
\begin{align}\label{eq:rate-estimate-u-omega-p-fully-discrete}
\|\be_{\bu}&\|_{\ell^\infty(0,T;\bH^1(\Omega))} 
+ \Delta\,t\,\|d_{t}\,\be_{\bu}\|_{\ell^2(0,T;\bH^1(\Omega))} 
+ \|\be_{\bu}\|_{\ell^{2}(0,T;\bL^2(\Omega))} \nonumber \\[2ex]
& + \|\be_{\bomega}\|_{\ell^2(0,T;\bL^2(\Omega))}
+ \|\be_{p}\|_{\ell^2(0,T;\L^2(\Omega))}
\,\leq\, \wh{C}(\ubu,p)\,\left( h^s + h^{s\,(\rho-1)} + \Delta\,t \right) \,.
\end{align}
\end{thm}

\begin{rem}
For the fully discrete scheme \eqref{eq:fully-discrete-weak-Brinkman-Forchheimer} we have considered the backward Euler method
only for the sake of simplicity. The analysis developed in Section \ref{sec:fully-discrete-approximation}
can be adapted to other time discretizations, such as BDF schemes or the Crank-Nicholson method.
\end{rem}


\section{Numerical results}\label{sec:numerical-results}

In this section we present three numerical results that illustrate the performance 
of the fully discrete method \eqref{eq:fully-discrete-weak-Brinkman-Forchheimer}.
The implementation is based on a {\tt FreeFem++} code \cite{Hecht2012}.
We use quasi-uniform triangulations and the finite element subspaces detailed in Section \ref{sec:existence-uniqueness-solution}
(cf. \eqref{eq:Taylor-Hood}). The nonlinearly is handled using a Newton--Raphson algorithm with a fixed tolerance $\tol=1\textrm{E}-06$. The iterative process is stopped when the relative error between two consecutive iterations of the complete coefficient vector, namely $\mathbf{coeff}^{m}$ and $\mathbf{coeff}^{m+1}$,
is sufficiently small, i.e.,
\begin{equation*}
\frac{\| \mathbf{coeff}^{m+1} \,-\, \mathbf{coeff}^m \|_\DoF}{\| \mathbf{coeff}^{m+1} \|_\DoF} 
\,\leq\, \tol\,,
\end{equation*}
where $\|\cdot\|_\DoF$ stands for the usual Euclidean norm in $\R^{\DoF}$, with $\DoF$ 
denoting the total number of degrees of freedom defined by the finite element subspaces 
$\bH^\bu_h, \bH^\bomega_h$ and $\H^p_h$ (cf. \eqref{eq:Taylor-Hood}).

Examples~1 and 2 are used to corroborate the rate of convergence in two and three dimensional domains, respectively. 
The total simulation time for these examples is $T=0.001$ and the time step is $\Delta\,t = 10^{-4}$.
The time step is sufficiently small, so that the time discretization error does not affect the convergence rates.
On the other hand, Example 3 is utilized to analyze the method's behavior under various scenarios, considering different Darcy and Forchheimer coefficients, as well as varying values of the elasticity parameter $\kappa$.
For these cases, the total simulation time and the time step are chosen as $T = 1$ and $\Delta\,t = 10^{-2}$, respectively.


\subsection*{Example 1: Two-dimensional smooth exact solution}

In this test we study the convergence for the space discretization using an analytical solution. 
The domain is the square $\Omega=(0,1)^2$. We consider $\rho=3, \nu = 1, \tD=1, \tF=10, \kappa=1$, 
and the datum $\f$ is adjusted so that the exact solution is given by the smooth functions
\begin{equation*}
\qquad \bu \,=\, \exp(t)\begin{pmatrix}
\sin(\pi x)\cos(\pi y)\\
-\cos(\pi x)\sin(\pi y) 
\end{pmatrix} \qan
p \,=\, \exp(t)\cos(\pi x)\sin\Big(\dfrac{\pi y}{2}\Big)\,.
\end{equation*}
The model problem is then complemented with the appropriate Dirichlet boundary condition and
initial data.

In Figure \ref{fig:example1} we display the solution obtained with the Crouzeix--Raviart-based approximation, $39,146$ triangle elements and $176,926\,\DoF$ at time $T = 0.001$.  
Table~\ref{table1-example1} shows the convergence history for a sequence of quasi-uniform mesh refinements, including the average number of Newton iterations. 
The results confirm that the optimal spatial rates of convergence
$\cO(h^{k+1})$ provided by Theorem \ref{thm:rate-of-convergence-fully-discrete} 
(see also Theorem \ref{thm:rate-of-convergence}) are attained for the Taylor--Hood based scheme, with $k=1$.
In addition, optimal order $\cO(h)$ is also obtained for the MINI-element and Crouzeix--Raviart based discretizations. The Newton's method exhibits behavior independent of
the mesh size, converging in $2.1$ iterations in average in all cases.

\begin{figure}
\begin{center}		
\includegraphics[width=5cm]{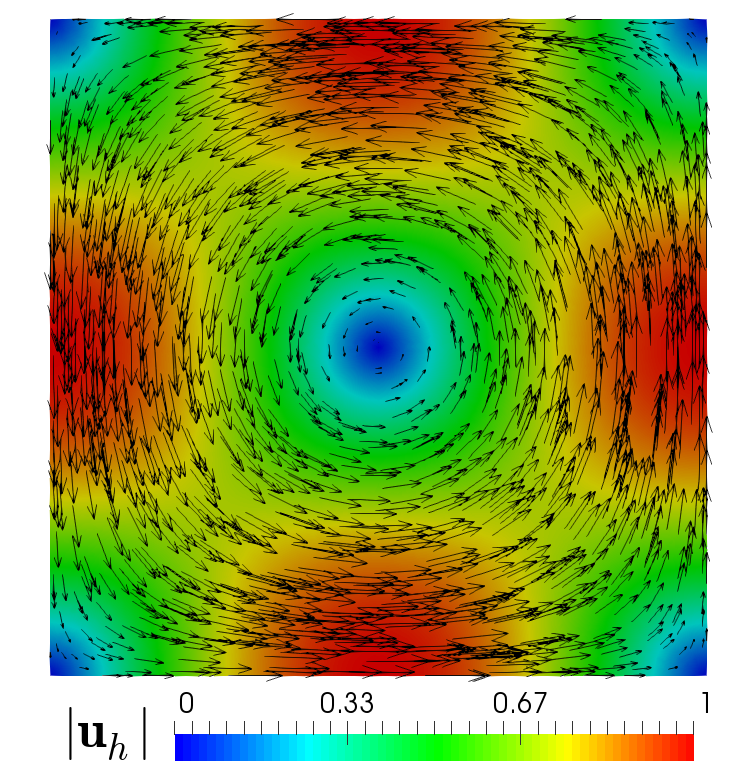}
\includegraphics[width=5cm]{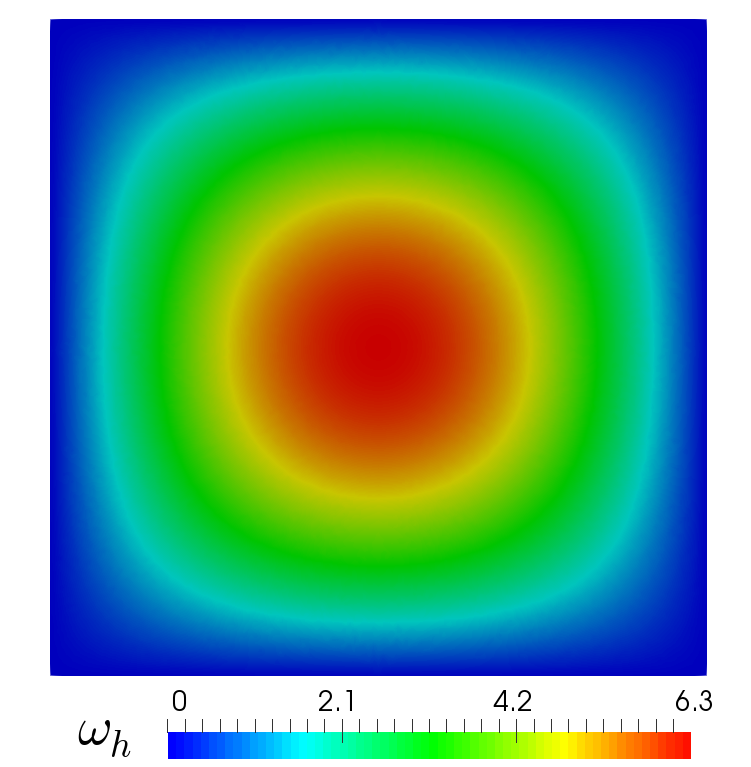}	
\includegraphics[width=5cm]{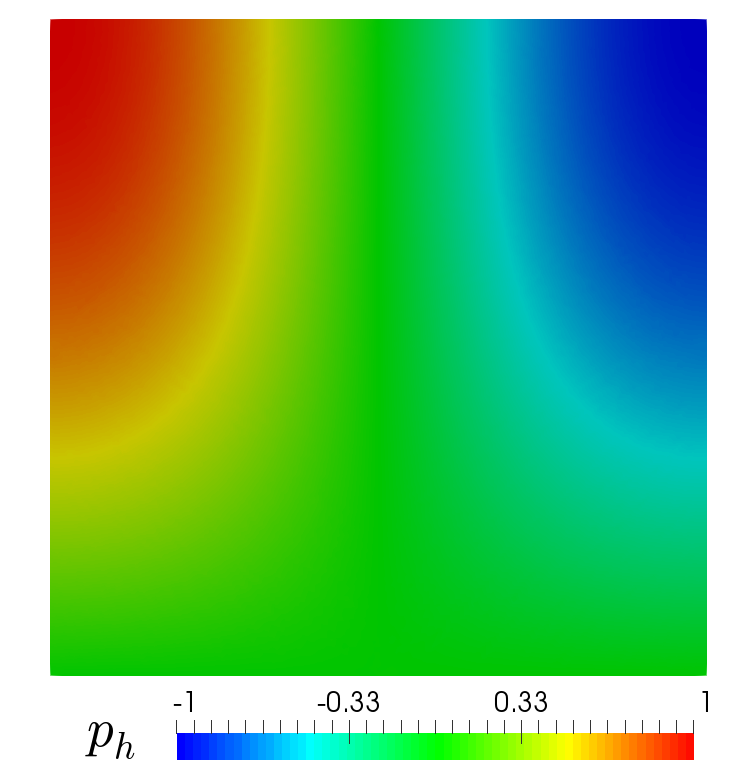}
		
\caption{[Example 1] Computed magnitude of the velocity, vorticity and pressure fields at time $T=0.001$.}\label{fig:example1}
\end{center}
\end{figure}

\begin{table}
\begin{center}	
\resizebox{\textwidth}{!}{
\begin{tabular}{|r|c|c||c c|c c|c c|c c|}
\hline	
\multicolumn{11}{|c|}{Taylor--Hood-based discretization}  \\	
\hline
& & & \multicolumn{2}{c|}{$\|\be_{\bu}\|_{\ell^\infty(0,T;\bH^1(\Omega))}$} 
& \multicolumn{2}{c|}{$\|\be_{\bu}\|_{\ell^{2}(0,T;\bL^2(\Omega))}$} 
& \multicolumn{2}{c|}{$\|\be_{\bomega}\|_{\ell^2(0,T;\bL^2(\Omega))}$} 
& \multicolumn{2}{c|}{$\|\be_{p}\|_{\ell^2(0,T;\L^2(\Omega))}$} \\
$\DoF$ & $h$ & $\iter$  &  error   & rate   &  error   & rate   &  error   & rate   &  error   & rate   \\ \hline \hline
232    & 0.373 & 2.1 & 1.95E-01 &   --  & 2.03E-04 &   --  & 6.37E-03 &   --  & 3.76E-00 &   --  \\
860    & 0.196 & 2.1 & 3.57E-02 & 2.653 & 1.80E-05 & 3.782 & 1.18E-03 & 2.636 & 5.15E-01 & 3.105 \\
3216   & 0.097 & 2.1 & 8.69E-03 & 2.001 & 2.17E-06 & 2.998 & 2.90E-04 & 1.986 & 9.16E-02 & 2.448 \\
12468  & 0.048 & 2.1 & 2.00E-03 & 2.074 & 2.51E-07 & 3.049 & 6.76E-05 & 2.058 & 1.42E-02 & 2.631 \\
49142  & 0.025 & 2.1 & 5.21E-04 & 2.013 & 3.29E-08 & 3.042 & 1.78E-05 & 2.000 & 4.03E-03 & 1.888 \\
197270 & 0.013 & 2.1 & 1.27E-04 & 2.160 & 3.94E-09 & 3.252 & 4.30E-06 & 2.176 & 8.17E-04 & 2.447 \\
\hline 
\end{tabular}
}	

\medskip
	
\resizebox{\textwidth}{!}{
\begin{tabular}{|r|c|c||c c|c c|c c|c c|}
\hline	
\multicolumn{11}{|c|}{MINI-element-based discretization}  \\	
\hline
& & & \multicolumn{2}{c|}{$\|\be_{\bu}\|_{\ell^\infty(0,T;\bH^1(\Omega))}$} 
& \multicolumn{2}{c|}{$\|\be_{\bu}\|_{\ell^{2}(0,T;\bL^2(\Omega))}$} 
& \multicolumn{2}{c|}{$\|\be_{\bomega}\|_{\ell^2(0,T;\bL^2(\Omega))}$} 
& \multicolumn{2}{c|}{$\|\be_{p}\|_{\ell^2(0,T;\L^2(\Omega))}$} \\
$\DoF$ & $h$ & $\iter$  &  error   & rate   &  error   & rate   &  error   & rate   &  error   & rate   \\ \hline \hline
180    & 0.373 & 2.1 & 1.22E-00 &   --  & 1.23E-03 &   --  & 9.50E-03 &   --  & 2.52E+01 &   --  \\
676    & 0.196 & 2.1 & 6.62E-01 & 0.961 & 2.95E-04 & 2.224 & 2.43E-03 & 2.132 & 6.37E-00 & 2.149 \\
2548   & 0.097 & 2.1 & 3.30E-01 & 0.985 & 7.39E-05 & 1.962 & 8.68E-04 & 1.456 & 3.31E-00 & 0.925 \\
9924   & 0.048 & 2.1 & 1.68E-01 & 0.957 & 1.87E-05 & 1.943 & 3.62E-04 & 1.236 & 1.50E-00 & 1.125 \\
39212  & 0.025 & 2.1 & 8.47E-02 & 1.024 & 4.69E-06 & 2.068 & 1.70E-04 & 1.130 & 7.25E-01 & 1.084 \\
157612 & 0.013 & 2.1 & 4.14E-02 & 1.098 & 1.15E-06 & 2.162 & 7.52E-05 & 1.252 & 3.51E-01 & 1.110 \\
\hline 
\end{tabular}
}	

\medskip

\resizebox{\textwidth}{!}{
\begin{tabular}{|r|c|c||c c|c c|c c|c c|}
\hline	
\multicolumn{11}{|c|}{Crouzeix--Raviart-based discretization}  \\	
\hline
& & & \multicolumn{2}{c|}{$\|\be_{\bu}\|_{\ell^\infty(0,T;\bH^1(\Omega))}$} 
& \multicolumn{2}{c|}{$\|\be_{\bu}\|_{\ell^{2}(0,T;\bL^2(\Omega))}$} 
& \multicolumn{2}{c|}{$\|\be_{\bomega}\|_{\ell^2(0,T;\bL^2(\Omega))}$} 
& \multicolumn{2}{c|}{$\|\be_{p}\|_{\ell^2(0,T;\L^2(\Omega))}$} \\
$\DoF$ & $h$ & $\iter$  &  error   & rate   &  error   & rate   &  error   & rate   &  error   & rate   \\ \hline \hline
187    & 0.373 & 2.1 & 7.59E-01 &   --  & 1.05E-03 &   --  & 1.45E-02 &   --  & 9.57E-00 &   --  \\
733    & 0.196 & 2.1 & 3.87E-01 & 1.050 & 2.61E-04 & 2.170 & 4.82E-03 & 1.722 & 5.67E-00 & 0.819 \\
2815   & 0.097 & 2.1 & 1.95E-01 & 0.974 & 6.59E-05 & 1.951 & 1.74E-03 & 1.447 & 3.04E-00 & 0.885 \\
11065  & 0.048 & 2.1 & 9.86E-02 & 0.962 & 1.70E-05 & 1.918 & 7.61E-04 & 1.168 & 1.47E-00 & 1.023 \\
43918  & 0.025 & 2.1 & 4.93E-02 & 1.037 & 4.22E-06 & 2.083 & 3.75E-04 & 1.057 & 7.91E-01 & 0.930 \\
176926 & 0.013 & 2.1 & 2.44E-02 & 1.081 & 1.03E-06 & 2.167 & 1.70E-04 & 1.219 & 3.76E-01 & 1.138 \\
\hline 
\end{tabular}
}	
	
\caption{[Example 1] Number of degrees of freedom, mesh sizes, average number of Newton iterations, errors, and rates of convergence with $\rho=3, \nu=1, \tD=1, \tF=10$, and $\kappa=1$.}\label{table1-example1}
\end{center}
\end{table}


\subsection*{Example 2: Three-dimensional smooth exact solution}

In the second example, we consider the cube domain $\Omega = (0,1)^3$ and the exact solution
\begin{equation*}
\bu \,=\, \exp(t)
\left(\begin{array}{c}
\sin(\pi\,x)\cos(\pi\,y)\cos(\pi\,z) \\ 
-2\,\cos(\pi\,x)\sin(\pi\,y)\cos(\pi\,z) \\
\cos(\pi\,x)\cos(\pi\,y)\sin(\pi\,z)
\end{array}\right)\qan
p \,=\, \exp(t)\,(x - 0.5)^3\sin(y+z)\,.
\end{equation*}
Similarly to the first example, we consider the parameters $\rho=4, \nu=1, \tD=1, \tF=10$, and $\kappa=1$, and the right-hand side function $\f$ is computed from \eqref{eq:KVBF-2} using the above solution.

The numerical solution obtained with the Taylor--Hood-based approximation, $63,888$ tetrahedral elements and $322,043\, \DoF$ at time $T = 0.001$  is shown in
Figure~\ref{fig:example2}. The convergence history for a set of
quasi-uniform mesh refinements using Taylor--Hood and MINI-element based approximations is shown in
Table~\ref{table2-example2}. Again, the mixed finite element method
converges optimally with order $\cO(h^2)$ and $\cO(h)$, respectively, as it was proved by Theorem~\ref{thm:rate-of-convergence-fully-discrete} 
(see also Theorem~\ref{thm:rate-of-convergence}).

\begin{figure}
\begin{center}		
\includegraphics[width=5cm]{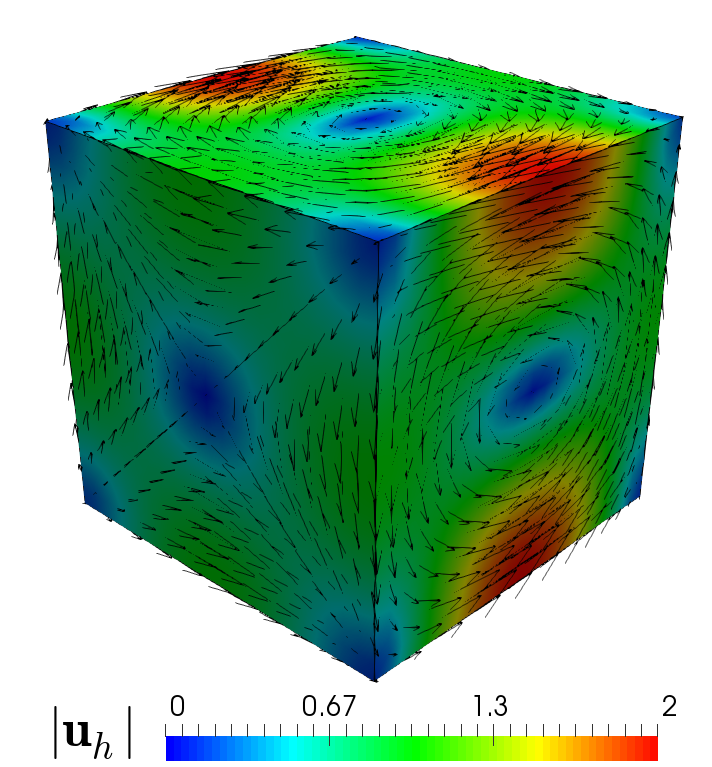}
\includegraphics[width=5cm]{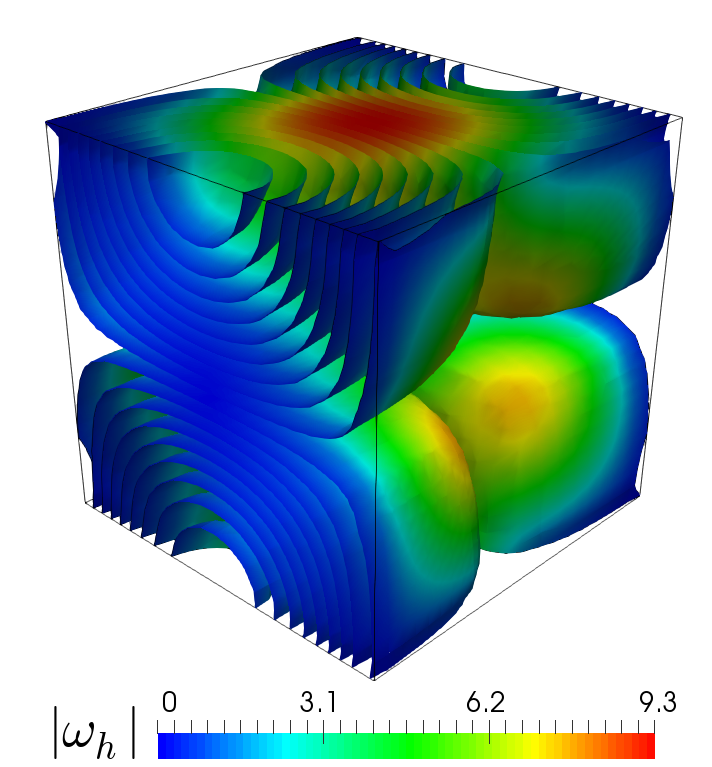}	
\includegraphics[width=5cm]{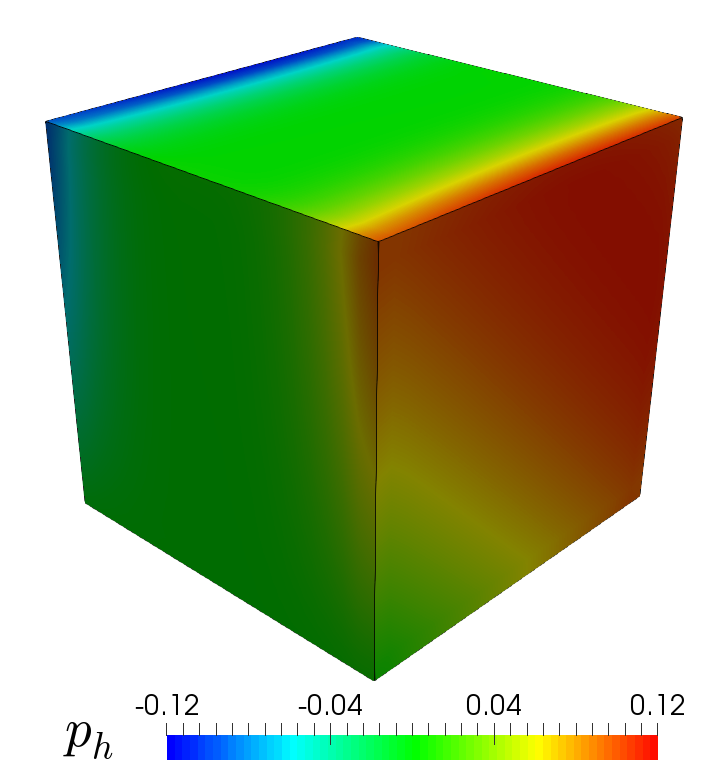}
		
\caption{[Example 2] Computed magnitude of the velocity and vorticity, and pressure field at time $T=0.001$.}\label{fig:example2}
\end{center}
\end{figure}

\begin{table}
\begin{center}	
\resizebox{\textwidth}{!}{
\begin{tabular}{|r|c|c||c c|c c|c c|c c|}
\hline	
\multicolumn{11}{|c|}{Taylor--Hood based discretization}  \\	
\hline
& & & \multicolumn{2}{c|}{$\|\be_{\bu}\|_{\ell^\infty(0,T;\bH^1(\Omega))}$} 
& \multicolumn{2}{c|}{$\|\be_{\bu}\|_{\ell^{2}(0,T;\bL^2(\Omega))}$} 
& \multicolumn{2}{c|}{$\|\be_{\bomega}\|_{\ell^2(0,T;\bL^2(\Omega))}$} 
& \multicolumn{2}{c|}{$\|\be_{p}\|_{\ell^2(0,T;\L^2(\Omega))}$} \\
$\DoF$ & $h$ & $\iter$  &  error   & rate   &  error   & rate   &  error   & rate   &  error   & rate   \\ \hline \hline
483    & 0.707 & 2.1 & 1.56E-00 &   --  & 2.76E-03 &   --  & 4.48E-02 &   --  & 6.44E+01 &   --  \\
2687   & 0.354 & 2.1 & 4.36E-01 & 1.842 & 3.79E-04 & 2.869 & 1.33E-02 & 1.750 & 4.33E-00 & 3.895 \\
17655  & 0.177 & 2.1 & 1.12E-01 & 1.956 & 4.89E-05 & 2.952 & 3.09E-03 & 2.106 & 2.75E-01 & 3.976 \\
86667  & 0.101 & 2.1 & 3.69E-02 & 1.988 & 9.21E-06 & 2.984 & 9.83E-04 & 2.047 & 3.05E-02 & 3.930 \\
322043 & 0.064 & 2.1 & 1.50E-02 & 1.996 & 2.38E-06 & 2.994 & 3.95E-04 & 2.019 & 5.21E-03 & 3.909 \\
\hline 
\end{tabular}
}	

\medskip

\resizebox{\textwidth}{!}{
\begin{tabular}{|r|c|c||c c|c c|c c|c c|}
\hline	
\multicolumn{11}{|c|}{MINI-element based discretization}  \\	
\hline
& & & \multicolumn{2}{c|}{$\|\be_{\bu}\|_{\ell^\infty(0,T;\bH^1(\Omega))}$} 
& \multicolumn{2}{c|}{$\|\be_{\bu}\|_{\ell^{2}(0,T;\bL^2(\Omega))}$} 
& \multicolumn{2}{c|}{$\|\be_{\bomega}\|_{\ell^2(0,T;\bL^2(\Omega))}$} 
& \multicolumn{2}{c|}{$\|\be_{p}\|_{\ell^2(0,T;\L^2(\Omega))}$} \\
$\DoF$ & $h$ & $\iter$  &  error   & rate   &  error   & rate   &  error   & rate   &  error   & rate   \\ \hline \hline
333    & 0.707 & 2.1 & 7.55E-00 &   --  & 1.08E-02 &   --  & 6.97E-02 &   --  & 1.27E+03 &   --  \\
2027   & 0.354 & 2.1 & 4.53E-00 & 0.738 & 3.52E-03 & 1.615 & 2.36E-02 & 1.563 & 6.43E+02 & 0.986 \\
14319  & 0.177 & 2.1 & 2.27E-00 & 0.999 & 8.77E-04 & 2.004 & 6.58E-03 & 1.842 & 1.87E+02 & 1.777 \\
73017  & 0.101 & 2.1 & 1.29E-00 & 1.010 & 2.79E-04 & 2.046 & 2.32E-03 & 1.862 & 6.79E+01 & 1.812 \\
276833 & 0.064 & 2.1 & 8.17E-01 & 1.007 & 1.12E-04 & 2.023 & 1.01E-03 & 1.848 & 3.02E+01 & 1.791 \\
\hline 
\end{tabular}
}	
\caption{[Example 2] Number of degrees of freedom, mesh sizes, average number of Newton iterations, errors, and rates of convergence with $\rho=4, \nu=1, \tD=1, \tF=10$, and $\kappa=1$.}\label{table2-example2}
\end{center}
\end{table}


\subsection*{Example 3: Flow through porous media with channel network}

Finally, inspired by \cite[Section 5.2.4]{akny2019}, we
focus on a flow through a porous medium with
a channel network.  
We consider the square domain $\Omega = (-1,1)^2$ with an internal channel network denoted as $\Omega_{\rc}$. 
The domain configuration and the prescribed mesh are described in the plots of the first column of Figure \ref{fig:example3-uh-omh-ph-T0-T20-Tend}. 
First, we consider the Kelvin--Voigt--Brinkman--Forchheimer
model \eqref{eq:KVBF-2} in the whole domain $\Omega$, 
with parameters $\rho=3, \nu=1$, and $\kappa=1$ but with different values of the parameters 
$\tD$ and $\tF$ for the interior and the exterior of the channel, that is,
\begin{equation}\label{eq:parameters-D-F-in-channel-network}
\tD \,=\, \left\{\begin{array}{rl}
1 &\mbox{in }\,\, \Omega_{\rc} \\
1000 &\mbox{in }\,\, \ov{\Omega}\setminus \Omega_{\rc}
\end{array}\right. \qan
\tF \,=\, \left\{\begin{array}{rl}
10 &\mbox{in }\,\, \Omega_{\rc} \\
1 &\mbox{in }\,\, \ov{\Omega}\setminus \Omega_{\rc}
\end{array}\right..
\end{equation}
The parameter choice corresponds to high permeability ($\tD = 1$)
in the channel and increased inertial effect ($\tF = 10$), compared to
low permeability ($\tD = 1000$) in the porous medium and reduced
inertial effect ($\tF = 1$).  In addition, the body force term is $\f
= \0$, the initial condition is zero, and the boundaries conditions
are
\begin{equation*}
\bu\cdot\bn = 0.2,\quad \bu\cdot\bt = 0 \qon \Gamma_{\mathrm{left}},\quad
\left( \kappa^2\,\dfrac{\partial\,\nabla\bu}{\partial\,t} - p\,\bI \right)\bn + \nu\,\bomega\bt = \0 \qon \Gamma\setminus\Gamma_{\mathrm{left}}\,,
\end{equation*}
which corresponds to inflow on the left boundary and zero viscoelastic stress outflow on the rest of the boundary.

In Figure \ref{fig:example3-uh-omh-ph-T0-T20-Tend} we display the computed magnitude of the velocity, vorticity and pressure at times $T=0.01$, $T=0.2$, and $T=1$, which were obtained using the MINI-element-based approximation on a mesh with $27,287$ 
triangle elements and $109,682\,\DOF$. 
As expected, we observe a faster flow through the channel network, accompanied by a significant change in vorticity across the interface between the channel and the porous medium.
The pressure field decreases as time increases.
This example illustrates the Kelvin--Voigt--Brinkman--Forchheimer model's capability to handle heterogeneous media with spatially varying parameters. It also demonstrates our three-field mixed finite element method's ability to resolve sharp vorticities in the presence of strong jump discontinuities in the parameters.
We further study the robustness of the method with respect to the elasticity parameter $\kappa$. 
In Figure \ref{fig:example3-uh-omh-ph} we display the computed magnitude of the velocity, vorticity, and pressure for the settings given by \eqref{eq:parameters-D-F-in-channel-network}, considering $\kappa\in \{3, 2, 1, 0.1, 0.01\}$. 
We observe that the elasticity parameter $\kappa$ has a dissipative effect, reducing the velocity in the channel and slightly affecting the pressure in the entire domain, while the vorticity increases as $\kappa$ decreases.
This study illustrates that the method produces stable and physically reasonable results across a wide range of physical parameters, such as $\tD$, $\tF$, and $\kappa$.




\begin{figure}[ht!]
\begin{flushright}
\includegraphics[width=4.1cm]{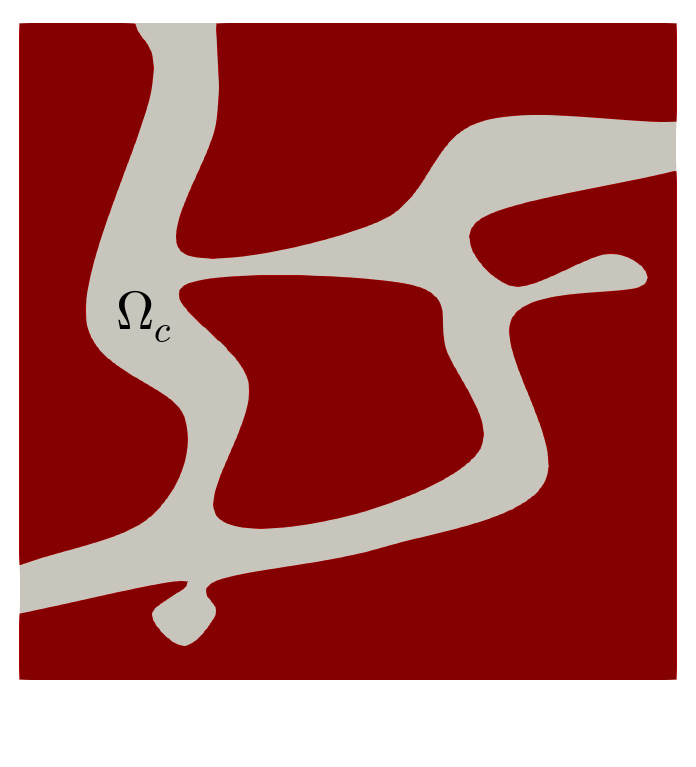}
\includegraphics[width=4.1cm]{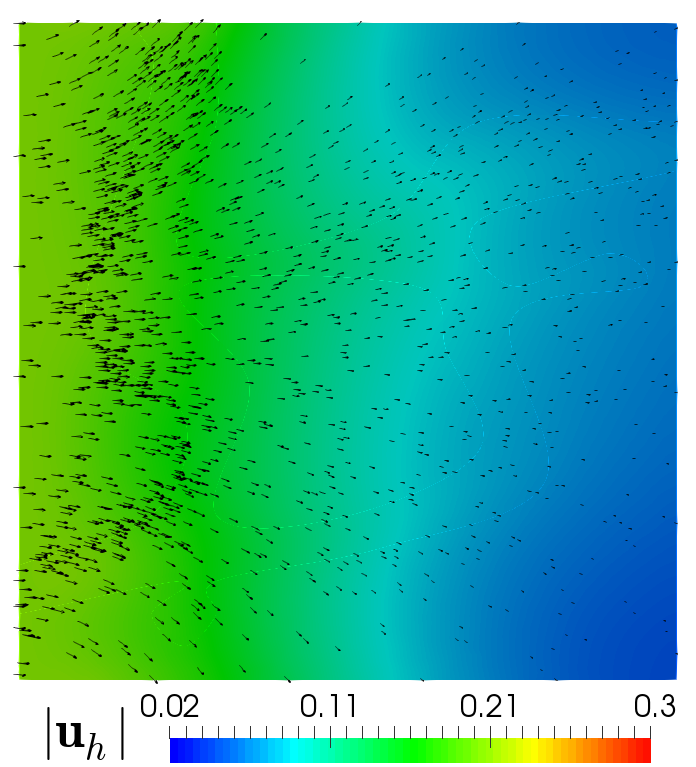}	
\includegraphics[width=4.1cm]{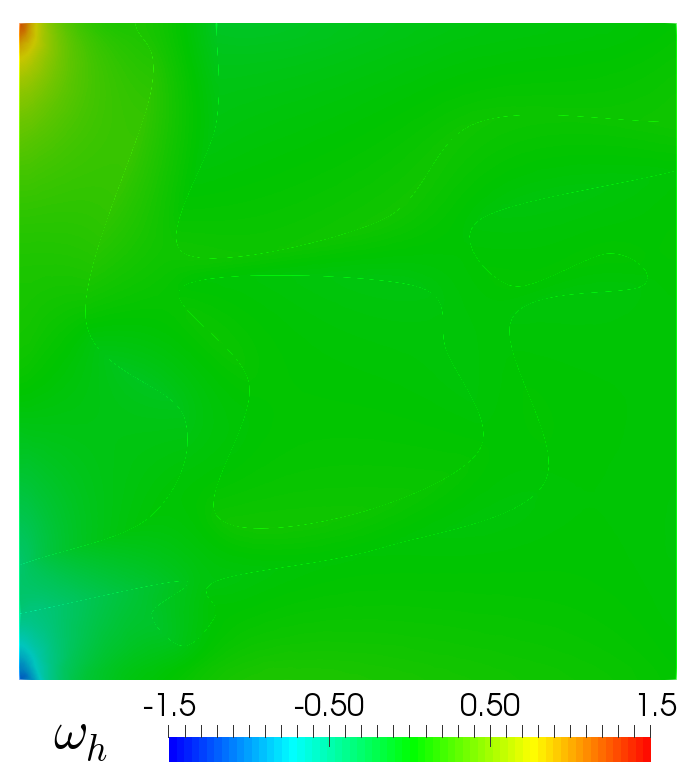}
\includegraphics[width=4.1cm]{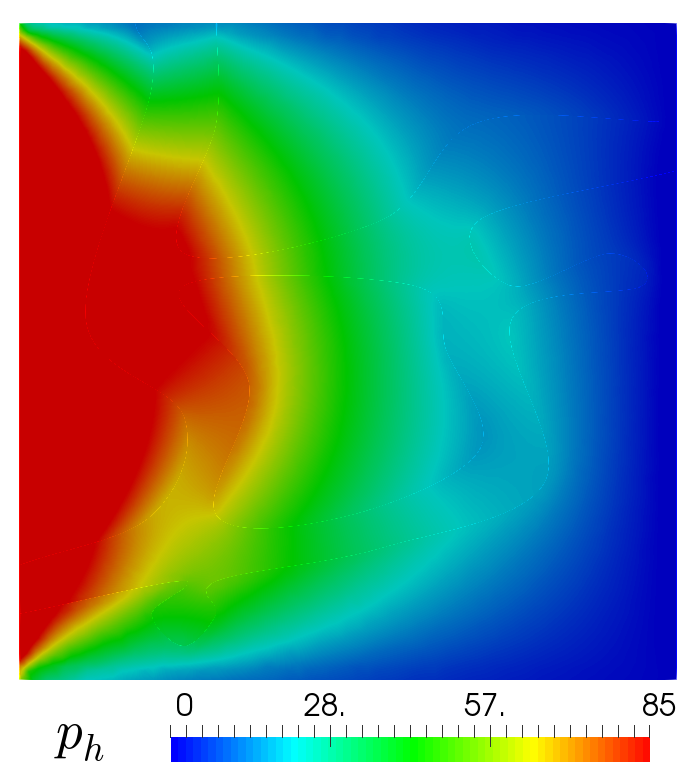}

\includegraphics[width=4.1cm]{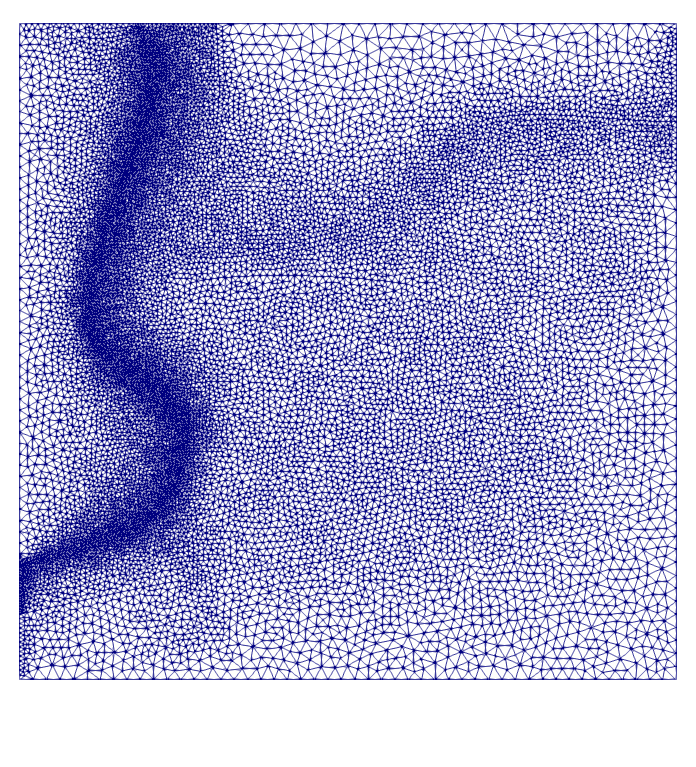}	
\includegraphics[width=4.1cm]{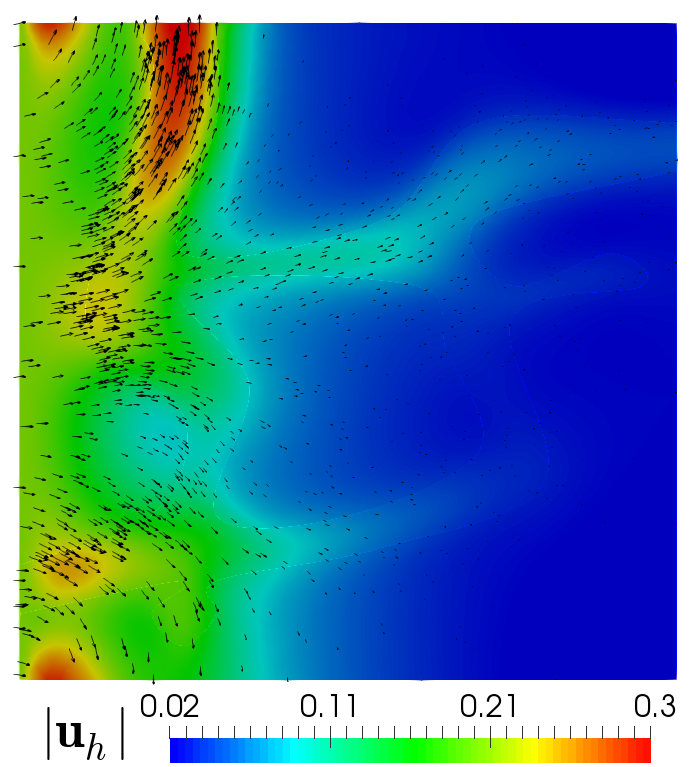}
\includegraphics[width=4.1cm]{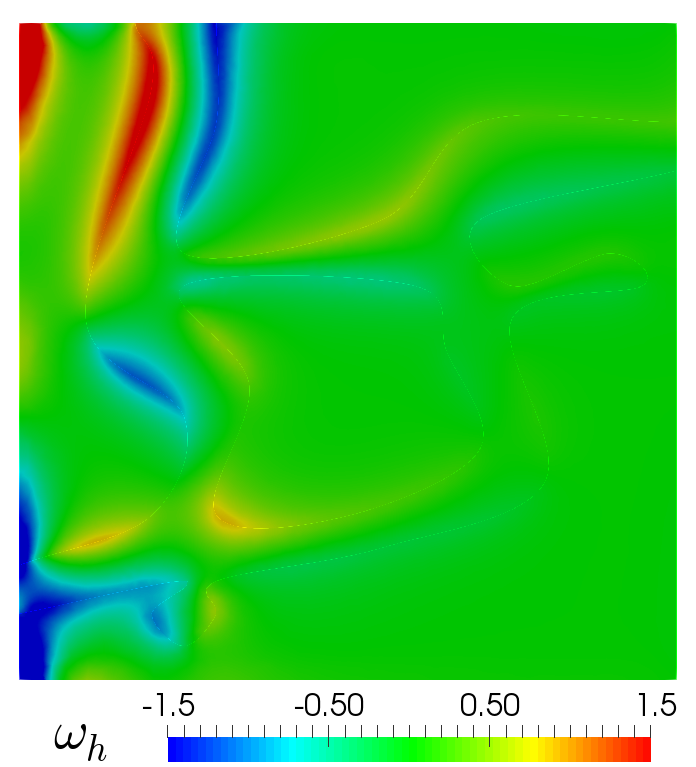}
\includegraphics[width=4.1cm]{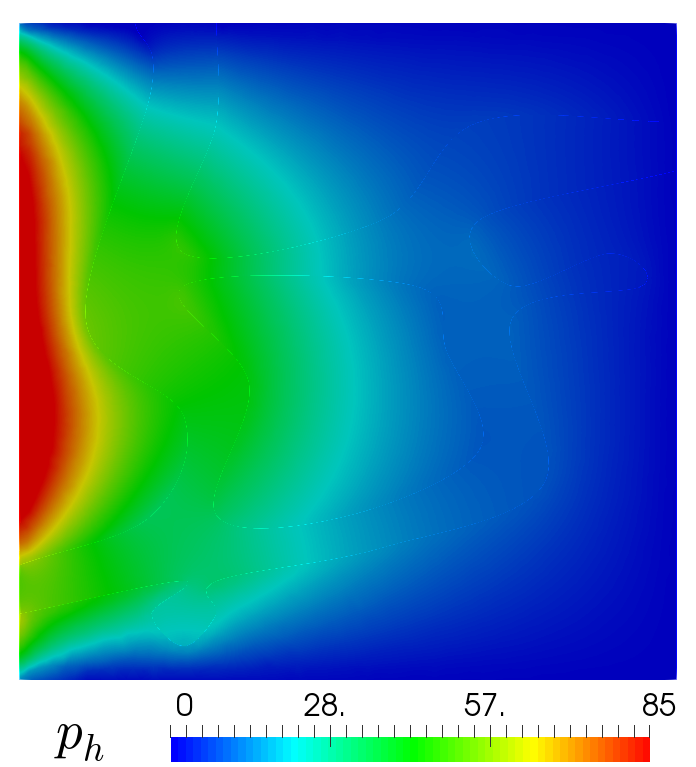}

\includegraphics[width=4.1cm]{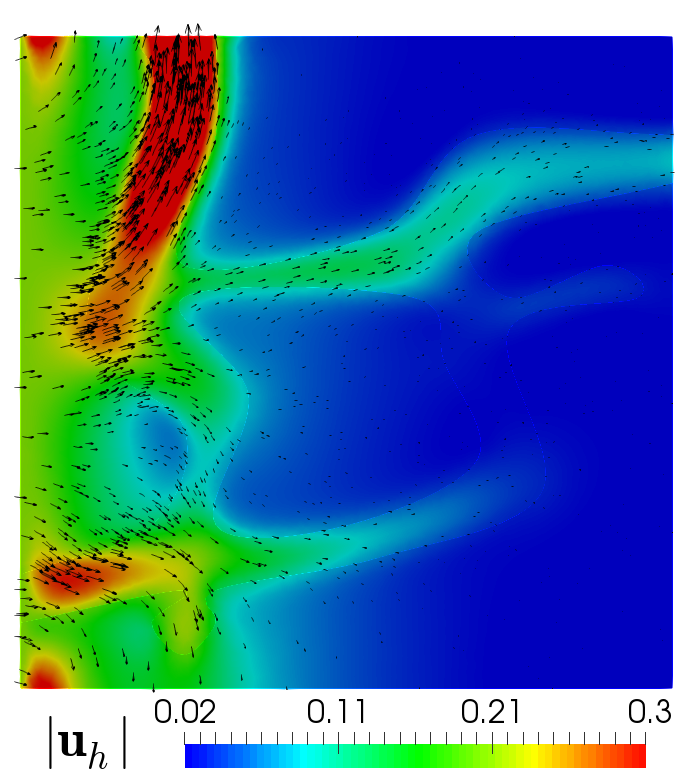}
\includegraphics[width=4.1cm]{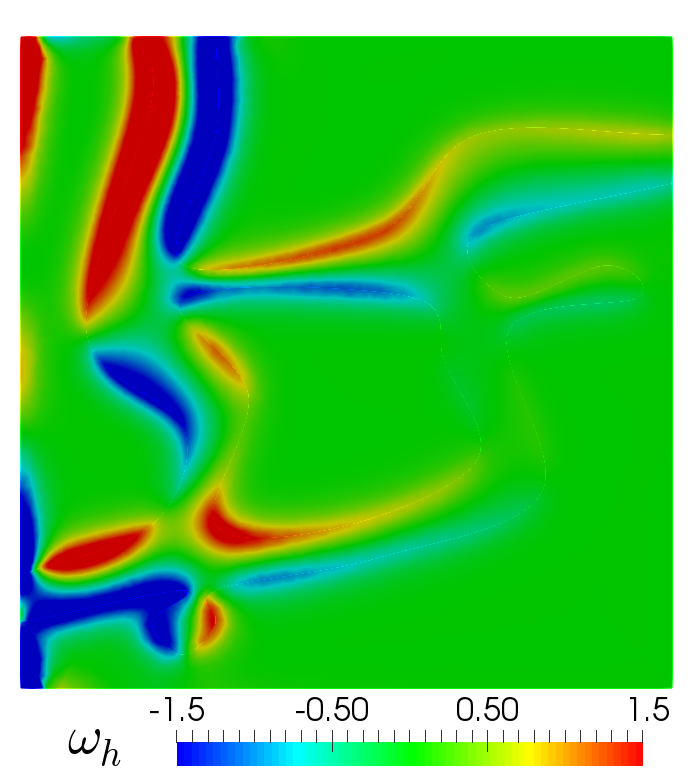}
\includegraphics[width=4.1cm]{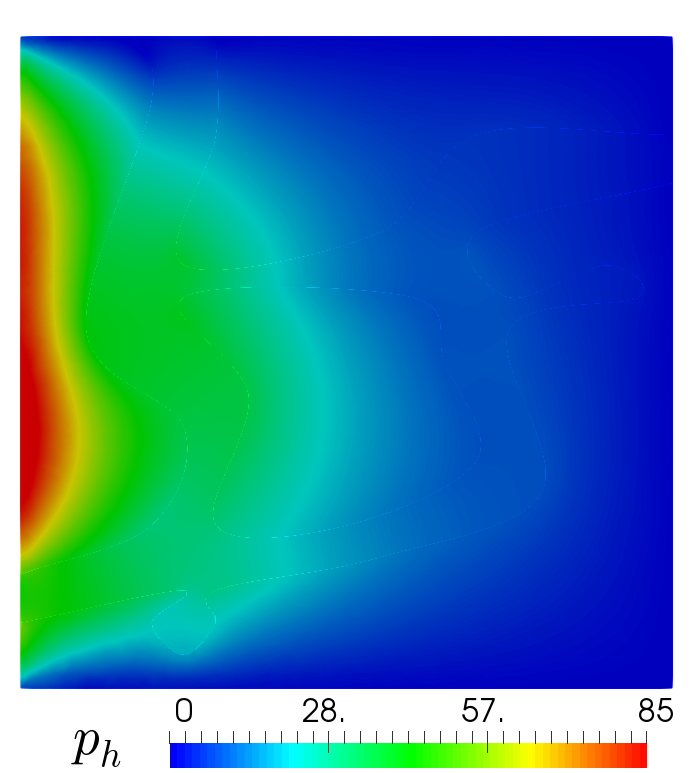}
		
\caption{[Example 3] Domain configuration and prescribed mesh (plots in first column), and computed magnitude of the velocity, vorticity and pressure field at time $T=0.01$ (top plots), at time $T=0.2$ (middle plots), and at time $T=1$ (bottom plots).}\label{fig:example3-uh-omh-ph-T0-T20-Tend}
\end{flushright}
\end{figure}

\begin{figure}[ht!]
\begin{center}		
\includegraphics[width=3.3cm]{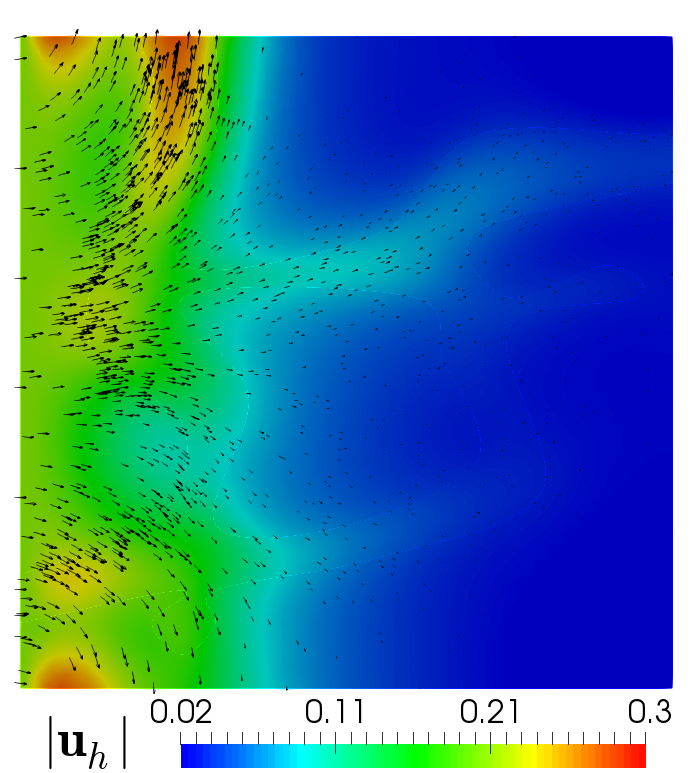}	
\includegraphics[width=3.3cm]{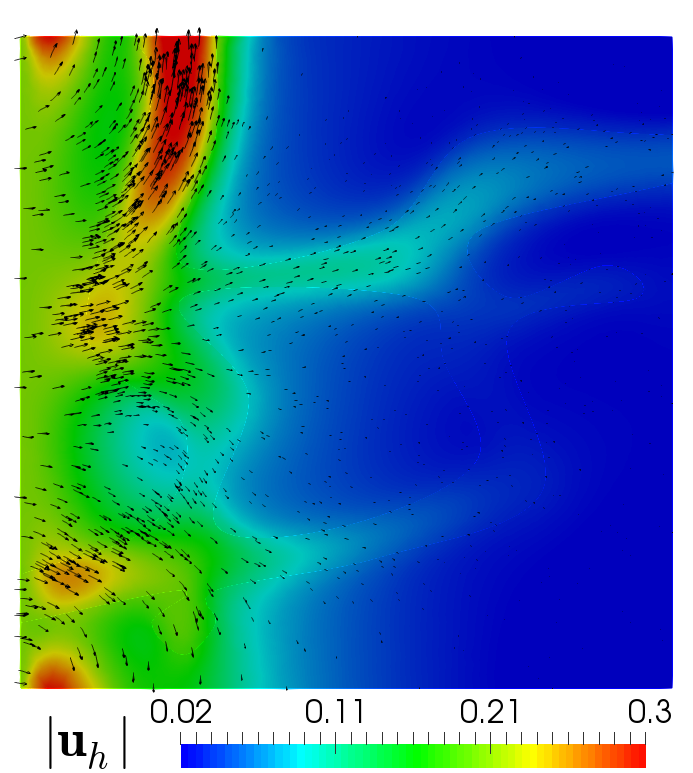}	
\includegraphics[width=3.3cm]{images/ex3-uh-magnitude-k1.png}
\includegraphics[width=3.3cm]{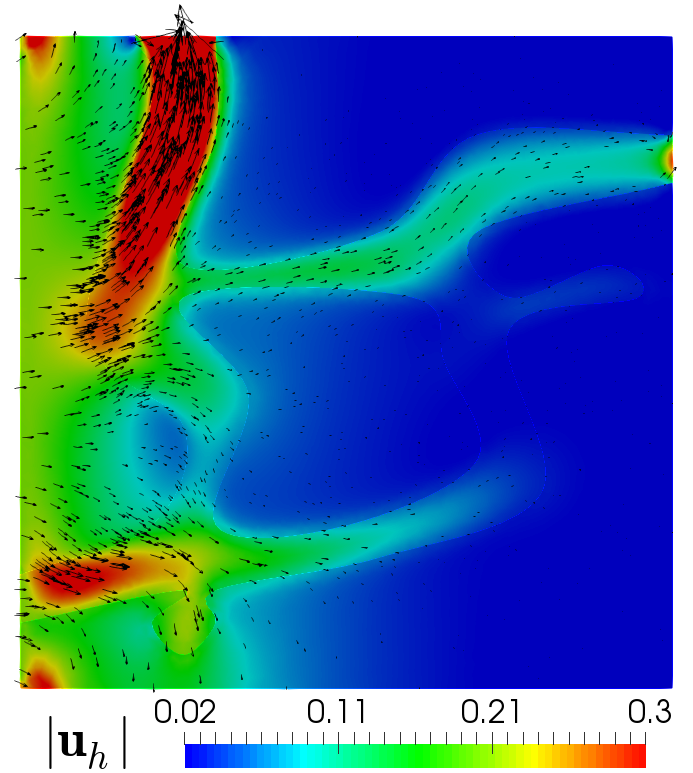}
\includegraphics[width=3.3cm]{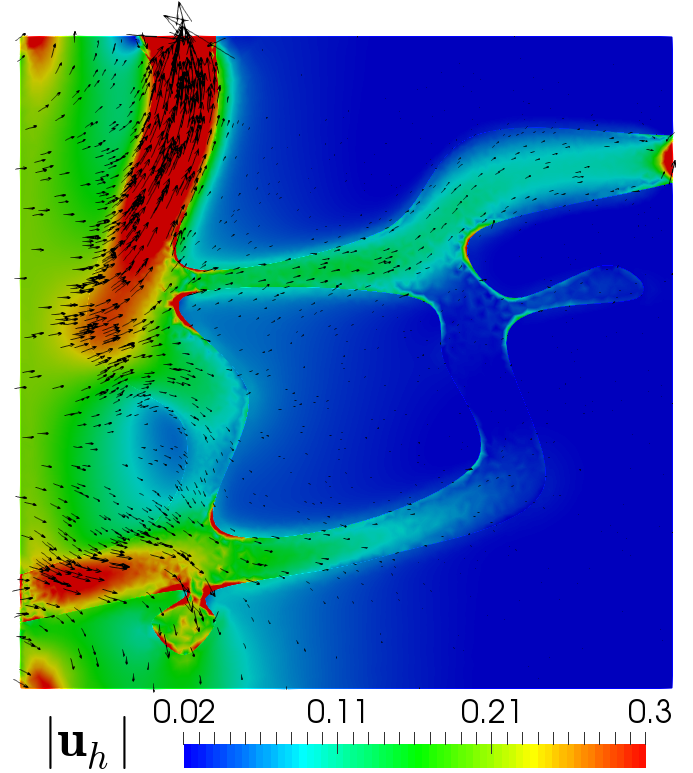}

\includegraphics[width=3.3cm]{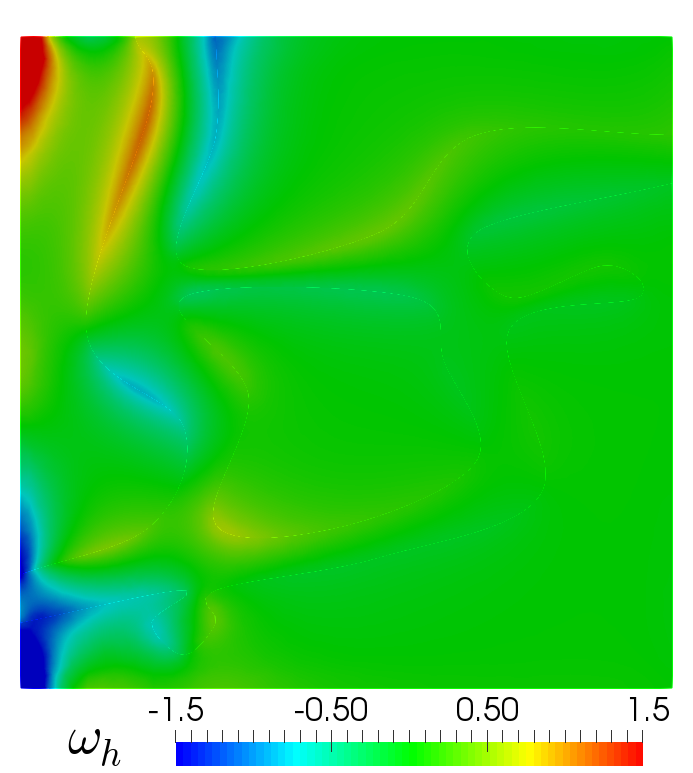}
\includegraphics[width=3.3cm]{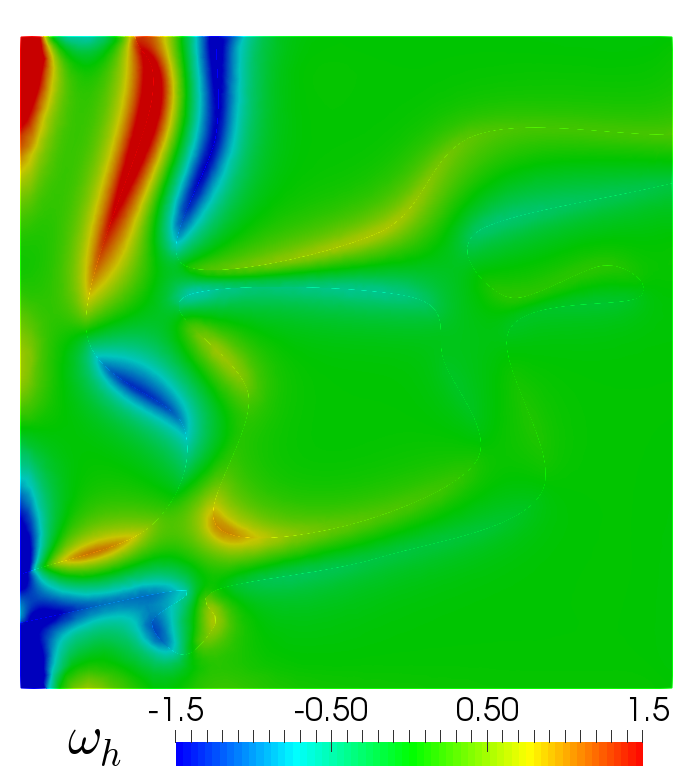}
\includegraphics[width=3.3cm]{images/ex3-omegah-k1.png}
\includegraphics[width=3.3cm]{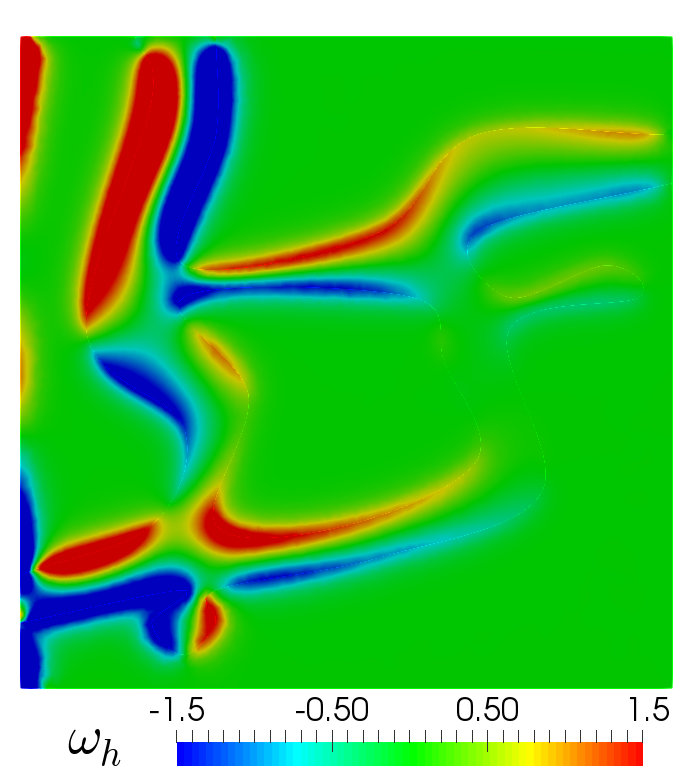}
\includegraphics[width=3.3cm]{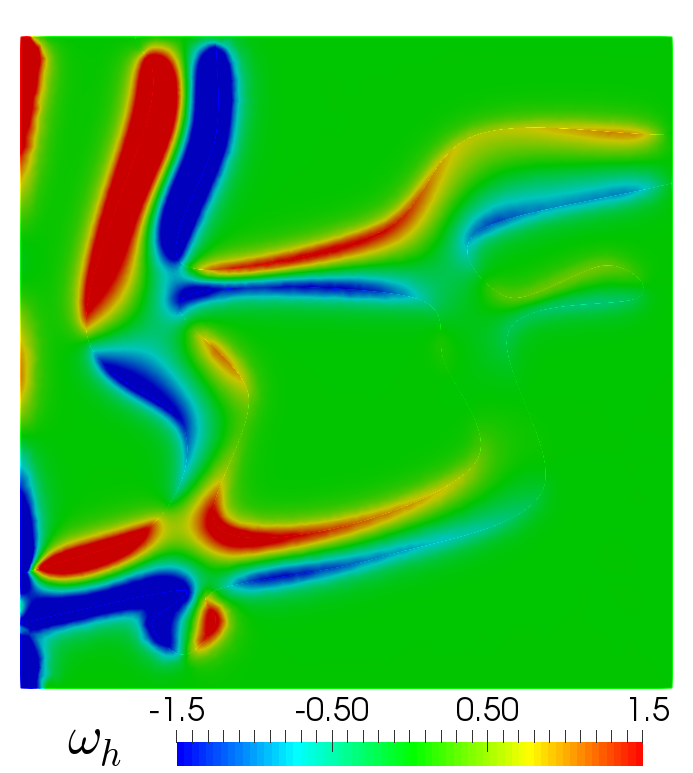}

\includegraphics[width=3.3cm]{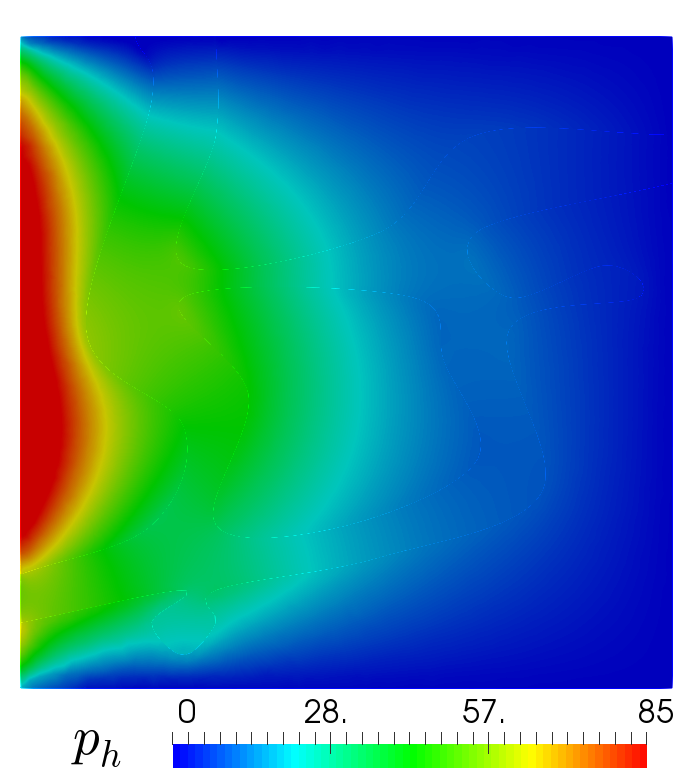}
\includegraphics[width=3.3cm]{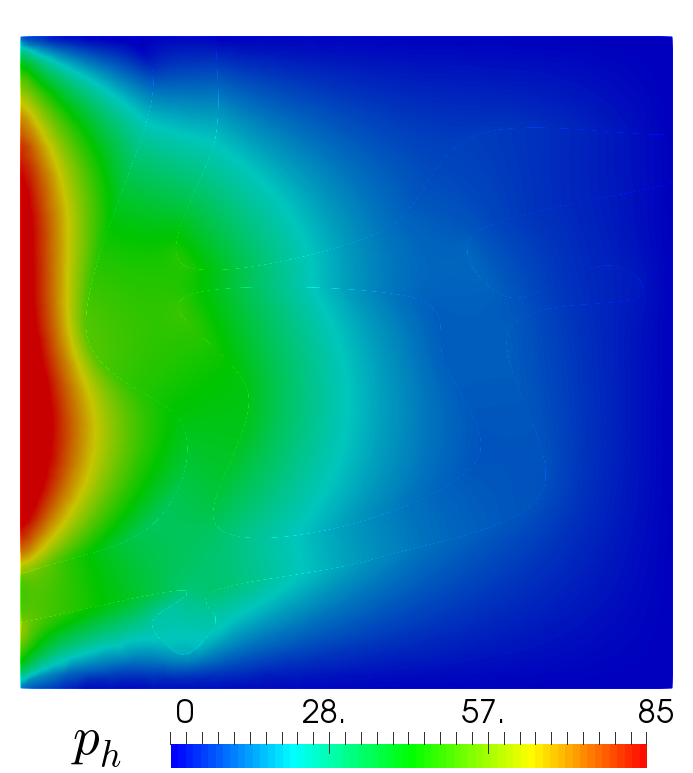}
\includegraphics[width=3.3cm]{images/ex3-ph-k1.png}
\includegraphics[width=3.3cm]{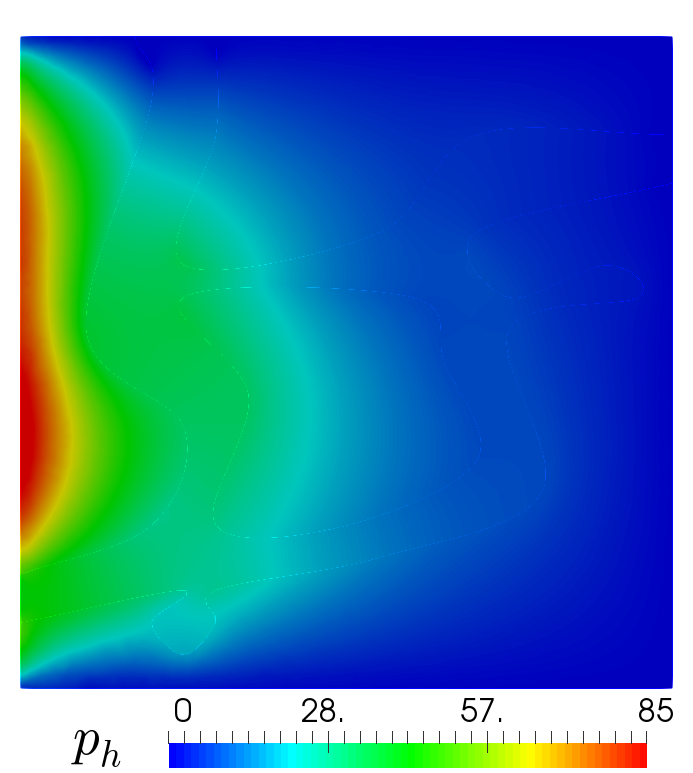}
\includegraphics[width=3.3cm]{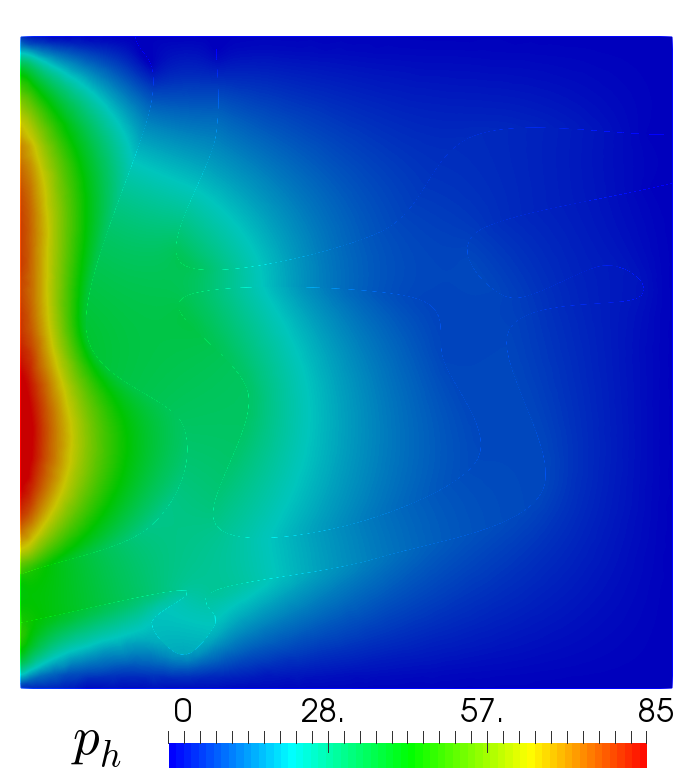}
		
\caption{[Example 3] Computed magnitude of the velocity, vorticity and pressure field at time $T=1$, with $\rho=3$, channel setting $\tF=10$ and $\tD=1$, and porous media setting $\tF=1$ and $\tD=1000$, for $\kappa\in \{ 3, 2, 1, 0.1, 0.01 \}$ (from left to right).}\label{fig:example3-uh-omh-ph}
\end{center}
\end{figure}

\section{Conclusions}\label{sec:conclusions}
In this paper we presented a new velocity-vorticity-pressure formulation for the Kelvin--Voigt--Brink-man--Forchheimer equations and its mixed finite element approximation. The system models fast unsteady viscoelastic flows in highly porous media. The formulation has several advantages, including an accurate and smooth approximation of the vorticity, well posedness for large data, and optimal convergence rates without a mesh quasi-uniformity assumption. Well-posedness of the weak formulation, as well as stability and error analysis for the semidiscrete and fully discrete mixed finite element approximations are presented. The numerical results illustrate that the method is robust for a wide range of parameters, the ability of the system to model heterogeneous media exhibiting both Stokes and Darcy flow regimes, as well as the dissipative effect of the elasticity parameter.

\bibliographystyle{abbrv}
\bibliography{caucao-yotov-2}

\end{document}